\documentclass[11pt, a4paper, oneside]{amsart}

\usepackage[english]{babel}
\usepackage{amsmath, amsthm, amsfonts, mathrsfs, amssymb}
\usepackage{mathtools}
\mathtoolsset{centercolon}
\usepackage{booktabs}
\usepackage[shortlabels]{enumitem}
\setlist[itemize]{leftmargin=20pt}
\usepackage[colorlinks, citecolor = blue]{hyperref}
\usepackage[hmargin=3cm,vmargin=3cm]{geometry}
\usepackage[color=green!40]{todonotes}

\usepackage{color}
\usepackage{graphicx}
\usepackage{tikz-cd}
\usetikzlibrary{arrows}

\newcommand{\N}{\ensuremath{\mathbf{N}}}
\newcommand{\Z}{\ensuremath{\mathbf{Z}}}

\newcommand{\R}{\ensuremath{\mathbf{R}}}

\newcommand{\D}{\ensuremath{\mathcal{D}}}

\DeclarePairedDelimiter\abs{\lvert}{\rvert}

\DeclarePairedDelimiter{\nrm}\lVert\rVert

\newcommand{\cbraceb}[1]{\bigl\{#1\bigr\}}

\newcommand{\nrms}[1]{\Bigl\|#1\Bigr\|}

\newcommand{\has}[1]{\Bigl(#1\Bigr)}
\newcommand{\cbraces}[1]{\Bigl\{#1\Bigr\}}

\DeclareMathOperator{\loc}{loc}
\DeclareMathOperator{\supp}{supp}
\DeclareMathOperator{\ind}{\mathbf{1}}
\DeclareMathOperator{\UMD}{UMD}

\DeclareMathOperator*{\esssup}{ess\,sup}
\DeclareMathOperator*{\essinf}{ess\,inf}

\DeclareMathOperator{\BHT}{BHT}

\DeclareMathOperator{\pv}{p.v.}

\DeclareMathOperator{\BMO}{BMO}

\renewcommand{\emptyset}{\varnothing}
\def\avint_#1{\mathchoice{\mathop{\kern 0.2em\vrule width 0.6em height 0.69678ex depth -0.58065ex \kern -0.8em \intop}\nolimits_{\kern -0.4em#1}}{\mathop{\kern 0.1em\vrule width 0.5em height 0.69678ex depth -0.60387ex \kern -0.6em \intop}\nolimits_{#1}} {\mathop{\kern 0.1em\vrule width 0.5em height 0.69678ex depth -0.60387ex \kern -0.6em \intop}\nolimits_{#1}} {\mathop{\kern 0.1em\vrule width 0.5em height 0.69678ex depth -0.60387ex \kern -0.6em \intop}\nolimits_{#1}}}

\newtheorem{TheoremLetter}{Theorem}
{}
\newtheorem{theorem}{Theorem}
\newtheorem{corollary}[theorem]{Corollary}
\newtheorem{lemma}[theorem]{Lemma}
\newtheorem{proposition}[theorem]{Proposition}

\newtheorem{conjecture}[theorem]{Conjecture}

\theoremstyle{remark}
\newtheorem{remark}[theorem]{Remark}
\newtheorem{example}[theorem]{Example}

\theoremstyle{definition}
\newtheorem{definition}[theorem]{Definition}

\numberwithin{theorem}{section}
\numberwithin{equation}{section}
\title[Extrapolation in Banach function spaces]{Extrapolation in general quasi-Banach function spaces}
\author{Zoe Nieraeth}
\thanks{The author is supported by the Basque Government through the BERC 2022-2025 program and by the Spanish State Research Agency through BCAM Severo
Ochoa excellence accreditation SEV-2017-0718 and through project PID2020-113156GB-I00/AEI/10.13039/501100011033 funded by Agencia Estatal de
Investigación and acronym ``HAPDE''}

\address[Zoe Nieraeth]{BCAM\textendash  Basque Center for Applied Mathematics, Bilbao, Spain}
\email{zoe.nieraeth@gmail.com}
\allowdisplaybreaks

\begin{document}
\begin{abstract}
In this work we prove off-diagonal, limited range, multilinear, vector-valued, and two-weight versions of the Rubio de Francia extrapolation theorem in general quasi-Banach function spaces. We prove mapping properties of the generalization of the Hardy-Littlewood maximal operator to very general bases that includes a method to obtain self-improvement results that are sharp with respect to its operator norm. Furthermore, we prove bounds for the Hardy-Littlewood maximal operator in weighted Lorentz, variable Lebesgue, and Morrey spaces, and recover and extend several extrapolation theorems in the literature. Finally, we provide an application of our results to the Riesz potential and the Bilinear Hilbert transform.
\end{abstract}




\maketitle

\section{Introduction}
There is a rich history of Rubio de Francia extrapolation in the theory of $A_p$ weights. In essence, this theory states that if for a fixed $p\in[1,\infty]$ and all weights $w$ satisfying the $A_p$ condition
\[
[w]_p:=\sup_Q|Q|^{-1}\|w\ind_Q\|_{L^p(\R^d)}\|w^{-1}\ind_Q\|_{L^{p'}(\R^d)}<\infty,
\]
an operator is bounded on the space $L^p_w(\R^d)$, consisting of the measurable functions $f$ such that $fw\in L^p(\R^d)$, then it is automatically also bounded on a whole range of spaces, including $L^q_w(\R^d)$ for all $q\in (1,\infty)$. The original argument dates back to \cite{Ru82, GR85} and many variants have appeared in the literature since.

The $A_p$ condition is deeply related to the Hardy-Littlewood maximal operator
\[
Mf(x):=\sup_Q\left(\frac{1}{|Q|}\int_Q\!|f|\,\mathrm{d}x\right)\ind_Q(x),
\]
and, as a matter of fact, can be characterized by the boundedness of $M$ on $L^p_w(\R^d)$, or its associate space $L^{p'}_{w^{-1}}(\R^d)$.

Such a condition also appears in spaces other than $L^p_w(\R^d)$. As a matter of fact, given a Banach function space $X$ on which $M$ is bounded, we have
\[
\sup_Q|Q|^{-1}\|\ind_Q\|_{X}\|\ind_Q\|_{X'}\leq\|M\|_{X\to X},
\]
where $X'$ is the K\"othe dual of $X$. If one now sets $X=L^p_w(\R^d)$, we recover the $A_p$ condition for the weight $w$.

\bigskip

The main goal of this work is to replace the $A_p$ condition on a weight $w$ and the constant $[w]_p$ by the intrinsic condition $M:X\to X$ of a Banach function space $X$ and the constant $\|M\|_{X\to X}$, and to study under which conditions one can still obtain extrapolation theorems. Moreover, we do not rely on any techniques involving rearrangement invariance.

Our results are in a limited range setting, i.e., rather than considering bounds for $p\in[1,\infty]$, we consider operators that have bounds in a range $p\in[r,s]$, where $0<r<s\leq\infty$. In this case, the appropriate weight classes $A_{p,(r,s)}$ are defined as
\[
[w]_{p,(r,s)}:=\sup_Q |Q|^{-\big(\frac{1}{r}-\frac{1}{s}\big)}\|w\ind_Q\|_{L^{\frac{1}{\frac{1}{p}-\frac{1}{s}}}(\R^d)}\|w^{-1}\ind_Q\|_{L^{\frac{1}{\frac{1}{r}-\frac{1}{p}}}(\R^d)}<\infty.
\]
To understand the geometry of the situation, note that the interval $[\frac{1}{s},\frac{1}{r}]$ is mapped to $[0,1]$ through the affine transformation $\phi$ given by
\[
\phi(t):=\frac{t-\frac{1}{s}}{\frac{1}{r}-\frac{1}{s}}.
\]
For $p\in[r,s]$ we then set
\[
\frac{1}{p_{r,s}}:=\phi\Big(\frac{1}{p}\Big)=\frac{\frac{1}{p}-\frac{1}{s}}{\frac{1}{r}-\frac{1}{s}},\quad \frac{1}{p_{r,s}'}=1-\phi\Big(\frac{1}{p}\Big)=\frac{\frac{1}{r}-\frac{1}{p}}{\frac{1}{r}-\frac{1}{s}},
\]
so that $p_{r,s}\in[1,\infty]$. For $X=L^p_w(\R^d)$, we then write
\[
X_{r,s}:=L^{p_{r,s}}_{w^{\frac{1}{\frac{1}{r}-\frac{1}{s}}}}(\R^d)
\]
and note that $w\in A_{p,(r,s)}$ precisely when we have the bound
\[
M:X_{r,s}\to X_{r,s}
\]
or
\[
M:(X_{r,s})'\to (X_{r,s})'.
\]
As it turns out, this rescaling of a general quasi-Banach function space $X$ to an appropriate Banach function space $X_{r,s}$ is possible precisely when $X$ is $r$-convex and $s$-concave, in which case we have
\[
X_{r,s}=\Big[[(X^r)']^{\big(\frac{s}{r}\big)'}\Big]',
\]
where the space $X^p$ is defined so that $|f|^p\in X^p$ whenever $f\in X$. One of our main tools in extrapolating to $r$-convex and $s$-concave quasi-Banach function spaces $X$ is through the fact that, as a consequence of Lozanovskii's duality theorem, any $f\in X$ can be written as
\[
|f|=h^{\frac{1}{r}-\frac{1}{s}} k^{\frac{1}{s}}
\]
with $0\leq h\in X_{r,s}$, $0\leq k\in L^1(\R^d)$.

Not only does our result allow us to obtain extrapolation results in the limited range setting, we can also obtain extrapolation results in the off-diagonal setting. Indeed, suppose we have exponents $p_1\in[1,\infty]$, $p_2\in(0,\infty]$ and $\alpha\in\R$ such that
\[
\alpha=\frac{1}{p_2}-\frac{1}{p_1},
\]
and an operator $T$ such that for each weight satisfying
\[
[w]_{p_1,p_2}:=\sup_Q|Q|^{-(1+\alpha)}\|w\ind_Q\|_{L^{p_2}(\R^d)}\|w^{-1}\ind_Q\|_{L^{p_1'}(\R^d)}<\infty
\]
we have the bound
\[
T:L^{p_1}_w(\R^d)\to L^{p_2}_w(\R^d),
\]
then it is natural to ask for which spaces $X$, $Y$ we also have the bound
\begin{equation}\label{eq:fullrangeoffdiagintro1}
T:X\to Y.
\end{equation}
As it turns out, this can be viewed as a problem in a limited range setting. Indeed, if we define $s_1$, $r_2$ such that
\[
\frac{1}{s_1}=-\alpha,\quad \frac{1}{r_2}=1+\alpha,
\]
then we are in the geometric situation where
\[
\alpha+\Big[\frac{1}{s_1},1\Big]=\Big[0,\frac{1}{r_2}\Big]
\]
and we have
\[
[w]_{p_1,p_2}=[w]_{p_2,(r_2,\infty)}=[w]_{p_1,(1,s_1)}.
\]
Furthermore, when $X=L_w^{p_1}(\R^d)$, $Y=L^{p_2}_w(\R^d)$ we have the relation
\[
X_{1,s_1}=Y_{r_2,\infty},
\]
and the $A_{p_1,p_2}$ condition is equivalent to the bounds
\begin{equation}\label{eq:fullrangeoffdiagintro2}
M:Y^{r_2}\to Y^{r_2},\quad M:(X')^{s_1'}\to (X')^{s_1'}.
\end{equation}
Hence, as long as we require the spaces $X$, $Y$ to satisfy the relation $X_{1,s_1}=Y_{r_2,\infty}$, or equivalently
\[
(X')^{s_1'}=(Y^{r_2})',
\]
and the bounds \eqref{eq:fullrangeoffdiagintro2} hold, then one can expect the bound \eqref{eq:fullrangeoffdiagintro1} to hold.

In the limited range situation we have parameters $r_1,r_2\in(0,\infty)$, $s_1\in(0,\infty]$, $\frac{1}{s_2}\in\R$ satisfying $\frac{1}{r_1}>\frac{1}{s_1}$, $\frac{1}{r_2}>\frac{1}{s_2}$
and
\[
\frac{1}{r_2}-\frac{1}{r_1}=\frac{1}{s_2}-\frac{1}{s_1}=\alpha.
\]
Then for $p_1,p_2\in(0,\infty]$ satisfying $p_1\in[r_1,s_1]$, $\frac{1}{p_2}\in[\frac{1}{s_2},\frac{1}{r_2}]$ and $\frac{1}{p_2}-\frac{1}{p_1}=\alpha$ we can define the off-diagonal weight class $A_{\vec{p},(\vec{r},\vec{s})}$ through
\[
[w]_{\vec{p},(\vec{r},\vec{s})}=\sup_Q |Q|^{-\big(\frac{1}{r_1}-\frac{1}{s_1}\big)}\|w\ind_Q\|_{L^{\frac{1}{\frac{1}{p_2}-\frac{1}{s_2}}}(\R^d)}\|w^{-1}\ind_Q\|_{L^{\frac{1}{\frac{1}{r_1}-\frac{1}{p_1}}}(\R^d)}<\infty.
\]
In this case we have
\[
A_{\vec{p},(\vec{r},\vec{s})}=A_{p_1,(r_1,s_1)}=A_{p_2,(r_2,s_2)},
\]
and we can proceed as before.

\bigskip

We now state our main theorem. The relevant definitions and related notions for Lorentz, variable Lebesgue, and Morrey spaces can be found in Section~\ref{sec:weightedbfsmbound}.
\begin{TheoremLetter}\label{thm:A}
Let $\alpha\in\R$ and suppose $r_1,r_2\in(0,\infty)$, $s_1\in(0,\infty]$, $\frac{1}{s_2}\in\R$ satisfy $\frac{1}{r_1}>\frac{1}{s_1}$, $\frac{1}{r_2}>\frac{1}{s_2}$
and
\[
\frac{1}{r_2}-\frac{1}{r_1}=\frac{1}{s_2}-\frac{1}{s_1}=\alpha.
\]
Let $p_1,p_2\in(0,\infty]$ with $\frac{1}{p_1}\in[\frac{1}{s_1},\frac{1}{r_1}]$, $\frac{1}{p_2}\in[\frac{1}{s_2},\frac{1}{r_2}]$ and $\frac{1}{p_2}-\frac{1}{p_1}=\alpha$. Let
\[
T:\bigcup_{w\in A_{\vec{p},(\vec{r},\vec{s})}}L^{p_1}_w(\R^d)\to L^0(\R^d)
\]
be a map for which there is an increasing function $\phi:[0,\infty)\to[0,\infty)$ such for all weights $w\in A_{\vec{p},(\vec{r},\vec{s})}$ we have
\[
\|Tf\|_{L^{p_2}_w(\R^d)}\leq\phi([w]_{\vec{p},(\vec{r},\vec{s})})\|f\|_{L^{p_1}_w(\R^d)}
\]
for all $f\in L^{p_1}_w(\R^d)$.

Let $X$, $Y$ be a pair of quasi-Banach function spaces over $\R^d$ that are respectively $r_1$- and $r_2$-convex and $s_1$- and $s_2$-concave, that satisfy $X_{r_1,s_1}=Y_{r_2,s_2}$, and both of the following conditions hold:
\begin{itemize}
\item If $p_1\neq r_1$,
\[
M:X_{r_1,s_1}\to X_{r_1,s_1}
\]
is bounded;
\item If $p_1\neq s_1$,
\[
M:(X_{r_1,s_1})'\to (X_{r_1,s_1})'
\]
is bounded.
\end{itemize}
Then $Tf$ is well-defined for all $f\in X$ and bounded $T:X\to Y$ with
\[
\|T\|_{X\to Y}\leq 2^{\frac{1}{r_1}-\frac{1}{s_1}}\phi\big(2^{\frac{1}{r_1}-\frac{1}{s_1}}\|M\|_{X_{r_1,s_1}\to X_{r_1,s_1}}^{\frac{1}{r_1}-\frac{1}{p_1}}\|M\|_{(X_{r_1,s_1})'\to (X_{r_1,s_1})'}^{\frac{1}{p_2}-\frac{1}{s_2}}\big).
\]

Thus, if $p_1\in(r_1,s_1)$, we have the following results:
\begin{itemize}
\item Lorentz spaces: For all $p_1,q_1\in(r_1,s_1)$, $p_2$ and $q_2$ satisfying
\[
\frac{1}{p_2}-\frac{1}{p_1}=\frac{1}{q_2}-\frac{1}{q_1}=\alpha,
\]
and weights $v\in A_{p_1,(r_1,s_1)}$, we have
\[
T:L^{p_1,q_1}_v(\R^d)\to \big[L^{p_2,q_2}_{v^{\frac{p_1}{p_2}}}(\R^d)\big](v^{1-\frac{p_1}{p_2}}),
\]
where
\[
\|f\|_{\big[L^{p_2,q_2}_{v^{\frac{p_1}{p_2}}}(\R^d)\big](v^{1-\frac{p_1}{p_2}})}=\|fv^{1-\frac{p_1}{p_2}}\|_{L^{p_2,q_2}_{v^{\frac{p_1}{p_2}}}(\R^d)}.
\]
\item Variable Lebesgue spaces: For all $p_1:\R^d\to(r_1,s_1)$ satisfying
\[
r_1<\essinf p_1\leq\esssup p_1< s_1
\]
and $p_1(\cdot)\in LH_0\cap LH_\infty$, $p_2\in L^0(\R^d)$ satisfying
\[
\frac{1}{p_2(x)}-\frac{1}{p_1(x)}=\alpha,
\]
and all weights $v$ satisfying
\[
\sup_Q|Q|^{-\big(\frac{1}{r_1}-\frac{1}{s_1}\big)}\|v\ind_Q\|_{L^{\frac{1}{\frac{1}{p_2(\cdot)}-\frac{1}{s_2}}}(\R^d)}\|v^{-1}\ind_Q\|_{L^{\frac{1}{\frac{1}{r_1}-\frac{1}{p_1(\cdot)}}}(\R^d)}<\infty,
\]
we have
\[
T:L^{p_1(\cdot)}_v(\R^d)\to L^{p_2(\cdot)}_v(\R^d).
\]
\item Morrey spaces: If $\alpha=0$, $s_1=s_2=\infty$, $r=r_1=r_2$, then for all $p\in(r,\infty)$, $q\in[p,\infty)$ and weights $v$ satisfying
\[
\sup_Q|Q|^{-\big(\frac{1}{r}-\frac{1}{p}+\frac{1}{q}\big)}\|v\ind_Q\|_{L^q(\R^d)}\|v^{-1}\ind_Q\|_{L^{\frac{1}{\frac{1}{r}-\frac{1}{p}}}(\R^d)}<\infty,
\]
we have
\[
T:\mathcal{L}_v^{p,q}(\R^d)\to\mathcal{L}_v^{p,q}(\R^d).
\]
\end{itemize}
\end{TheoremLetter}

The proof of this result can be found in Section~\ref{sec:mainthm}. Here we do not only consider weights associated to cubes, but rather to any basis of sets in a $\sigma$-finite measure space $(\Omega,|\cdot|)$. Moreover, we also consider a two-weight version of this result. 

We note that the reason we do not obtain bounds for Morrey spaces for finite $s_1$ or $s_2$ is because Morrey spaces are not concave in general. We refer the reader to Section~\ref{subsec:morreyspaces} for an elaboration on Morrey spaces. Moreover, we refer the reader to Section~\ref{sec:sparsedom} for a discussion on how the $s$-concavity condition can be removed for operators satisfying sparse domination, and, in particular, to Theorem~\ref{thm:sparseimpliesmorrey} for bounds in Morrey spaces for finite $s$.

\bigskip

Several extrapolation results in various Banach function spaces more general than Lebesgue spaces have already appeared throughout the literature.

Extending the $A_\infty$ and $A_1$ extrapolation result of \cite{CMP04}, there are several extrapolation results in general and so-called rearrangement invariant Banach function spaces \cite{CGMP06, CMP11}. In these spaces one relies on the theory of Boyd indices to obtain the boundedness of the maximal operator $M$. These results are presented not only on the usual bases of cubes or balls, but on Muckenhoupt bases, which are general bases on which the weight classes are characterized by the boundedness of the maximal operator of that basis on weighted Lebesgue spaces. Our result extends these results in several ways, and we recover general versions of their results in general Banach function spaces, without having to rely on any rearrangement invariance. Moreover, our results hold in very general bases that do not have to Muckenhoupt bases. In particular, we get the full range of expected boundedness by extrapolating from a single Lebesgue exponent.

A variant of an $A_1$ off-diagonal extrapolation theorem in Banach function spaces over spaces of homogeneous type with the basis of balls, which is a Muckenhoupt basis, is considered in \cite[Theorem~2]{KM20}. Here they first use the sharp $A_{p_0,q_0}$ off-diagonal Lebesgue space extrapolation to extrapolate from the weighted pair of Lebesgue spaces at an exponent pair $p_0,q_0\in(1,\infty)$ to an exponent pair $\widetilde{p}_0,\widetilde{q}_0\in(1,\infty)$ with $\alpha=\frac{1}{\widetilde{p}_0}-\frac{1}{\widetilde{q}_0}=\frac{1}{p_0}-\frac{1}{q_0}$, and then use an $A_1$ off-diagonal Banach function space extrapolation theorem at this new exponent pair. Our result recovers the latter argument as the endpoint case $r_1=p_1=\widetilde{p}_0$, $r_2=p_2=\widetilde{q}_0$, $s_2=\infty$ and extends their result to the setting of quasi-Banach function spaces.

Another $A_1$ off-diagonal extrapolation theorem in so-called ball quasi-Banach function spaces is considered in \cite{DSL23}. In principal, the quasi-Banach function spaces we consider in this work are more general than ball quasi-Banach function spaces. However, as we show in Lemma~\ref{lem:indE}, under the condition that the maximal operator is bounded on one of our quasi-Banach function spaces, we precisely recover the class of ball quasi-Banach function spaces in the case where we restrict ourselves to the setting of $\R^d$ with the basis of balls. In this setting we completely recover their main result \cite[Theorem~3.1]{DSL23} as the special endpoint case $\alpha=\frac{1}{q_0}-\frac{1}{p_0}$, $r_1=p_1=p_0$, $r_2=p_2=q_0$, $\frac{1}{s_2}=\frac{1}{p_2}(1-\delta_0)$ in Theorem~\ref{thm:A}.

Various extrapolation results in general weighted Banach function spaces in the on-diagonal case were considered in \cite{CMM22}. The spaces in our work are more general, and, in particular, our limited range results extend theirs, not only in the range of bounds we obtain, but also by considering bases more general than Muckenhoupt bases. Furthermore, various extrapolation theorems involving changes of weights in the Banach function spaces are considered in \cite{CO22}. We discuss and compare some of our results in Section~\ref{subsec:weightedbfs}.

Extrapolation theorems in the limited range and an off-diagonal setting in weighted variable Lebesgue spaces was proven in \cite{CW17}. We not only completely recover this result, but extend it to cover the entire range of expected weighted bounds.

Extrapolation theorems in the full and limited range settings in weighted Morrey spaces were considered in \cite{DR18, DR20}. While we obtain results through sparse domination in the off-diagonal and limited range setting, their result is more general in the on-diagonal setting. Indeed, since their proof does not rely directly on the boundedness of the maximal operator, their technique does not only allow them to remove the $s$-concavity condition so that they can consider finite $s$, but also allows them to obtain bounds for the larger class of power weights in the class $A_{q,(r,s)}$. For a comparison, we elaborate on which power weights we get precisely in Section~\ref{subsec:morreyspaces}. The reason we do not fully recover their weighted bounds is because the corresponding weighted norm bounds for the maximal operator on weighted Block spaces, which are the dual spaces of weighted Morrey spaces, are not known. If, however, the conjectured bounds on these spaces were to hold, we would be able to recover their result with our technique when $s=\infty$. We prove new weighted bounds in Block spaces in this work, which do allow us to recover all the weights obtained in \cite{DR18}.

\bigskip

In the multilinear setting, there are various extrapolation results in Lebesgue spaces. It was shown in \cite{GM04} that for an $m$-tuple of exponents $p_1,\ldots,p_m\in(1,\infty)$ with $\frac{1}{p}=\sum_{j=1}^m\frac{1}{p_j}>0$ and an $m$-tuple of weights $w_1,\ldots,w_m$, each satisfying the $A_{p_j}$ condition, one can extrapolate the bound
\begin{equation}\label{eq:multilinearinitialboundintro}
T:L^{p_1}_{w_1}(\R^d)\times\cdots\times L^{p_m}_{w_m}(\R^d)\to L^p_w(\R^d),
\end{equation}
where $w=\prod_{j=1}^m w_j$, to hold for all $p_1,\ldots,p_m\in(1,\infty)$ with $\frac{1}{p}=\sum_{j=1}^m\frac{1}{p_j}>0$.

With an application to the Bilinear Hilbert transform in mind, this was later extended in \cite{CM17} to a limited range setting. For fixed $0<r_j<s_j\leq\infty$, $j\in\{1,\ldots,m\}$, they considered the initial bound \eqref{eq:multilinearinitialboundintro} for $w_j\in A_{p_j,(r_j,s_j)}$, and could extrapolate this bound to all $p_j\in(r_j,s_j)$.

It was found in \cite{LOPTT09} that the appropriate weight class to consider for the bound \eqref{eq:multilinearinitialboundintro} in the case of multilinear Calder\'on-Zygmund operators is the class $A_{\vec{p}}$ which consists of $m$-tuples of weights $\vec{w}=(w_1,\ldots,w_m)$ for which
\[
[\vec{w}]_{\vec{p}}:=\sup_Q\Big(\langle\prod_{j=1}^m w_j^{-1}\rangle_{p_j',Q}\Big)\langle w\rangle_{p,Q}<\infty,
\]
where $w=\prod_{j=1}^m w_j$, $\frac{1}{p}=\sum_{j=1}^m\frac{1}{p_j}$. Extrapolation for limited range generalizations of these weight classes was later proven in \cite{LMO20} and extended in \cite{Ni19, LMMOV21}. In particular, in \cite{Ni19} a multilinear Rubio de Francia algorithm was developed that allows for sharp extrapolation to weighted Lebesgue spaces. Unfortunately, this technique only seems to work specifically in Lebesgue spaces, and it is not clear how to extend this to more general Banach function spaces.

Nonetheless, under the same initial assumption on our operator as in \cite{CM17}, we are able to prove the following result:
\begin{TheoremLetter}\label{thm:B}
Let $r_1,\ldots,r_m\in(0,\infty)$, $s_1,\ldots, s_m\in(0,\infty]$ with $r_j<s_j$ and $p_j\in[r_j,s_j]$ for $j\in\{1,\ldots,m\}$ with $\frac{1}{p}:=\sum_{j=1}^m\frac{1}{p_j}$. Let
\[
T:\bigcup_{w_1\in A_{p_1,(r_1,s_1)}}L^{p_1}_{w_1}(\R^d)\times\cdots\times \bigcup_{w_m\in A_{p_m,(r_m,s_m)}}L^{p_m}_{w_m}(\R^d)\to L^0(\R^d)
\]
is a map for which there is an increasing function $\phi:[0,\infty)^m\to[0,\infty)$ such that for all $w_j\in A_{p_j,(r_j,s_j)}$ and $f_j\in L^{p_j}_{w_j}(\R^d)$, $j\in\{1,\ldots,m\}$, we have
\[
\|T(f_1,\ldots,f_m)\|_{L^p_w(\R^d)}\leq\phi([w_1]_{p_1,(r_1,s_1)},\ldots,[w_m]_{p_m,(r_m,s_m)})\prod_{j=1}^m\|f_j\|_{L^{p_j}_{w_j}(\R^d)}.
\]
Let $X_1,\ldots, X_m$ be quasi-Banach function spaces over $\R^d$ for which $X_j$ is $r_j$-convex and $s_j$-concave and for all $j\in\{1,\ldots,m\}$ we have:
\begin{itemize}
\item  If $p_j\neq r_j$,
\[
M:(X_j)_{r_j,s_j}\to (X_j)_{r_j,s_j}
\]
is bounded.
\item If $p_j\neq s_j$,
\[
M:\big[(X_j)_{r_j,s_j}\big]'\to\big[(X_j)_{r_j,s_j}\big]'
\]
is bounded.
\end{itemize}
Then $T(f_1,\ldots,f_m)$ is well-defined for all $f_j\in X_j$ and $T:X_1\times\cdots X_m\to X$ is bounded with
\[
\|T\|_{X_1\times\cdots\times X_m\to X}\leq2^{\frac{1}{r}-\frac{1}{s}}\phi(C_1,\ldots,C_m),
\]
where $\frac{1}{r}=\sum_{j=1}^m\frac{1}{r_j}$, $\frac{1}{s}=\sum_{j=1}^m\frac{1}{s_j}$, and
\[
C_j=2^{\frac{1}{r_j}-\frac{1}{s_j}}\|M\|_{(X_j)_{r_j,s_j}\to (X_j)_{r_j,s_j}}^{\frac{1}{r_j}-\frac{1}{p_j}}\|M\|_{\big[(X_j)_{r_j,s_j}\big]'\to\big[(X_j)_{r_j,s_j}\big]'}^{\frac{1}{p_j}-\frac{1}{s_j}}.
\]
\end{TheoremLetter}
The proof of this is rather straightforward. One uses a version of Lozanovskii's duality theorem to obtain the factorization
\[
\big[(X^r)'\big]^{\frac{1}{r}}=\prod_{j=1}^m\big[(X_j^{r_j})'\big]^{\frac{1}{r_j}}.
\]
From here, we have reduced the problem back to the case $m=1$.

This result is proven in the more general Theorem~\ref{thm:multilinearextrapolation} below.

\bigskip

One can also obtain vector-valued bounds from extrapolation for an operator in $\ell^p(\N)$ spaces through a trick involving Fubini's theorem. However, there is a rich theory for obtaining bounds in the Bochner spaces $L^p_w(\R^d;Y)$ in a Banach space $Y$. We discuss these ideas further in Section~\ref{subsection:fubini}.

In \cite[Theorem~5]{Ru85} Rubio de Francia showed that one can obtain bounds in these spaces as long as $Y$ is a Banach function space over some $\sigma$-finite measure space $\Omega$ satisfying the so-called $\UMD$ property. His proof relied on the idea that an order-continuous Banach function space $Y$ has the $\UMD$ property, if and only if the Hardy-Littlewood maximal operator $M$ has a bounded vector-valued extension $\widetilde{M}$ to $L^p(\R^d;Y)$ and $L^{p'}(\R^d;Y')$ for all $p\in(1,\infty)$ \cite{Bo84, Ru86}. Here, a function $f\in L^p(\R^d;Y)$ can be viewed as a function of two variables $(x,y)\in\R^d\times\Omega$, and
\[
\widetilde{M}f(x,y):=M(f(\cdot,y))(x).
\]
The property that $\widetilde{M}$ is bounded on $L^p(\R^d;Y)$ is usually referred to as saying that $Y$ has the Hardy-Littlewood property.

This result was extended to the limited range setting $0<r<s\leq\infty$, where the assumption on the space becomes that $Y$ must be a $r$-convex and $s$-concave quasi-Banach function space, and $Y_{r,s}$ has the $\UMD$ property. This was first done for $r\neq 1$, $s=\infty$ in \cite{ALV17}, and later for general $0<r<s\leq\infty$ in \cite{LN19}. The weaker notion of $Y\in\UMD_{r,s}$, which is equivalent to assuming that $Y$ is order-continuous and $Y^r$ and $(Y_{r,s})'$ have the so-called Hardy-Littlewood property, was introduced in \cite{LN22} and it was shown in \cite{Nidiss} that the vector-valued extrapolation theorem still holds for such spaces.

In this work we take these results beyond the setting of Lebesgue spaces. For a quasi-Banach function space $X$ over $\R^d$, we define the mixed-norm space $X(Y)$ of functions $f\in L^0(\R^d\times\Omega)$ for which $f(x,\cdot)\in Y$ for all $x\in\R^d$, and
\[
\big\|\|f\|_Y\big\|_X<\infty.
\]
As a simple consequence of combining the result of \cite{LN19, Nidiss} and Theorem~\ref{thm:B}, we obtain the following:

\begin{TheoremLetter}\label{thm:C}
Let $r_1,\ldots,r_m\in(0,\infty)$, $s_1,\ldots,s_m\in(0,\infty]$ with $r_j<s_j$. Suppose $T$ is an $m$-(sub)linear operator such that for all $p_j\in(r_j,s_j)$ it is well-defined on all $m$-tuples of functions $(f_1,\ldots,f_m)$ with $f_j\in L^{p_j}_{w_j}(\R^d)$ with $w_j\in A_{p_j,(r_j,s_j)}$, $j\in\{1,\ldots,m\}$ and, moreover, there is an increasing function $\phi_{\vec{p}}:[0,\infty)^m\to[0,\infty)$ such that for all $w_j\in A_{p_j,(r_j,s_j)}$ and $f_j\in L^{p_j}_{w_j}(\R^d)$, $j\in\{1,\ldots,m\}$ we have
\[
\|T(f_1,\ldots,f_m)\|_{L^p_w(\R^d)}\leq\phi_{\vec{p}}([w_1]_{p_1,(r_1,s_1)},\ldots,[w_m]_{p_m,(r_m,s_m)})\prod_{j=1}^m\|f_j\|_{L^{p_j}_{w_j}(\R^d)}.
\]
Let $Y_1,\ldots,Y_m$ be an $m$-tuple of quasi-Banach function spaces over a $\sigma$-finite measure space $(\Omega,|\cdot|)$ such that $Y_j$ is order-continuous, $r_j$-convex and $s_j$-concave for all $j\in\{1,\ldots,m\}$ such that for all simple functions $f_j\in L^\infty_c(\R^d;Y_j)$ the function
\[
\widetilde{T}(f_1,\ldots,f_m)(x,y):=T(f_1(\cdot,y),\cdots, f_m(\cdot,y))(x)
\]
is strongly measurable in $Y=\prod_{j=1}^m Y$. If $Y_j^{r_j}$ and $((Y_j)_{r_j,s_j})'$ have the Hardy-Littlewood property for all $j\in\{1,\ldots,m\}$, then there is an increasing function $\phi_{\vec{Y},\vec{r},\vec{s}}:[0,\infty)^{2m}\to [0,\infty)$ such that for all $m$-tuples $X_1,\ldots, X_m$ of quasi-Banach function spaces over $\R^d$ for which $X_j$ is $r_j$-convex, $s_j$-concave, and for all $j\in\{1,\ldots,m\}$ we have that
\[
M:(X_j)_{r_j,s_j}\to (X_j)_{r_j,s_j},\quad M:\big[(X_j)_{r_j,s_j}\big]'\to\big[(X_j)_{r_j,s_j}\big]'
\]
are bounded, we have that
\[
\widetilde{T}:X_1(Y_1)\times\ldots\times X_m(Y_m)\to X(Y)
\]
is bounded, where $X=\prod_{j=1}^m X_j$, with
\begin{align*}
\|\widetilde{T}&\|_{X_1(Y_1)\times\ldots\times X_m(Y_m)\to X(Y)}\\
&\leq\phi_{\vec{Y},
\vec{r},\vec{s}}\big((\|M\|_{(X_j)_{r_j,s_j}\to (X_j)_{r_j,s_j}},\|M\|_{\big[(X_j)_{r_j,s_j}\big]'\to \big[(X_j)_{r_j,s_j}\big]'})_{j=1}^m\big).
\end{align*}
\end{TheoremLetter}
This result is proven as Theorem~\ref{thm:multivvextrapolation} below.

\bigskip

This work does not only include results in $\R^d$, but for very general bases of sets in a $\sigma$-finite measure space $(\Omega,|\cdot|)$, and we obtain several additional new results. In Theorem~\ref{thm:maximalselfimprove} we obtain a new sharp self-improvement property for maximal operators. Many results in the literature rely on this, and our results provides a new way to sharply track the involved constants, which we give an example of in Theorem~\ref{thm:mboundnos}.

Another new result we obtain is bounds for the maximal operator on weighted block spaces in Theorem~\ref{thm:mboundblockspaces}, which allows us to extrapolate to weighted Morrey spaces.

Finally, we mention that in Section~\ref{sec:mainthm} we obtain several two-weight extrapolation theorems that have not previously appeared in the literature.
\bigskip

The remainder of this work is organized as follows:
\begin{itemize}
  \item In Section~\ref{sec:prelim} we give an overview of the results on quasi-Banach function spaces, classes of weights over general bases of sets, and mapping properties of the associated maximal operator on Banach function spaces.
  \item In Section~\ref{sec:weightedbfsmbound} we provide examples of weighted quasi-Banach function spaces on which the maximal operator is bounded. More precisely, we consider weighted Lorentz spaces, weighted variable Lebesgue spaces, and weighted Morrey spaces.
  \item In Section~\ref{sec:proofs} we first prove an abstract extrapolation theorem in a dual form, which serves as the foundation for all our extrapolation theorems. We also discuss the utility of extrapolation pairs, and obtain extrapolation theorems in the on-diagonal case. Moreover, we use this to deduce various further results including vector-valued, weak-type, $A_1$, and $A_\infty$ extrapolation results. We also compare our results to known extrapolation theorems in the literature.
  \item In Section~\ref{sec:mainthm} we prove our main result in the off-diagonal and two-weight setting, and provide applications.
  \item In Section~\ref{sec:sparsedom} we compare the bounds in Banach function spaces that can be obtained from sparse domination with the bounds we obtain in our extrapolation theorem. Moreover, we prove our multilinear extrapolation theorems and provide an application to the Bilinear Hilbert transform.
\end{itemize}

\subsection*{Notation}
We sometimes write $A\lesssim_{a,b,\ldots}B$ to mean that there is a constant $C_{a,b,\ldots}$, depending only on the parameters $a,b,\ldots$, such that $A\leq C_{a,b,\ldots}B$. We use an analogous convention for $\gtrsim_{a,b,\ldots}$ and, moreover, we write $A\eqsim_{a,b,\ldots} B$ when $A\lesssim_{a,b,\ldots} B$ and $A\gtrsim_{a,b,\ldots} B$.

For a parameter $p\in[1,\infty]$, we define $p'\in[1,\infty]$ through $\frac{1}{p}+\frac{1}{p'}=1$.

\section{Preliminaries}\label{sec:prelim}
\subsection{Quasi-Banach function spaces}\label{sec:BFS}
Let $(\Omega,|\cdot|)$ be a $\sigma$-finite measure space of positive measure and let $L^0(\Omega)$ denote the space of a.e. finite measurable functions on $\Omega$. A vector space $X \subseteq L^0(\Omega)$ equipped with a quasi-norm $\nrm{\,\cdot\,}_X$ is called a \emph{quasi-Banach function space} over $\Omega$ if it satisfies the following properties:
\begin{itemize}
  \item \textit{Ideal property:} If $f\in X$ and $g\in L^0(\Omega)$ with $|g|\leq|f|$, then $g\in X$ with $\nrm{g}_X\leq \nrm{f}_X$.
  \item \textit{Fatou property:} If $0\leq f_n \uparrow f$ for $(f_n)_{n\in\N}$ in $X$ and $\sup_{n\in \N}\nrm{f_n}_X<\infty$, then $f \in X$ and $\nrm{f}_X=\sup_{n\in\N}\nrm{f_n}_X$.
  \item \textit{Saturation:} For every measurable $E\subseteq\Omega$ of positive measure, there exists a measurable $F\subseteq E$ of positive measure with $\ind_F\in X$.
\end{itemize}
If $\nrm{\,\cdot\,}_X$ is a norm then $X$ is called a \emph{Banach function space} over $\Omega$. 

Note that it follows from the ideal property that $f\in X$ if and only if $|f|\in X$ with $\||f|\|_X=\|f\|_X$.

Since $\|\cdot\|_{X}$ is a quasi-norm, that means there is a constant $K\geq 1$ such that
\[
\|f+g\|_X\leq K(\|f\|_X+\|g\|_X).
\]
for all $f,g\in X$. We denote the optimal constant in this inequality by $K_X$. In this case, we have the following result:
\begin{proposition}\label{prop:quasisum}
Let $(f_n)_{n=0}^\infty$ be a sequence in $X$ for which
\[
\sum_{n=1}^\infty \|f_n\|_X<\infty.
\]
Then we have
\[
\sum_{n=1}^\infty K_X^{-n}f_n\in X
\]
with
\[
\Big\|\sum_{n=1}^\infty K_X^{-n}f_n\Big\|_X\leq\sum_{n=1}^\infty\|f_n\|_X.
\]
\end{proposition}
\begin{proof}
Let $N\in\N$. Then, by induction, we have
\[
\Big\|\sum_{n=1}^N K_X^{-n}f_n\Big\|_X\leq\sum_{n=1}^N \|f_n\|_X.
\]
Hence, by the Fatou property we have $\sum_{n=1}^\infty K_X^{-n}f_n\in X$ with
\[
\Big\|\sum_{n=1}^\infty K_X^{-n}f_n\Big\|_X=\sup_{N\in\N}\Big\|\sum_{n=1}^N K_X^{-n}f_n\Big\|_X\leq\sum_{n=1}^\infty \|f_n\|_X,
\]
as desired.
\end{proof}

We define the \emph{K\"othe dual} $X'$ of a quasi-Banach function space $X$ by
\[
X':=\{g\in L^0(\Omega):fg\in L^1(\Omega)\text{ for all $f\in X$}\}.
\]
Moreover, we define a seminorm on $X'$ through
\[
\|g\|_{X'}:=\sup_{\substack{f\in X\\ \|f\|_X=1}}\|fg\|_{L^1(\Omega)}.
\]
This is finite for every $f\in X'$ by the following result:
\begin{proposition}
Let $g\in L^0(\Omega)$ with $fg\in L^1(\Omega)$ for all $f\in X$. Then $\|g\|_{X'}<\infty$.
\end{proposition}
\begin{proof}
Suppose that $\|g\|_{X'}=\infty$. Then for each $n\in\N$ there is an $f_n\in X$ with $\|f_n\|_X=1$ for which
\[
\|f_n g\|_{L^1(\Omega)}>K_X^n n^3.
\]
By Proposition~\ref{prop:quasisum} we have $F:=\sum_{n=1}^\infty K_X^{-n}\frac{|f_n|}{n^2}\in X$. However, since also
\[
\|Fg\|_{L^1(\Omega)}\geq K_X^{-n}\frac{\|f_n g\|_{L^1(\Omega)}}{n^2}>n
\]
for all $n\in\N$, we conclude that $Fg\notin L^1(\Omega)$. The result follows by contraposition.
\end{proof}

As it turns out, $\|\cdot\|_{X'}$ is a norm and not just a seminorm. This follows exactly from the saturation property. Indeed, we have the following result:
\begin{proposition}\label{prop:weakorderunit}
Let $X$ be a quasi-Banach function space over $\Omega$. Then the saturation property is equivalent to any of the following statements:
\begin{enumerate}[(i)]
\item\label{it:propsat1} The seminorm $\|\cdot\|_{X'}$ on $X'$ is a norm;
\item\label{it:propsat2} The space $X$ has a \emph{weak order unit}, i.e., there is a $\rho\in X$ with $\rho>0$ a.e. in $\Omega$.
\end{enumerate}
\end{proposition}
\begin{proof}
We first show that the saturation property implies \ref{it:propsat2}. Assume $X$ is saturated. Since $\Omega$ is $\sigma$-finite, it follows from \cite[Ch.\,15, §\,67, Theorem~3]{Za67} that there exists an increasing sequence of measurable sets $F_n\subseteq\Omega$ with $\ind_{F_n}\in X$ and $\bigcup_{n=1}^\infty F_n=\Omega$. We note here that this theorem can be applied this way in a quasi-Banach function space, since the only property that is required is that if $\ind_{F_1},\ind{F_2}\in X$, then also $\ind_{F_1}+\ind_{F_2}\in X$.

Now define
\[
\rho:=\sum_{n=1}^\infty K_X^{-n}2^{-n}\frac{\ind_{F_n}}{1+\|\ind_{F_n}\|_X}.
\]
By Proposition~\ref{prop:quasisum} we have $\rho\in X$, and the result follows from the fact that $\rho>0$.

For \ref{it:propsat2}$\Rightarrow$\ref{it:propsat1}, note that if $\|g\|_{X'}=0$, this means that $\|fg\|_{L^1(\R^d)}=0$ for all $f\in X$. In particular, we have $\|\rho g\|_{L^1(\Omega)}=0$. This means that $\rho g=0$ and hence, since $\rho>0$, we must have $g=0$ a.e. in $\Omega$, as desired.

It remains to show that \ref{it:propsat1} implies the saturation property. Assume that $X$ does not have the saturation property. Then there is a set $E\subseteq\Omega$ of positive measure such that $\ind_F\notin X$ for all $F\subseteq E$ of positive measure. For $f\in X$, consider the sets $F_n:=\{x\in E:|f(x)|>\frac{1}{n}\}$. Since 
\[
\ind_{F_n}\leq nf\in X,
\]
it follows from the ideal property of $X$ that $\ind_{F_n}\in X$ for all $n\in\N$. But this means that $|F_n|=0$ and hence, 
\[
|\{x\in E:|f(x)|>0\}|=\Big|\bigcup_{n=1}^\infty F_n\Big|\leq \sum_{n=1}^\infty|F_n|=0.
\]
Since this means that every function $f\in X$ vanishes on $E$, we have $\ind_E\in X'$ with $\|\ind_E\|_{X'}=0$. Since $E$ has positive measure, we conclude that $\|\cdot\|_{X'}$ is not a norm. By contraposition, this proves the result.
\end{proof}

As it turns out, if $X$ is a Banach function space over $\Omega$, then $X'$ is again a Banach function space over $\Omega$. Indeed, by the Lorentz--Luxemburg Theorem, the Fatou property is equivalent to the statement that $X''=X$ with equal norm, which by \ref{it:propsat1} implies that $X'$ also has the saturation property. We emphasize that the norm on $X$ being equal to the norm on $X''$ means that
\[
\|f\|_X=\sup_{\substack{g\in X'\\ \|g\|_{X'}=1}}\|fg\|_{L^1(\Omega)}.
\]
However, in the case that $X$ is a quasi-Banach function space, $X'$ is not necessarily a Banach function space. While $\|\cdot\|_{X'}$ is a norm and $X'$ has the ideal and Fatou property, it does not necessarily have the saturation property. Indeed, as an example we note that if $X=L^p(\R^d)$ for $p\in(0,1)$, then $X'=\{0\}$. As we will see in Section~\ref{subsec:maxinqbfs}, if the Hardy-Littlewood maximal operator is bounded on $X$, then $X'$ is saturated and, hence, a Banach function space.

In general we have the following result:
\begin{proposition}\label{prop:envelopeembedding}
Let $X$ be a quasi-Banach function space over $\Omega$. Then $X\subseteq X''$ with
\[
\|f\|_{X''}\leq\|f\|_X.
\]
Moreover, we have $X'''=X'$ with
\[
\|g\|_{X'''}=\|g\|_{X'}.
\]
\end{proposition}
\begin{proof}
Let $f\in X$. Then, per definition of $X'$, for all $g\in X'$ we have $fg\in L^1(\Omega)$ with $\|fg\|_{L^1(\Omega)}\leq\|f\|_X\|g\|_{X'}$. Hence, we have $f\in X''$ with
\[
\|f\|_{X''}=\sup_{\|g\|_{X'}=1}\|fg\|_{L^1(\Omega)}\leq\|f\|_X,
\]
as asserted.

For the second assertion, note that by applying the first result to $X$ replaced by $X'$ we have $X'\subseteq X'''$ with $\|\cdot\|_{X'''}\leq\|\cdot\|_{X'}$. Conversely, by again applying the first result, we have
\[
\|g\|_{X'}=\sup_{\|f\|_X\leq 1}\|fg\|_{L^1(\Omega)}\leq\sup_{\|f\|_{X''}\leq 1}\|fg\|_{L^1(\Omega)}=\|g\|_{X'''}.
\]
The assertion follows.
\end{proof}

A quasi-Banach function space $X$ is called \emph{order-continuous} if for any sequence $0\leq f_n\downarrow 0\in X$ we have $\nrm{f_n}_X \to 0$.  When a Banach function space is order-continuous, then its dual space coincides isometrically with its K\"othe dual through the canonical embedding $X'\hookrightarrow X^\ast$. We also point out that a Banach function space $X$ is reflexive if and only if $X$ and $X'$ are order-continuous.

\begin{remark}
Proofs of the statements about Banach function spaces in this section and a general introduction into Banach function spaces can be found in \cite[Chapter~15]{Za67}. Note that in that work, a Banach function space $X$ is defined through a function norm $\rho:L^0(\Omega)_+\to[0,\infty]$, where $f\in X$ if and only if $\|f\|_X:=\rho(|f|)<\infty$. This is equivalent to the definition of a Banach space we are using here. To see this, we note that given a Banach function space $X$ following our definition, we can define $\rho:L^0(\Omega)_+\to[0,\infty]$ through $\rho(f):=\|f\|_X$ if $f\in X$ and $\rho(f)=\infty$ if $f\notin X$. Then $\rho$ is a function norm according to the definition in \cite{Za67} satisfying the Fatou and saturation property. Hence, the results from that work apply in our setting. We also refer the reader to the recent survey \cite{LN23b} on quasi-Banach function spaces.
\end{remark}

\begin{remark}
Very often in the literature, the following additional property is imposed on Banach function spaces:
\begin{itemize}
\item For each $E\subseteq\Omega$ with $|E|<\infty$ we have $\ind_E\in X$ and there exists a constant $C_E>0$ such that
\[
\int_E\!|f|\,\mathrm{d}x\leq C_E\|f\|_X
\]
for all $f\in X$.
\end{itemize}
This is equivalent to the statement that for all $E$ of finite measure we have $\ind_E\in X$ and $\ind_E\in X'$ (with $\|\ind_E\|_{X'}=C_E$). We note here that the  condition $\ind_E\in X$ implies the saturation property.

This property is not necessary for any of our results, and, as a matter of fact, they are restrictive in which spaces we can consider. An example of this is the Morrey space $X=\mathcal{L}^{p,q}(\R^d)$, which does not satisfy the property that $\ind_E\in X'$, see \cite[Example~3.3]{ST15}. We refer the reader to Section~\ref{subsec:morreyspaces} for the results we obtain in these spaces.
\end{remark}

\begin{definition}
Let $(\Omega,|\cdot|)$ be a $\sigma$-finite measure space. Let $p,q\in(0,\infty]$ and let $X$ be a quasi-Banach function space over $\Omega$. We say that $X$ is \emph{$p$-convex}, if for for all $f,g\in X$ we have
\[
  \nrms{\has{\abs{f}^p+\abs{g}^p}^{\frac{1}{p}}}_X \leq \has{ \nrm{f}_X^p+\nrm{g}_X^p}^{\frac{1}{p}}.
\]
and \emph{$q$-concave} if for all $f,g\in X$ we have
\[
\has{ \nrm{f}_X^q+\nrm{g}_X^q}^{\frac{1}{q}}\leq \nrms{\has{\abs{f}^q+\abs{g}^q}^{\frac{1}{q}}}_X,
\]
with the usual modification when $p=\infty$ or $q=\infty$.
\end{definition}
Any quasi-Banach function space is $\infty$-concave by the ideal property and $X$ is a Banach function space exactly when it is $1$-convex. As an example, we note that $L^r(\Omega)$ for $r\in(0,\infty]$ is $p$-convex if $p\leq r$ and $q$-concave if $q\geq r$.

We also note that if $X$ is $p$-convex or $q$-concave, then it is also respectively $\widetilde{p}$-convex or $\widetilde{q}$-concave whenever $\widetilde{p}\leq p$ or $\widetilde{q}\geq q$.

Usually for Banach function spaces a constant is allowed in the defining inequalities for $p$-convexity and $q$-concavity, but it is shown in \cite[Theorem 1.d.8]{LT79} that one can choose an equivalent norm on $X$ such that these constants are equal to $1$.

For $p\in (0,\infty)$ we define the \emph{$p$-concavification} of $X$ by
\begin{equation*}
  X^p:=\cbraceb{f \in L^0(\Omega):\abs{f}^{\frac{1}{p}} \in X}.
\end{equation*}
This is again a quasi-Banach function space when equipped with the quasinorm
\[
\nrm{f}_{X^p}:= \nrm{\abs{f}^{\frac{1}{p}}}_X^p.
\]
We note that $\nrm{\,\cdot\,}_{X^p}$ is a norm, and hence, $X^p$ is a Banach function space, if and only if $X$ is $p$-convex.  Moreover, if $X$ is a Banach function space and $p,q\in[1,\infty]$, then $X$ is respectively $p$-convex or $q$-concave if and only if $X'$ is $p'$-concave or $q'$-convex.

We note that this allows us to reasonably extend the definition of $q$-concave for a Banach function space $X$ to exponents $q$ with $\frac{1}{q}<0$ to mean that $X'$ is $\frac{1}{1-\frac{1}{q}}$-convex, in which case we formally write $q':=\frac{1}{1-\frac{1}{q}}$. We do point out however that this is not an additional assumption on the space. Indeed, in this case $q'<1$ and hence, since $X'$ is $1$-convex, it is automatically also $q'$-convex. Thus, the extension of the notion of concavity for negative $\frac{1}{q}$ is merely a formality that is useful only in simplifying our notation.

\begin{definition}
Let $\frac{1}{r}\in(0,\infty]$, $\frac{1}{s}\in\R$ with $\frac{1}{r}>\frac{1}{s}$ and let $X$ be a quasi-Banach function space. Then we define the \emph{$(r,s)$-rescaled Banach function space of $X$} as
\[
X_{r,s}:=\Big[\big[(X^r)'\big]^{\left(\frac{s}{r}\right)'}\Big]'.
\]
\end{definition}

If $X$ is an $r$-convex and $s$-concave quasi-Banach function space, then $X_{r,s}$ is a Banach function space. Indeed, note that the $r$-convexity of $X$ implies that $X^r$ is a Banach function space, and hence, so is the K\"othe dual $(X^r)'$. Moreover, since $X^r$ is $\frac{s}{r}$-concave, this means that $(X^r)'$ is $\left(\frac{s}{r}\right)'$-convex, and hence, $\big[(X^r)'\big]^{\left(\frac{s}{r}\right)'}$ is a Banach function space. This implies that its K\"othe dual $X_{r,s}$ is also Banach function space.

If $X$ is only $r$-convex, then $\big[(X^r)'\big]^{\left(\frac{s}{r}\right)'}$ is a quasi-Banach function space, and its K\"othe dual $X_{r,s}$ might not have the saturation property. In the cases where $X$ is not $s$-concave, we will usually impose the extra condition that $X_{r,s}$ is saturated to ensure that it is a Banach function space. When $s=\infty$, we note that now $X_{r,\infty}=(X^r)''=X^r$ by the Lorentz-Luxemburg theorem. In particular, we have $X_{1,\infty}=X$ when $X$ is a Banach function space.

\begin{definition}
For two quasi-Banach function spaces $X_1$, $X_2$ over the same $\sigma$-finite measure space $(\Omega,|\cdot|)$, we define their product space by
\[
X_1\cdot X_2:=\cbraces{f\in L^0(\Omega):\text{there exist } 0\leq f_1\in X_1,\,0\leq f_2\in X_2\text{ such that }\abs{f} \leq f_1f_2}.
\]
We equip this space with the quasi-norm
\[
\|f\|_{X_1\cdot X_2}:=\inf \cbraces{\|f_1\|_{X_1}\|f_2\|_{X_2}: 0\leq f_1 \in X_1,\,0\leq f_2\in X_2\text{ for which }\abs{f} \leq f_1f_2}.
\]
\end{definition}

For a quasi-Banach function space $X$ we have $X\cdot X'\subseteq L^1(\Omega)$. Indeed, if $h\in X\cdot X'$, then there are $0\leq f\in X$, $0\leq g\in X'$ with $|h|\leq fg\in L^1(\Omega)$. Hence, $h\in L^1(\Omega)$ by the ideal property, with
\[
\|h\|_{L^1(\Omega)}\leq\|fg\|_{L^1(\Omega)}\leq\|f\|_X\|g\|_{X'}.
\]
Taking an infimum over all possible $0\leq f\in X$, $0\leq g\in X'$ with $|h|\leq fg$ yields
\begin{equation}\label{eq:lozfacquasi}
\|h\|_{L^1(\Omega)}\leq\|h\|_{X\cdot X'},
\end{equation}
as desired.

As a matter of fact, if $X$ is a Banach function space, then $X\cdot X'=L^1(\Omega)$ with equal norm by Lozanovskii's factorization theorem. A particular consequence of this is known as Lozanovskii's duality theorem:
\begin{theorem}[{\cite[Theorem 2]{Lo69}}]\label{prop:lozprod}
Let $\theta\in[0,1]$ and let $X_0$ and $X_1$ be Banach function spaces over the same $\sigma$-finite measure space $(\Omega,|\cdot|)$. Then $X_0^{1-\theta}\cdot X_1^\theta$ is again a Banach function space with
\[
(X_0^{1-\theta}\cdot X_1^\theta)'=\big[(X_0)'\big]^{1-\theta}\cdot\big[(X_1)'\big]^\theta.
\]
\end{theorem}

The following result is a special case:
\begin{theorem}[{\cite[Theorem~2.4,\,Theorem~2.9]{Sc10}}]\label{thm:spacesplitting}
Let $(\Omega,|\cdot|)$ be a $\sigma$-finite measure space. Let $q\in[1,\infty)$ and let $Y$ be a $q$-convex Banach function space over $\Omega$. Then
\[
Y'=\big[(Y^q)'\big]^{\frac{1}{q}}\cdot L^{q'}(\Omega)
\]
isometrically. Moreover, for any $f\in Y'$ the infimum in the norm is attained, i.e., there exist $0\leq h\in \big[(Y^q)'\big]^{\frac{1}{q}}$, $0\leq k\in L^{q'}(\Omega)$ with
\[
|f|=hk,\quad \|f\|_{Y'}=\|h\|_{\big[(Y^q)'\big]^{\frac{1}{q}}}\|k\|_{L^{q'}(\Omega)}.
\]
\end{theorem}

By applying this theorem, we obtain the following factorization result:
\begin{corollary}\label{cor:factorization}
Let $r\in(0,\infty)$, $s\in(0,\infty]$ with $r<s$ and let $X$ be an $r$-convex and $s$-concave quasi-Banach function space. Then
\[
X=(X_{r,s})^{\frac{1}{r}-\frac{1}{s}}\cdot L^s(\Omega).
\]
Moreover, the infimum in the norm is attained, i.e., for any $f\in X$ there exist $0\leq h\in (X_{r,s})^{\frac{1}{r}-\frac{1}{s}}$, $0\leq k\in L^s(\Omega)$ with
\[
|f|=hk,\quad \|f\|_{X}=\|h\|_{(X_{r,s})^{\frac{1}{r}-\frac{1}{s}}}\|k\|_{L^s(\Omega)}.
\]
\end{corollary}
\begin{proof}
By Theorem~\ref{thm:spacesplitting} applied with $Y=(X^r)'$ and $q=\big(\frac{s}{r}\big)'$, we have
\[
X^r=X_{r,s}^{1-\frac{r}{s}}\cdot L^{\frac{s}{r}}(\Omega)=\big((X_{r,s})^{\frac{1}{r}-\frac{1}{s}}\cdot L^s(\Omega)\big)^r.
\]
The result follows from taking the $\frac{1}{r}$-concavification of both sides. Moreover, for any $f\in X$ we have $|f|^r\in X^r=Y'$, so there are $0\leq\widetilde{h}\in X_{r,s}^{1-\frac{r}{s}}$ and $0\leq\widetilde{k}\in L^{\frac{s}{r}}(\Omega)$ with
\[
|f|^r=\widetilde{h}\widetilde{k},\quad \||f|^r\|_{X^r}=\|\widetilde{h}\|_{X_{r,s}^{1-\frac{r}{s}}}\|\widetilde{k}\|_{L^{\frac{s}{r}}(\Omega)}.
\]
Setting $h:=\widetilde{h}^{\frac{1}{r}}$, $k:=\widetilde{k}^{\frac{1}{r}}$ then yields the desired factorization.
\end{proof}

In the case that $X$ is not $r$-convex or $s$-concave, we have the following result:
\begin{theorem}\label{thm:forcedfactorization}
Let $r\in(0,\infty)$, $s\in(0,\infty]$ with $r<s$ and let $X$ be quasi-Banach function space for which $(X^r)'$ and $X_{r,s}$ are saturated. Then the quasi-Banach space
\[
Z:=(X_{r,s})^{\frac{1}{r}-\frac{1}{s}}\cdot L^s(\Omega)
\]
is $r$-convex, $s$-concave, and satisfies
\[
Z_{r,s}=X_{r,s}.
\]
\end{theorem}
\begin{proof}
Note that since $(X^r)'$ is saturated, so is $\big[(X^r)'\big]^{(\frac{r}{s})'}$, and hence, the seminorm on $X_{r,s}$ is a norm by Proposition~\ref{prop:weakorderunit}. Since $X_{r,s}$ is also saturated, this implies that $X_{r,s}$ is a Banach function space. Since we have
\[
Z^r=X_{r,s}^{1-\frac{r}{s}}\cdot L^1(\Omega)^{\frac{r}{s}},
\]
and $X_{r,s}$ and $L^1(\Omega)$ are Banach function spaces, it follows from Lozanovskii's duality theorem that $Z^r$ is also a Banach function space, and hence, $Z$ is $r$-convex, with
\[
(Z^r)'=\big[(X_{r,s})'\big]^{1-\frac{r}{s}}\cdot L^\infty(\Omega)^{\frac{r}{s}}=\big[(X_{r,s})'\big]^{1-\frac{r}{s}}.
\]
Hence,
\[
\big[(Z^r)'\big]^{(\frac{s}{r})'}=(X_{r,s})'
\]
is a Banach function space, proving that $Z$ is $s$-concave. Finally, using Proposition~\ref{prop:envelopeembedding} we conclude from taking the K\"othe dual that $Z_{r,s}=X_{r,s}$. This proves the result.
\end{proof}

In the case where $X$ is already a Banach function space and $1\leq r<s\leq\infty$, we note that $X$ is $r$-convex and $s$-concave precisely when $X'$ is $s'$-convex and $r'$-concave, and hence, we can consider the space $(X')_{s',r'}$. In this case we have the following equality:
\begin{proposition}\label{prop:bfsrescaleddualequality}
Let $r\in[1,\infty)$ $s\in(1,\infty]$ with $r<s$, and suppose $X$ is an $r$-convex and $s$-concave Banach function space. Then
\[
(X_{r,s})'=(X')_{s',r'}.
\]
\end{proposition}
\begin{proof}
By Theorem~\ref{thm:spacesplitting} we have
\[
X'=[(X^r)']^{\frac{1}{r}}\cdot L^{r'}(\Omega).
\]
Hence, setting
\[
\theta:=s'\big(\frac{1}{r}-\frac{1}{s}\big)=\frac{1}{\big(\frac{r'}{s'}\big)'}\in[0,1],
\]
we have
\[
(X')^{s'}=\big([(X^r)']^{\big(\frac{s}{r}\big)'}\big)^{\theta}\cdot L^{\frac{1}{1-\theta}}(\Omega),
\]
so that another application of Theorem~\ref{thm:spacesplitting} yields
\[
[(X')^{s'}]'=(X_{r,s})^{\theta}.
\]
The result now follows from taking the $\big(\frac{r'}{s'}\big)'$-concavification of both sides and taking the K\"othe dual.
\end{proof}

\subsection{General bases of sets}
Let $(\Omega,|\cdot|)$ be a $\sigma$-finite measure space. A collection of measurable sets $\mathcal{E}$ with $0<|E|<\infty$ is called a \emph{basis of sets} in $\Omega$ if
\begin{enumerate}[(i)]
\item\label{it:basis1} $\bigcup_{E\in\mathcal{E}}E=\Omega$;
\item\label{it:basis2} For every $x_1,x_2\in\Omega$ there is an $E\in\mathcal{E}$ with $x_1,x_2\in E$.
\item\label{it:basis3} There exists a countable collection $\mathcal{E}'\subseteq\mathcal{E}$ with the property that for all $\varepsilon>0$ and $E\in\mathcal{E}$ there is an $E'\in\mathcal{E'}$ with $E\subseteq E'$ and $|E'|\leq (1+\varepsilon)|E|$.
\end{enumerate}
Property \ref{it:basis3} is automatically satisfied if $\mathcal{E}$ is countable. Moreover, if $\mathcal{E}'$ is as in this property then, by property \ref{it:basis1}, we also have that
\[
\bigcup_{E'\in\mathcal{E}'}E'=\Omega.
\]
Thus, the assumption that $(\Omega,|\cdot|)$ is $\sigma$-finite is necessary for the existence of such a basis.

When $\mathcal{E}$ and $\mathcal{F}$ are two bases of sets in $\Omega$, then we write $\mathcal{E}\lesssim\mathcal{F}$ if there is a constant $C>0$ such that for each $E\in\mathcal{E}$ there is a $F\in\mathcal{F}$ such that $E\subseteq F$ and $|F|\leq C|E|$. Moreover, we write $\mathcal{E}\sim\mathcal{F}$ when $\mathcal{E}\lesssim\mathcal{F}$ and $\mathcal{F}\lesssim\mathcal{E}$. We note that $\sim$ is an equivalence relation.

For $p\in(0,\infty]$ and $f\in L^0(\Omega)$ we write $f\in L^p_{\mathcal{E}}(\Omega)$ when $f\ind_E\in L^p(\Omega)$ for all $E\in\mathcal{E}$. In this case we write
\[
\langle f\rangle_{p,E}:=|E|^{-\frac{1}{p}}\|f\ind_E\|_{L^p(\Omega)}.
\]
Moreover, we define
\[
M_p^{\mathcal{E}}f:=\sup_{E\in\mathcal{E}}\langle f\rangle_{p,E}\ind_E
\]
and set $M^\mathcal{E}f:=M_1^\mathcal{E}f$. Property \ref{it:basis3} ensures that $M_p^\mathcal{E}f$ is a measurable function:

\begin{proposition}\label{prop:mmeasurable}
Let $\mathcal{E}$ and $\mathcal{F}$ be two bases of sets in $\Omega$ with $\mathcal{E}\lesssim \mathcal{F}$ with constant $C$. Then for all $p\in(0,\infty]$ we have $L^p_{\mathcal{F}}(\Omega)\subseteq L^p_{\mathcal{E}}(\Omega)$ and for all $f\in L^p_{\mathcal{F}}(\Omega)$ we have
\[
M_p^{\mathcal{E}}f\leq C^{\frac{1}{p}} M_p^{\mathcal{F}} f.
\]

In particular, if $\mathcal{E}\sim\mathcal{F}$, then $M^{\mathcal{E}}f\eqsim M^{\mathcal{F}}f$ for all $f\in L^1_{\mathcal{E}}(\Omega)=L^1_{\mathcal{F}}(\Omega)$. 

Moreover, if $\mathcal{E}'$ is as in property \ref{it:basis3} then we have
\[
M_p^\mathcal{E}f=M_p^{\mathcal{E}'}f.
\]
In particular, $M_p^\mathcal{E}f$ is measurable.
\end{proposition}
\begin{proof}
Let $f\in L^p_{\mathcal{F}}(\Omega)$. Fix $E\in\mathcal{E}$ and pick $F\in\mathcal{F}$ with $E\subseteq F$ and $|F|\leq C|E|$. Since $f\ind_E\leq f\ind_F\in L^p(\Omega)$, we also have $f\ind_E\in L^p(\Omega)$ and hence, $f\in L^p_{\mathcal{E}}(\Omega)$, proving the inclusion. Since also
\[
|E|^{-\frac{1}{p}}\|f\ind_E\|_{L^p(\Omega)}\leq C^{\frac{1}{p}}|F|^{-\frac{1}{p}}\|f\ind_F\|_{L^p(\Omega)},
\]
this proves the first assertion.

For the second assertion, note that $\mathcal{E}\lesssim\mathcal{E}'$ with constant $C=1+\varepsilon$ for all $\varepsilon>0$. Hence, we have
\[
M^\mathcal{E}_p f\leq(1+\varepsilon)^{\frac{1}{p}}M^{\mathcal{E}'}_p f
\]
Letting $\varepsilon\downarrow 0$, we conclude that $M^\mathcal{E}_p f\leq M^{\mathcal{E}'}_p f$. Since $\mathcal{E}'\subseteq\mathcal{E}$ we also have $M^{\mathcal{E}'}_p f\leq M^\mathcal{E}_p f$, proving that $M^\mathcal{E}_p f=M^{\mathcal{E}'}_p f$. Since the latter is a countable supremum over measurable functions, it is a measurable function, as desired.
\end{proof}

\begin{proposition}\label{prop:nonzeromax}
Let $\mathcal{E}$ be a basis of sets in $\Omega$. If $f\in L^1_{\mathcal{E}}(\Omega)$ is non-zero, then $M^\mathcal{E}f>0$ in $\Omega$.
\end{proposition}

\begin{lemma}\label{lem:basisaltprop}
Let $\mathcal{E}$ be a collection of measurable subsets of $\mathcal{E}$ with $0<|E|<\infty$. Then it satisfies properties \ref{it:basis1} and \ref{it:basis2} if and only if for all $x\in\Omega$ we have
\begin{equation}\label{eq:basisaltprop}
\bigcup_{\substack{E\in\mathcal{E}\\ x\in E}}E=\Omega.
\end{equation}
\end{lemma}
\begin{proof}
First assume \eqref{eq:basisaltprop}. Then \ref{it:basis1} is immediate. For \ref{it:basis2}, let $x_1,x_1\in\Omega$. Since $x_2\in\Omega=\bigcup_{\substack{E\in\mathcal{E}\\ x_1\in E}}E$, there is a set $E\in\mathcal{E}$ so that $x_1,x_2\in E$, proving condition \ref{it:basis2}.

For the converse, assume $\mathcal{E}$ is a basis of sets. Let $x,y\in\Omega$. By \ref{it:basis2} there is an $E\in\mathcal{E}$ such that $x,y\in E$. Hence, we have $y\in \bigcup_{\substack{E\in\mathcal{E}\\ x\in E}}E$, and thus, \eqref{eq:basisaltprop} holds. This proves the result.
\end{proof}

\begin{proof}[Proof of Proposition~\ref{prop:nonzeromax}]
We prove this by contraposition. If $M^\mathcal{E}f(x)=0$ at a point $x\in\Omega$,
then
\[
\langle f\rangle_{1,E}=0
\]
for all $E\in\mathcal{E}$ containing $x$. Hence, $f\equiv 0$ on $\bigcup_{\substack{E\in\mathcal{E}\\ x\in E}}E$. Since this set is equal to $\Omega$ by Lemma~\ref{lem:basisaltprop}, the assertion follows.
\end{proof}

We provide several examples of bases.
\begin{example}[Maximal operators in $\R^d$]
We consider the space $\R^d$ with the Lebesgue measure $|\cdot|$. Then the bases of all balls $\mathcal{B}$ is a basis. Indeed, the countable subcollection $\mathcal{B}'\subseteq\mathcal{B}$ of balls with rational centers and rational radii satisfies property \ref{it:basis3}.

Similarly, the collection of cubes whose sides are parallel to the coordinate axes $\mathcal{Q}$ in $\R^d$ is a basis, where we can take $\mathcal{Q}'\subseteq\mathcal{Q}$ to be the collection of cubes with rational corners and rational side lengths.

In this case we have $\mathcal{B}\sim\mathcal{Q}$. Moreover, the associated maximal operator $M^\mathcal{B}\eqsim M^\mathcal{Q}$ is the Hardy-Littlewood maximal operator, and $L^p_{\mathcal{Q}}(\R^d)=L^p_{\mathcal{B}}(\R^d)=L^p_{\loc}(\R^d)$.

An analogous argument shows that the collection of rectangles whose sides are parallel to the coordinate axes $\mathcal{R}$ in $\R^d$ is also a basis which satisfies $\mathcal{Q},\mathcal{B}\lesssim\mathcal{R}$. Moreover, the associated maximal operator $M^\mathcal{R}$ is the strong maximal operator.
\end{example}

\begin{example}[Dyadic grids]
We again consider $\R^d$ with the Lebesgue measure $|\cdot|$. Then any of the standard dyadic grids \[
\D:=\bigcup_{k\in\Z}\{2^k([0,1)^d+m):m\in\Z^d\},
\]
in $\R^d$ are \emph{not} a basis of sets in $\R^d$, since they fail property \ref{it:basis2}. For example, in the case $d=1$, for $x_1,x_2\in\R$ with $x_1<0$ and $x_2>0$ there is no set in $\mathcal{E}$ that contains both $x_1$ and $x_2$.

To remedy this, we can instead choose a dyadic grid as in \cite{LN15}, which does satisfy property \ref{it:basis2}. Alternatively, if we take the standard dyadic grids as above, then, by the $3^d$-lattice theorem, there are $3^d$ translates $(\D_\alpha)_{\alpha=1}^{3^d}$ of $\D$ such that
\[
\mathcal{Q}\sim\bigcup_{\alpha=1}^{3^d}\D_\alpha,
\]
where $\lesssim$ holds with constant $C=6^d$, see \cite{LN15}, and this is a again a basis of sets.

A similar result holds in general spaces of homogeneous type $(\Omega,d,|\cdot|)$, which is a space $\Omega$ with a quasi-metric $d$ and a measure $|\cdot|$, where the measure $|\cdot|$ is doubling with respect to the balls generated by $d$. Denoting the collection of balls in $\Omega$ by $\mathcal{B}$, it was shown in \cite[Theorem 4.1]{HK12} that there exist a finite number of dyadic grids $(\D_\alpha)_{\alpha=1}^N$ in $\Omega$ with
\[
\mathcal{B}\sim\bigcup_{\alpha=1}^{N}\D_\alpha,
\]
where $N$ and the constants in this equivalence depend on the quasi-metric constant and the doubling constant of $(\Omega,d,|\cdot|)$.

We can also do the following. If we pick a cube $Q_0\in\R^d$, then we can consider $\Omega=Q_0$ with the induced Lebesgue measure. We let $\D(Q_0)$ denote the dyadic grid inside of $Q_0$, i.e., we add $Q_0$ to the collection, subdivide it into $2^d$ new cubes obtained by splitting each side of $Q_0$ into half, and then repeat this process to all these cubes. Since any two cubes in $\D(Q_0)$ are contained in $Q_0$, the collection $\mathcal{E}=\D(Q_0)$ satisfies property \ref{it:basis2}, and hence, is a basis of sets in $Q_0$.
\end{example}

\begin{remark}
If we replace the condition \ref{it:basis2} by the weaker condition that for every $E_1,E_2\in\mathcal{E}$ \emph{such that $E_1\cap E_2\neq\emptyset$} there is an $E\in\mathcal{E}$ with $E_1\cup E_1\subseteq E$, then the standard dyadic grids do fall within this framework. However, with this weaker condition, Proposition~\ref{prop:nonzeromax} is no longer true. For example, using the standard dyadic grid in $d=1$, then for any function $f_1\in L^1_{\loc}(\R)$ that is supported on the positive reals, we will have $M^\mathcal{D}f=0$ on the negative reals. This is a problem in the extrapolation arguments, because it means that the Rubio de Francia algorithm applied to a non-zero function is not necessarily a weight, and hence, the standard extrapolation argument fails.
\end{remark}

\subsection{Muckenhoupt weights in a two-weight off-diagonal setting on general bases of sets}

\begin{definition}
Let $\alpha\in\R$ and suppose $r_1,r_2\in(0,\infty)$, $s_1\in(0,\infty]$, $\frac{1}{s_2}\in\R$ satisfy $\frac{1}{r_1}>\frac{1}{s_1}$, $\frac{1}{r_2}>\frac{1}{s_2}$
and
\[
\frac{1}{r_2}-\frac{1}{r_1}=\frac{1}{s_2}-\frac{1}{s_1}=\alpha.
\]
We note that the condition $\frac{1}{s_2}\in\R$ can be relevant in the situation where $\alpha\leq 0$. The geometry behind this is that the interval $[\frac{1}{s_2},\frac{1}{r_2}]$ is the translate by $\alpha$ of the interval $[\frac{1}{s_1},\frac{1}{r_1}]$. When $\alpha\leq 0$, we allow for the case where this translation places part of the interval in the negative reals, which corresponds to $\frac{1}{s_2}<0$. In the case $\alpha\geq 0$ this is not an issue, since then $\frac{1}{s_2}\geq\frac{1}{s_1}\geq 0$ and hence, $s_2\in(0,\infty]$.

Let $p_1,p_2\in(0,\infty]$ with $\frac{1}{p_1}\in[\frac{1}{s_1},\frac{1}{r_1}]$, $\frac{1}{p_2}\in[\frac{1}{s_2},\frac{1}{r_2}]$ and $\frac{1}{p_2}-\frac{1}{p_1}=\alpha$. Let $(\Omega,|\cdot|)$ be a measure space and let $\mathcal{E}$ be a basis of sets in $\Omega$. Then, for a pair of weights $w_1$, $w_2$ on $\Omega$, we write $(w_1,w_2)\in A_{\vec{p},(\vec{r},\vec{s})}(\mathcal{E})$ when
\[
[w_1,w_2]^\mathcal{E}_{\vec{p},(\vec{r},\vec{s})}:=\sup_{E\in\mathcal{E}}\langle w_1^{-1}\rangle_{\frac{1}{\frac{1}{r_1}-\frac{1}{p_1}},E}\langle w_2\rangle_{\frac{1}{\frac{1}{p_2}-\frac{1}{s_2}},E}<\infty.
\]
Note that in this case we must have $w_1^{-1}\in L^{\frac{1}{\frac{1}{r_1}-\frac{1}{p_1}}}_{\mathcal{E}}(\Omega)$ and $w_2\in L^{\frac{1}{\frac{1}{p_2}-\frac{1}{s_2}}}_{\mathcal{E}}(\Omega)$.

\begin{proposition}
Let $\vec{p}$, $\vec{r}$, $\vec{s}$ be as above and let $\mathcal{E}$, $\mathcal{F}$ be two bases of sets in $\Omega$ with $\mathcal{E}\lesssim\mathcal{F}$. Then $A_{\vec{p},(\vec{r},\vec{s})}(\mathcal{F})\subseteq A_{\vec{p},(\vec{r},\vec{s})}(\mathcal{E})$ and for all $(w_1,w_2)\in A_{\vec{p},(\vec{r},\vec{s})}(\mathcal{F})$ we have

\[
[w_1,w_2]^\mathcal{E}_{\vec{p},(\vec{r},\vec{s})}\lesssim_{\vec{r},\vec{s}}[w_1,w_2]^\mathcal{F}_{\vec{p},(\vec{r},\vec{s})}.
\]
In particular, if $\mathcal{E}\sim\mathcal{F}$, then $A_{\vec{p},(\vec{r},\vec{s})}(\mathcal{E})=A_{\vec{p},(\vec{r},\vec{s})}(\mathcal{F})$ with 
\[
[w_1,w_2]^\mathcal{E}_{\vec{p},(\vec{r},\vec{s})}\eqsim_{\vec{r},\vec{s}} [w_1,w_2]^\mathcal{F}_{\vec{p},(\vec{r},\vec{s})}.
\]
\end{proposition}
\begin{proof}
Let $E\in\mathcal{E}$ and pick $F\in\mathcal{F}$ with $E\subseteq F$ and $|F|\leq C|E|$. Then
\begin{align*}
\langle w_1^{-1}\rangle_{\frac{1}{\frac{1}{r_1}-\frac{1}{p_1}},E}\langle w_2&\rangle_{\frac{1}{\frac{1}{p_2}-\frac{1}{s_2}},E}
=|E|^{-\big(\frac{1}{r_1}-\frac{1}{s_1}\big)}\|w_1^{-1}\ind_E\|_{L^{\frac{1}{\frac{1}{r_1}-\frac{1}{p_1}}}(\Omega)}\|w_2\ind_E\|_{L^{\frac{1}{\frac{1}{p_2}-\frac{1}{s_2}}}(\Omega)}\\
&\leq C^{\frac{1}{r_1}-\frac{1}{s_1}}|F|^{-\big(\frac{1}{r_1}-\frac{1}{s_1}\big)}\|w_1^{-1}\ind_F\|_{L^{\frac{1}{\frac{1}{r_1}-\frac{1}{p_1}}}(\Omega)}\|w_2\ind_F\|_{L^{\frac{1}{\frac{1}{p_2}-\frac{1}{s_2}}}(\Omega)}\\
&\leq C^{\frac{1}{r_1}-\frac{1}{s_1}}[w_1,w_2]^\mathcal{F}_{\vec{p},(\vec{r},\vec{s})}.
\end{align*}
Taking a supremum over all $E\in\mathcal{E}$ proves the assertion.
\end{proof}

For a weight $w$ in $\Omega$ we write $w\in A_{\vec{p},(\vec{r},\vec{s})}(\mathcal{E})$ when
\[
[w]^\mathcal{E}_{\vec{p},(\vec{r},\vec{s})}:=[w,w]^{\mathcal{E}}_{\vec{p},(\vec{r},\vec{s})}<\infty.
\]
\end{definition}

When $\alpha=0$, we have $p:=p_1=p_2$, $r:=r_1=r_2$, $s:=s_1=s_2$ and write $w\in A_{p,(r,s)}(\mathcal{E})$ when
\[
[w]^\mathcal{E}_{p,(r,s)}:=[w]^\mathcal{E}_{\vec{p},(\vec{r},\vec{s})}=\sup_{E\in\mathcal{E}}\langle w^{-1}\rangle_{\frac{1}{\frac{1}{r}-\frac{1}{p}},E}\langle w\rangle_{\frac{1}{\frac{1}{p}-\frac{1}{s}},E}<\infty.
\]

We also note that with this notation we have
\[
[w]_{\vec{p},(\vec{r},\vec{s})}=[w]_{p_1,(r_1,s_1)}=[w]_{p_2,(r_2,s_2)}.
\]

We sometimes also write $[w]^{\mathcal{E}}_{p}:=[w]^\mathcal{E}_{p,(1,\infty)}$ in the case where $r=1$ and $s=\infty$. We now have the symmetry
\begin{equation}\label{eq:wfullrangesym}
[w]^\mathcal{E}_p=[w^{-1}]_{p'}^\mathcal{E}.
\end{equation}
When $p=1$, the condition $w\in A_1(\mathcal{E})$ is equivalent to there being a constant $C>0$ such that
\[
M^\mathcal{E}w\leq Cw,
\]
and the smallest possible constant in this inequality is equal to $[w]_1^\mathcal{E}$. We also point out that by the symmetry \eqref{eq:wfullrangesym} we have $w\in A_\infty(\mathcal{E})$ if and only if $w^{-1}\in A_1(\mathcal{E})$. Since the notation $A_\infty(\mathcal{E})$ can be easily be confused with the class $A_\infty=\bigcup_{w\in A_p}A_p$ when we are in the classical setting of $\R^d$ with the basis of cubes, we often opt to write $w^{-1}\in A_1(\mathcal{E})$ instead of $w\in A_\infty(\mathcal{E})$.

If $w\in A_1(\mathcal{E})$, $v^{-1}\in A_1(\mathcal{E})$, then for $p\in[1,\infty]$ we have $w^{\frac{1}{p}}v^{\frac{1}{p'}}\in A_p(\mathcal{E})$ with
\[
[w^{\frac{1}{p}} v^{\frac{1}{p'}}]^\mathcal{E}_p\leq\big([w]_1^\mathcal{E}\big)^{\frac{1}{p}}\big([v]_\infty^\mathcal{E}\big)^{\frac{1}{p'}}.
\]
We wish to explore this interpolation result in the limited range case. We note that the symmetry \eqref{eq:wfullrangesym} now takes the form
\[
[w]^{\mathcal{E}}_{p,(r,s)}
=[w^{-1}]^\mathcal{E}_{\frac{1}{\frac{1}{r}+\frac{1}{s}-\frac{1}{p}},(r,s)},
\]
and we have that $w\in A_{r,(r,s)}$ if and only if $w^{-1}\in A_{s,(r,s)}$, if and only if there is a constant $C>0$ such that
\[
M^\mathcal{E}_{\frac{1}{\frac{1}{r}-\frac{1}{s}}}w\leq C w,
\]
and the smallest possible constant is $C=[w]^\mathcal{E}_{r,(r,s)}=[w^{-1}]_{s,(r,s)}$. We then have the following result:
\begin{proposition}\label{prop:limitedrangeweightinterpolation}
Let $w\in A_{r,(r,s)}(\mathcal{E})$, $v\in A_{s,(r,s)}(\mathcal{E})$. Then for $p\in[r,s]$ we have
\[
w^{\frac{\frac{1}{p}-\frac{1}{s}}{\frac{1}{r}-\frac{1}{s}}} v^{\frac{\frac{1}{r}-\frac{1}{p}}{\frac{1}{r}-\frac{1}{s}}}\in A_{p,(r,s)}(\mathcal{E})
\]
with
\[
\big[w^{\frac{\frac{1}{p}-\frac{1}{s}}{\frac{1}{r}-\frac{1}{s}}} v^{\frac{\frac{1}{r}-\frac{1}{p}}{\frac{1}{r}-\frac{1}{s}}}\big]_{p,(r,s)}^\mathcal{E}
\leq \big([w]^\mathcal{E}_{r,(r,s)}\big)^{\frac{\frac{1}{p}-\frac{1}{s}}{\frac{1}{r}-\frac{1}{s}}}
\big([v]^\mathcal{E}_{s,(r,s)}\big)^{\frac{\frac{1}{r}-\frac{1}{p}}{\frac{1}{r}-\frac{1}{s}}}.
\]
\end{proposition}
To make sense of the exponents, we consider the geometry of the situation. The interval $[\frac{1}{s},\frac{1}{r}]$ is mapped to $[0,1]$ through the affine transformation $\phi$ given by
\[
\phi(t):=\frac{t-\frac{1}{s}}{\frac{1}{r}-\frac{1}{s}}.
\]
For $p\in[r,s]$, we then have
\[
\phi\Big(\frac{1}{p}\Big)=\frac{\frac{1}{p}-\frac{1}{s}}{\frac{1}{r}-\frac{1}{s}},\quad 1-\phi\Big(\frac{1}{p}\Big)=\frac{\frac{1}{r}-\frac{1}{p}}{\frac{1}{r}-\frac{1}{s}},
\]
which are the exponents appearing in the proposition.
\begin{proof}[Proof of Proposition~\ref{prop:limitedrangeweightinterpolation}]
Write
\[
\theta:=\frac{\frac{1}{r}-\frac{1}{p}}{\frac{1}{r}-\frac{1}{s}}\in[0,1]
\]
so that
\[
\frac{1}{p}=\frac{1-\theta}{r}+\frac{\theta}{s}.
\]
Since then
\[
\frac{1}{p}-\frac{1}{s}=(1-\theta)\big(\frac{1}{r}-\frac{1}{s}\big),\quad \frac{1}{r}-\frac{1}{p}=\theta\big(\frac{1}{r}-\frac{1}{s}\big)
\]
for any $E\in\mathcal{E}$ we have
\[
\langle w^{1-\theta} v^\theta\rangle_{\frac{1}{\frac{1}{p}-\frac{1}{s}},E}\leq\langle w\rangle^{1-\theta}_{\frac{1}{\frac{1}{r}-\frac{1}{s}},E}\langle v\rangle_{\infty,E}^\theta
\]
and
\[
\langle w^{-(1-\theta)} v^{-\theta}\rangle_{\frac{1}{\frac{1}{r}-\frac{1}{p}},E}\leq\langle w^{-1}\rangle^{1-\theta}_{\infty,E}\langle v^{-1}\rangle^\theta_{\frac{1}{\frac{1}{r}-\frac{1}{s}},E}.
\]
Combining these estimates yields
\[
\langle w^{1-\theta} v^\theta\rangle_{\frac{1}{\frac{1}{p}-\frac{1}{s}},E}\langle w^{-(1-\theta)} v^{-\theta}\rangle_{\frac{1}{\frac{1}{r}-\frac{1}{p}},E}
\leq \big([w]_{r,(r,s)}^\mathcal{E}\big)^{1-\theta}\big([v]_{s,(r,s)}^\mathcal{E}\big)^\theta.
\]
Taking a supremum over all $E\in\mathcal{E}$ proves the result.
\end{proof}

Finally, we wish to point out that in the literature, the class $A_{p,(r,s)}(\mathcal{E})$ is often denoted as the class of weights $w$ with $w^p\in A_{\frac{p}{r}}\cap RH_{(s/p)'}$. The equivalence of the two first appeared in \cite{JN91} and in our situation takes the following form:
\begin{proposition}\label{prop:jnreverseholder}
Let $r\in(0,\infty)$, $s\in(0,\infty]$ with $r<s$, and $p\in[r,s]$. Let $\mathcal{E}$ be a basis of sets in $\Omega$. Then the following are equivalent:
\begin{enumerate}[(i)]
\item\label{it:jnequiv1} $w\in A_{p,(r,s)}(\mathcal{E})$;
\item\label{it:jnequiv2} $w\in A_{p,(r,\infty)}(\mathcal{E})$ and there is a $C\geq 1$ such that for all $E\in\mathcal{E}$ we have
\[
\langle w\rangle_{\frac{1}{\frac{1}{p}-\frac{1}{s}},E}\leq C\langle w\rangle_{p,E}.
\]
\end{enumerate}
Moreover, if $C$ is the optimal constant in \ref{it:jnequiv2}, then we have
\begin{equation}\label{eq:jnequiv1}
\max(C,[w]^\mathcal{E}_{p,(r,\infty)})\leq[w]^\mathcal{E}_{p,(r,s)}\leq C[w]^\mathcal{E}_{p,(r,\infty)}.
\end{equation}
\end{proposition}
\begin{proof}
It suffices to prove \eqref{eq:jnequiv1}. Since $\frac{1}{p}-\frac{1}{s}\leq\frac{1}{p}$, the inequality $[w]^\mathcal{E}_{p,(r,\infty)}\leq[w]^\mathcal{E}_{p,(r,s)}$ follows from H\"older's inequality. Moreover, since  for any $E\in\mathcal{E}$ we have
\[
1=\langle ww^{-1}\rangle_{r,E}\leq \langle w\rangle_{p,E}\langle w^{-1}\rangle_{\frac{1}{\frac{1}{r}-\frac{1}{p}},E},
\]
it follows that
\[
\langle w\rangle_{\frac{1}{\frac{1}{p}-\frac{1}{s}},E}\leq[w]_{p,(r,s)}^\mathcal{E}\langle w\rangle_{p,E},
\]
proving that $C\leq[w]_{p,(r,s)}$. This proves the first inequality in \eqref{eq:jnequiv1}.

For the second one, note that
\[
\langle w\rangle_{\frac{1}{\frac{1}{p}-\frac{1}{s}},E}\langle w^{-1}\rangle_{\frac{1}{\frac{1}{r}-\frac{1}{p}},E}\leq C\langle w\rangle_{p,E}\langle w^{-1}\rangle_{\frac{1}{\frac{1}{r}-\frac{1}{p}},E}\leq C[w]_{p,(r,\infty)}
\]
for all $E\in\mathcal{E}$. The result follows.
\end{proof}

When $\mathcal{E}$ is the classical basis of cubes in $\R^d$, then we note that $w\in A_p(\mathcal{E})$, if and only if $w^p$ is in the classical Muckenhoupt $A_p$ class, and in this case we have
\[
[w]_p:=[w]^{\mathcal{E}}_{p}=[w^p]^{\frac{1}{p}}_{A_p}.
\]

\subsection{Maximal operators in quasi-Banach function spaces}\label{subsec:maxinqbfs}
In this section we prove some results related to the boundedness of $M^\mathcal{E}$ on (quasi-)Banach function spaces $X$.

\begin{proposition}\label{prop:maxopnorm}
Let $\mathcal{E}$ be a basis of sets in $\Omega$. Let $X$ be a quasi-Banach function space over $\Omega$ and suppose that $M^\mathcal{E}:X\to X$ is bounded. Then
\[
\|M^\mathcal{E}\|_{X\to X}\geq 1.
\]
\end{proposition}
\begin{lemma}\label{lem:indE}
Suppose $X$ is a quasi-Banach function space for which 
\[
M^\mathcal{E}:X\to X
\]
is bounded. Then $\ind_E\in X$ and $\ind_E\in X'$ with
\[
|E|\leq\|\ind _E\|_X\|\ind_E\|_{X'}\leq\|M^\mathcal{E}\|_{X\to X}|E|.
\]
for all $E\in\mathcal{E}$.

In particular, $X'$ is saturated and, hence, a Banach function space.
\end{lemma}
\begin{proof}
By Proposition~\ref{prop:weakorderunit} there is a weak order unit $\rho\in X$. Fix $E\in\mathcal{E}$. Then by the definition of $M^\mathcal{E}$ we have
\[
\langle\rho\rangle_{1,E}\ind_E\leq M^\mathcal{E}\rho\in X.
\]
Since $\langle\rho\rangle_{1,E}>0$, it follows from the ideal property that $\ind_E\in X$.

To see that also $\ind_E\in X'$, note that for any $f\in X$ we have
\begin{align*}
\int_\Omega\!\ind_E |f|\,\mathrm{d}x
&=\langle f\rangle_{1,E}|E|=\frac{|E|}{\|\ind_E\|_X}\|\langle f\rangle_{1,E}\ind_E\|_X\\
&\leq\|M^\mathcal{E}\|_{X\to X}\frac{|E|}{\|\ind_E\|_X}\|f\|_X.
\end{align*}
Hence, $\ind_E\in X'$ with
\[
\|\ind_E\|_{X'}\leq\|M^\mathcal{E}\|_{X\to X}\frac{|E|}{\|\ind_E\|_X}.
\]
Noting that also $|E|=\|\ind_E\ind_E\|_{L^1(\Omega)}\leq\|\ind_E\|_X\|\ind_{E'}\|_{X'}$, the first assertion follows.

Finally, to see that $X'$ is saturated, let $\mathcal{E'}$ be as in property \ref{it:basis3} of a basis of sets and let $E\subseteq\Omega$ with $|E|>0$. Since $\mathcal{E'}$ is countable and satisfies $\bigcup_{E'\in\mathcal{E}'} E'=\Omega$, there exists a set $E'\in\mathcal{E}'$ such that $|E'\cap E|>0$. Since $\ind_{E'\cap E}\leq\ind_{E'}\in X'$, it follows from the ideal property of $X'$ that $\ind_{E'\cap E}\in X'$. We conclude that the set $F:=E'\cap E\subseteq E$ has the desired properties, proving that $X'$ is saturated, as asserted.
\end{proof}
\begin{proof}[Proof of Proposition~\ref{prop:maxopnorm}]
We have
\[
\|\ind_E\|_X=\|\langle \ind_E\rangle_{1,E}\ind_E\|_X\leq\|M^\mathcal{E}\ind_E\|_X\leq\|M^\mathcal{E}\|_{X\to X}\|\ind_E\|_X.
\]
Hence, by Lemma~\ref{lem:indE}, $\|M^\mathcal{E}\|_{X\to X}\geq 1$, as asserted.
\end{proof}

Since $X'$ is saturated when $M^\mathcal{E}$ is bounded on $X$, this means that $X''$ is a Banach function space containing $X$ by Proposition~\ref{prop:envelopeembedding}. As it turns out, in this case we also have $M^\mathcal{E}:X''\to X''$.
\begin{proposition}
Let $\mathcal{E}$ be a basis of sets in $\Omega$. Let $X$ be a quasi-Banach function space over $\Omega$ and suppose that
\[
M^\mathcal{E}:X\to X
\]
is bounded. Then 
\[
M^\mathcal{E}:X''\to X''
\]
is also bounded with $\|M^\mathcal{E}\|_{X''\to X''}\leq\|M^\mathcal{E}\|_{X\to X}$.
\end{proposition}
\begin{proof}
By Proposition~\ref{prop:mmeasurable} we may assume without loss of generality that $\mathcal{E}$ is countable by replacing it by $\mathcal{E}'$ if necessary. Equipping $\mathcal{E}$ with the counting measure, for $p\in[1,\infty]$ we define the space $X[\ell^p(\mathcal{E})]$ as the space of sequences of measurable functions $(f_E)_{E\in\mathcal{E}}$ for which
\[
\|(f_E)_{E\in\mathcal{E}}\|_{X[\ell^p(\mathcal{E})]}:=\big\|\|(f_E)_{E\in\mathcal{E}}\|_{\ell^p(\mathcal{E})}\big\|_X<\infty.
\]
For $E\in\mathcal{E}$ we define the linearized average $\langle f\rangle_E:=\frac{1}{|E|}\int_E\!f\,\mathrm{d}x$ and we define
\[
\mathcal{M}':X'[\ell^1(\mathcal{E})]\to X',\quad \mathcal{M}'g:=\sum_{E\in\mathcal{E}}\left(\frac{1}{|E|}\int_\Omega\!g_E\,\mathrm{d}x\right)\ind_E.
\]
Note that this is well-defined, since for any $f\in X$ and $g=(g_E)_{E\in\mathcal{E}}\in X'[\ell^1(\mathcal{E})]$ we have
\[
\left|\int_\Omega\!f\mathcal{M}'g\,\mathrm{d}x\right|=\left|\int_\Omega\!\sum_{E\in\mathcal{E}}\langle f\rangle_E g_E\,\mathrm{d}x\right|\leq\|M^\mathcal{E}f\|_X\|g\|_{X'[\ell^1(\mathcal{E})]}
\]
so that $\|\mathcal{M}'\|_{X'[\ell^1(\mathcal{E})]\to X'}\leq\|M^\mathcal{E}\|_{X\to X}$.

Since $\mathcal{M}'$ is a bounded linear operator between Banach spaces, it has a bounded dual operator
\[
(\mathcal{M}')^\ast:(X')^\ast\to (X'[\ell^1(\mathcal{E})])^\ast
\]
with the same operator norm as $\mathcal{M}'$. Let $\iota:X''\hookrightarrow (X')^\ast$ and $\kappa:X''[\ell^\infty(\mathcal{E})]\hookrightarrow (X'[\ell^1(\mathcal{E})])^\ast$ denote the canonical isometric embeddings
\[
 \iota(f)(g):=\int_\Omega\!fg\,\mathrm{d}x,\quad \kappa((f_E)_{E\in\mathcal{E}})((g_E)_{E\in\mathcal{E}}):=\int_\Omega\!\sum_{E\in\mathcal{E}} f_Eg_E\,\mathrm{d}x
\]
and define
\[
\mathcal{M}:X''\to X''[\ell^\infty(\mathcal{E})],\quad\mathcal{M}f:=(\langle f\rangle_{E})_{E\in\mathcal{E}}.
\]
We claim that
\[
\kappa\circ \mathcal{M}=(\mathcal{M}')^\ast\circ\iota
\]
and hence, $\mathcal{M}$ is a well-defined bounded operator with 
\begin{equation}\label{eq:mboundbidual1}
\|\mathcal{M}\|_{X''\to X''[\ell^\infty(\mathcal{E})]}\leq\|(\mathcal{M}')^\ast\|_{(X')^\ast\to (X'[\ell^1(\mathcal{E})])^\ast}=\|\mathcal{M}'\|_{X'[\ell^1(\mathcal{E})]\to X'}\leq\|M^\mathcal{E}\|_{X\to X}.
\end{equation}

To prove the claim, note that for $g=(g_E)_{E\in\mathcal{E}}\in X'[\ell^1(\mathcal{E})]$ we have
\begin{align*}
(\mathcal{M}')^\ast(\iota(f))(g)&=\iota(f)(\mathcal{M}'g)=\int_\Omega\!f\sum_{E\in\mathcal{E}}\left(\frac{1}{|E|}\int_\Omega\!g_E\,\mathrm{d}x\right)\ind_E\,\mathrm{d}y\\
&=\int_\Omega\!\sum_{E\in\mathcal{E}}\langle f\rangle_E g_E\,\mathrm{d}x=\kappa(\mathcal{M} f)(g),
\end{align*}
as desired.

Now, it follows from \eqref{eq:mboundbidual1} that
\[
\|M^\mathcal{E}f\|_{X''}=\|\mathcal{M}(|f|)\|_{X''[\ell^\infty(\mathcal{E})]}\leq\|M^\mathcal{E}\|_{X\to X}\||f|\|_{X''}=\|M^\mathcal{E}\|_{X\to X}\|f\|_{X''}.
\]
The assertion follows.
\end{proof}

Next, we present several rescaling results:
\begin{proposition}\label{prop:mqboundrescale}
Let $X$ be a quasi-Banach function space over $\Omega$ and let $q\in(0,\infty)$. Then the following are equivalent:
\begin{enumerate}[(i)]
\item\label{it:mqbound1} $M^\mathcal{E}:X^q\to X^q$ is bounded;
\item\label{it:mqbound2} $M^\mathcal{E}_q:X\to X$ is bounded.
\end{enumerate}
Moreover, in this case we have
\[
\|M_q^\mathcal{E}\|^q_{X\to X}=\|M^\mathcal{E}\|_{X^q\to X^q}.
\]
\end{proposition}
\begin{proof}
Note that
\begin{align*}
\|M^\mathcal{E}\|_{X^q\to X^q}
&=\sup_{\|f\|_{X^q}=1}\|M^\mathcal{E}f\|_{X^q}
=\sup_{\||f|^{\frac{1}{q}}\|_X=1}\|(M^\mathcal{E}f)^{\frac{1}{q}}\|_X^q\\
&=\sup_{\|h\|_X=1}\|M_q^\mathcal{E}h\|_X^q
=\|M_q^\mathcal{E}\|_{X\to X}^q,
\end{align*}
which proves the result.
\end{proof}

\begin{corollary}\label{cor:smallconcavification}
Let $X$ be a Banach function space over $\Omega$ and let $q\in(0,1)$. If $M^\mathcal{E}:X\to X$ is bounded, then so is $M^\mathcal{E}:X^q\to X^q$, with
\[
\|M^\mathcal{E}\|_{X^q\to X^q}\leq \|M^\mathcal{E}\|_{X\to X}^q.
\]
\end{corollary}
\begin{proof}
Since $M^\mathcal{E}_q f\leq M^\mathcal{E}f$ for all $f\in X$, it follows from the ideal property of $X$ that
\[
\|M^\mathcal{E}_q f\|_X\leq \| M^\mathcal{E}f\|_X.
\]
The result now follows from Proposition~\ref{prop:mqboundrescale}.
\end{proof}

\begin{proposition}\label{prop:mboundbiggerrs}
Let $r\in(0,\infty)$, $s\in(0,\infty]$ with $r<s$, and let $X$ be an $r$-convex and $s$-concave quasi-Banach function space over $\Omega$. Suppose that $M^\mathcal{E}:L^p(\Omega)\to L^p(\Omega)$ is bounded for all $p\in(1,\infty]$, and that
\[
M^\mathcal{E}:X_{r,s}\to X_{r,s}
\]
is bounded.

Then for every $0<\widetilde{r}<r$, $s\leq\widetilde{s}\leq\infty$ we also have that
\[
M^\mathcal{E}:X_{\widetilde{r},\widetilde{s}}\to X_{\widetilde{r},\widetilde{s}},
\]
with
\[
\|M^\mathcal{E}\|^{\frac{1}{\widetilde{r}}-\frac{1}{\widetilde{s}}}_{X_{\widetilde{r},\widetilde{s}}\to X_{\widetilde{r},\widetilde{s}}}\lesssim_{\widetilde{r},\widetilde{s}} \|M^\mathcal{E}\|^{\frac{1}{r}-\frac{1}{s}{}}_{X_{r,s}\to X_{r,s}}.
\]
In particular, we have
\[
M^\mathcal{E}:X^{\widetilde{r}}\to X^{\widetilde{r}}
\]
is bounded for all $0<\widetilde{r}<r$ with
\[
\|M^\mathcal{E}\|^{\frac{1}{\widetilde{r}}}_{X^{\widetilde{r}}\to X^{\widetilde{r}}}\lesssim_{\widetilde{r}} \|M^\mathcal{E}\|^{\frac{1}{r}-\frac{1}{s}{}}_{X_{r,s}\to X_{r,s}}.
\]
\end{proposition}
For the proof we require the following lemma:
\begin{lemma}\label{lem:mboundinterpolation}
Let $\theta\in[0,1]$ and suppose that $X$, $Y$ are two quasi-Banach function spaces over $\Omega$ for which $M^\mathcal{E}:X\to X$ and $M^\mathcal{E}:Y\to Y$ are bounded. Then
\[
M^\mathcal{E}:X^{1-\theta}\cdot Y^\theta\to X^{1-\theta}\cdot Y^\theta
\]
is bounded with
\[
\|M^\mathcal{E}\|_{X^{1-\theta}\cdot Y^\theta\to X^{1-\theta}\cdot Y^\theta}\leq\|M^\mathcal{E}\|_{X\to X}^{1-\theta}\|M^\mathcal{E}\|_{Y\to Y}^{\theta}.
\]
\end{lemma}
\begin{proof}
Let $f\in X^{1-\theta}\cdot Y^\theta$. Then there exist $0\leq h\in X$, $0\leq k\in Y$ such that $|f|\leq h^{1-\theta}k^\theta$. Hence, by H\"older's inequality,
\[
\langle f\rangle_{1,E}\leq\langle h\rangle_{1,E}^{1-\theta}\langle k\rangle^\theta_{1,E}
\]
for all $E\in\mathcal{E}$ so that $M^\mathcal{E} f\leq (M^\mathcal{E}h)^{1-\theta}(M^\mathcal{E}k)^\theta$. Hence,
\[
\|M^\mathcal{E} f\|_{X^{1-\theta}\cdot Y^\theta}\leq\|M^\mathcal{E} h\|_X^{1-\theta}\|M^\mathcal{E} k\|_Y^\theta
\leq \|M^\mathcal{E}\|_{X\to X}^{1-\theta}\|M^\mathcal{E}\|_{Y\to Y}^{\theta}\|h\|_X^{1-\theta}\| k\|_Y^\theta.
\]
Taking an infimum over all possible $0\leq h\in X$, $0\leq k\in Y$ with $|f|\leq h^{1-\theta}k^\theta$ now proves the result.
\end{proof}
\begin{lemma}\label{lem:mboundbiggerrs}
Let $0<\widetilde{r}\leq r$, $s\leq\widetilde{s}\leq\infty$ and let $X$ be an $r$-convex and $s$-concave quasi-Banach function space over $\Omega$. Set
\[
\theta:=\frac{\frac{1}{r}-\frac{1}{s}}{\frac{1}{\widetilde{r}}-\frac{1}{\widetilde{s}}}\in(0,1]
\]
and
\[
\alpha=\frac{\frac{1}{\widetilde{r}}-\frac{1}{r}}{\frac{1}{\widetilde{r}}-\frac{1}{r}+\frac{1}{s}-\frac{1}{\widetilde{s}}}\in[0,1].
\]
Then
\[
X_{\widetilde{r},\widetilde{s}}=X_{r,s}^\theta\cdot (L^{\frac{1}{1-\alpha}}(\Omega))^{1-\theta}.
\]
\end{lemma}
\begin{proof}
By Theorem~\ref{thm:spacesplitting} we have
\[
(X^{\widetilde{r}})'=(X^{r\frac{\widetilde{r}}{r}})'=[(X^r)']^{\frac{\widetilde{r}}{r}}\cdot L^{\big(\frac{r}{\widetilde{r}}\big)'}(\Omega).
\]
Hence, we have
\[
[(X^{\widetilde{r}})']^{\big(\frac{\widetilde{s}}{\widetilde{r}}\big)'}=\Big([(X^r)']^{\big(\frac{s}{r}\big)'}\Big)^{\theta}
\cdot (L^{\frac{1}{\alpha}}(\Omega))^{1-\theta},
\]
so that the result follows from Lozanovskii's duality theorem.
\end{proof}

\begin{proof}[Proof of Proposition~\ref{prop:mboundbiggerrs}]
We wish to combine Lemma~\ref{lem:mboundinterpolation} and Lemma~\ref{lem:mboundbiggerrs}. Since $\widetilde{r}<r$ we have $\alpha>0$ and hence, $\frac{1}{1-\alpha}>1$. Thus, we obtain
\[
\|M^\mathcal{E}\|_{X_{\widetilde{r},\widetilde{s}}\to X_{\widetilde{r},\widetilde{s}}}
\leq \|M^\mathcal{E}\|_{X_{r,s}\to X_{r,s}}^\theta\|M^\mathcal{E}\|_{L^{\frac{1}{1-\alpha}}(\Omega)\to L^{\frac{1}{1-\alpha}}(\Omega)}^{1-\theta},
\]
which proves the result.
\end{proof}

If we are in $\R^d$ with one of the usual bases of balls, cubes, or a dyadic grid, we have the property that if $w\in A_1$, then there is a dimensional constant $C_d>0$ such that if $q>1$ satisfies $C_dq'\geq[w]_{A_1}$, then $w$ satisfies the reverse H\"older property
\[
\langle w\rangle_{q,Q}\leq 2\langle w\rangle_{1,Q},
\]
and hence, $M^\mathcal{E}_q w\leq 2 M^\mathcal{E} w$.
\begin{definition}
Let $\mathcal{E}$ be a basis of sets in $\Omega$. We say that $\mathcal{E}$ has the \emph{$A_1$ self-improvement property} if for each $w\in A_1(\mathcal{E})$ there is a $q>1$ such that $w^q\in A_1(\mathcal{E})$.

We say that $\mathcal{E}$ has the \emph{sharp $A_1$ self-improvement property} if there are constants $C_1,C_2>0$ such that for every $w\in A_1(\mathcal{E})$ and any $q>1$ satisfying $C_1q'\geq [w]_1^\mathcal{E}$ we have $M_q^\mathcal{E}w\leq C_2 M^\mathcal{E} w$.
\end{definition}

With this property, we get the following self-improvement result, which is a sharp version of one of the main results in \cite{LO10}:
\begin{theorem}\label{thm:maximalselfimprove}
Let $\mathcal{E}$ be a basis of sets in $\Omega$ with the $A_1$ self-improvement property.

Let $X$ be a Banach function space over $\Omega$ and suppose that
\[
M^\mathcal{E}:X\to X
\]
is bounded. Then there exists a $q>1$ such that
\[
M_q^\mathcal{E}:X\to X,
\]
is bounded. If $\mathcal{E}$ has the sharp $A_1$ self-improvement property, then we also obtain the following quantitative statement:

Let $q>1$ with $C_1q'\geq 2\|M\|_{X\to X}$. Then
\begin{equation}\label{eq:sharpsi1}
M_q^\mathcal{E}f\leq C_2\sum_{k=0}^\infty\frac{(M^\mathcal{E})^{k+1}f}{(C_1q')^k}
\end{equation}
for every $f\in X$, and
\begin{equation}\label{eq:sharpsi2}
\|M_q^\mathcal{E}\|_{X\to X}\leq 2C_2\|M^\mathcal{E}\|_{X\to X}.
\end{equation}
\end{theorem}
\begin{proof}
Fix $f\in X$, $K\geq 2\|M\|_{X\to X}$, and define
\[
w:=\sum_{k=0}^\infty\frac{(M^\mathcal{E})^k f}{K^k}.
\]
then $w\in A_1(\mathcal{E})$ with $[w]^\mathcal{E}_1\leq K$. Hence, per assumption, there is a $q>1$ such that $w^q\in A_1(\mathcal{E})$. Since $|f|\leq w$, we have
\begin{equation}\label{eq:mqestimate1}
M_q^\mathcal{E}f\leq M_q^\mathcal{E}w\leq([w^q]^\mathcal{E}_1)^{\frac{1}{q}} w,
\end{equation}
and thus, since $\|w\|_X\leq 2\|f\|_X$, we have
\begin{equation}\label{eq:mqestimate2}
\|M_q^\mathcal{E} f\|_X\leq([w^q]^\mathcal{E}_1)^{\frac{1}{q}}\|w\|_X\leq 2([w^q]_1^\mathcal{E})^{\frac{1}{q}}\|f\|_X,
\end{equation}
proving the first result.

For the second result, suppose $q>1$ satisfies $C_1q'\geq 2\|M\|_{X\to X}$ and choose $K=C_1q'$. Then we have $[w]_1^\mathcal{E}\leq C_1q'$ and hence, it follows from the first inequality in \eqref{eq:mqestimate1} that
\[
M_q^\mathcal{E} f\leq C_2 M^\mathcal{E} w\leq C_2\sum_{k=0}^\infty\frac{(M^\mathcal{E})^{k+1}f}{(C_1q')^k},
\]
proving \eqref{eq:sharpsi1}. Finally, we note that \eqref{eq:sharpsi1} implies that
\[
\|M_q^\mathcal{E}f\|_X\leq C_2\sum_{k=0}^\infty\frac{\|(M^\mathcal{E})^{k+1}f\|_X}{2^k\|M\|^k_X}\leq 2C_2\|M\|_{X\to X}\|f\|_X,
\]
proving \eqref{eq:sharpsi2}. The assertion follows.
\end{proof}

\begin{remark}
Let $\Omega=\R^d$ with the basis of cubes. As a curiosity, we note that a consequence of this result is that there are dimensional constants $c_d,C_d>0$ such that for $q>1$ we have
\begin{equation}\label{eq:mrsi}
\sum_{k=0}^\infty\frac{M^{k+1}f}{(c_d q')^k}\lesssim M_q f\lesssim \sum_{k=0}^\infty\frac{M^{k+1}f}{(C_d q')^k}.
\end{equation}
Indeed, the second inequality follows from Theorem~\ref{thm:maximalselfimprove}, whereas for the first inequality we can use \cite[Theorem~7.2.7]{Gr14a}. Indeed, this result states that there is a dimensional constant $K_d>0$ such that $M(M_qf)\leq K_d q' M_q f$. Hence,
\[
M^{k+1}f=M^k(Mf)\leq M^k(M_q f)\leq (K_d q')^k M_q f.
\]
so that
\[
\sum_{k=0}^\infty\frac{M^{k+1}f}{(2K_d q')^k}\leq \sum_{k=0}^\infty\frac{M_q f}{2^k}=2 M_q f,
\]
proving the first inequality in \eqref{eq:mrsi} with $c_d=2K_d$. Interestingly enough, the equivalence \eqref{eq:mrsi} for $M_q$ is certainly well-known and is implicitly at the heart of most self-improvement results for $M$, but to the author's knowledge, it has not been explicitly written in the literature.
\end{remark}

\begin{theorem}\label{thm:mboundnos}
Let $\mathcal{E}$ be a basis of sets in $\Omega$ with the $A_1$ self-improvement property. Let $r,r_0\in(0,\infty)$, $s\in(0,\infty]$ with $r<r_0<s$, and let $X$ be an $r_0$-convex and $s$-concave quasi-Banach function space. If $M^\mathcal{E}:L^p(\Omega)\to L^p(\Omega)$ is bounded for all $p\in(1,\infty]$ and $M^\mathcal{E}:X_{r,s}\to X_{r,s}$ is bounded, then
\[
M^\mathcal{E}:X^r\to X^r
\]
and, if $\mathcal{E}$ has the sharp $A_1$ self-improvement property, we have
\begin{equation}\label{eq:mboundthm1}
\|M^\mathcal{E}\|_{X^r\to X^r}^{\frac{1}{r}}\lesssim_{r,s} \|M^\mathcal{E}\|^{\frac{1}{r}-\frac{1}{s}}_{X_{r,s}\to X_{r,s}},
\end{equation}
where the implicit constant depend on the constants $C_1$, $C_2$ appearing in the sharp $A_1$ self-improvement property.
\end{theorem}
\begin{lemma}\label{lem:mboundsir}
Let $\mathcal{E}$ be a basis of sets in $\Omega$ with the $A_1$ self-improvement property. Let $r,r_0\in(0,\infty)$, $s\in(0,\infty]$ with $r<r_0<s$, and let $X$ be an $r_0$-convex and $s$-concave quasi-Banach function space. If
\[
M^\mathcal{E}:X_{r,s}\to X_{r,s}
\]
is bounded, then there is an $r<\widetilde{r}<s$ such that
\[
M^\mathcal{E}:X_{\widetilde{r},s}\to X_{\widetilde{r},s}
\]
is bounded. If $\mathcal{E}$ has the sharp $A_1$ self-improvement property, then we can choose $r<\widetilde{r}<s$ such that
\[
\|M^\mathcal{E}\|^{\frac{1}{\widetilde{r}}-\frac{1}{s}}_{X_{\widetilde{r},s}\to X_{\widetilde{r},s}}\leq (2C_2)^{\frac{1}{r}-\frac{1}{s}}\|M^\mathcal{E}\|^{\frac{1}{r}-\frac{1}{s}}_{X_{r,s}\to X_{r,s}}.
\]
\end{lemma}
\begin{proof}
By Theorem~\ref{thm:maximalselfimprove} there is a $q>1$ such that
\[
M^\mathcal{E}:(X_{r,s})^q\to (X_{r,s})^q
\]
is bounded. Now set
\[
\frac{1}{\widetilde{r}}:=\frac{1}{q}\frac{1}{r}+\Big(1-\frac{1}{q}\Big)\frac{1}{s}\in\Big(\frac{1}{s},\frac{1}{r}\Big),
\]
so that
\[
\frac{1}{q}=\frac{\frac{1}{\widetilde{r}}-\frac{1}{s}}{\frac{1}{r}-\frac{1}{s}}.
\]
Choosing $q>1$ small enough, we have $\widetilde{r}\leq r_0$. Then it follows from Lemma~\ref{lem:mboundbiggerrs} that
\begin{equation}\label{eq:qtilder}
X_{\widetilde{r},s}=(X_{r,s})^q,
\end{equation}
proving the first result.

For the second assertion, note that by Theorem~\ref{thm:maximalselfimprove} we can choose $q>1$ small enough so that also
\[
C_1q'\geq 2\|M^\mathcal{E}\|_{X_{r,s}\to X_{r,s}},
\]
in which case we have
\[
\|M^\mathcal{E}\|^{\frac{1}{q}}_{X_{\widetilde{r},s}\to X_{\widetilde{r},s}}\leq 2C_2\|M^\mathcal{E}\|_{X_{r,s}\to X_{r,s}}
\]
by \eqref{eq:qtilder} and Proposition~\ref{prop:mqboundrescale}. This proves the second assertion.
\end{proof}

\begin{proof}[Proof of Theorem~\ref{thm:mboundnos}]
By Lemma~\ref{lem:mboundsir} we can pick $r<\widetilde{r}<s$ such that $M^\mathcal{E}:X_{\widetilde{r},s}\to X_{\widetilde{r},s}$ is bounded. Then, since $r<\widetilde{r}$, by Proposition~\ref{prop:mboundbiggerrs} we have that $M:X^r\to X^r$ is bounded, with
\begin{equation}\label{eq:eq:mboundthm1lhs}
\|M^\mathcal{E}\|^{\frac{1}{r}}_{X^r\to X^r}\lesssim_r\|M^\mathcal{E}\|^{\frac{1}{\widetilde{r}}-\frac{1}{s}}_{X_{\widetilde{r},s}\to X_{\widetilde{r},s}},
\end{equation}
proving the assertion. Moreover, if $\mathcal{E}$ satisfies the sharp $A_1$ self-improvement property, then we can pick $r<\widetilde{r}<s$ such that
\[
\|M^\mathcal{E}\|^{\frac{1}{\widetilde{r}}-\frac{1}{s}}_{X_{\widetilde{r},s}\to X_{\widetilde{r},s}}\lesssim_{r,s}\|M^\mathcal{E}\|^{\frac{1}{r}-\frac{1}{s}}_{X_{r,s}\to X_{r,s}},
\]
which by \eqref{eq:eq:mboundthm1lhs} proves \eqref{eq:mboundthm1}.
\end{proof}

In many of our upcoming results, we assume that $M^\mathcal{E}$ is bounded on both $X_{r,s}$ and $(X_{r,s})'$. However, there are various hints, which we discuss in this work, which show that we might only require $M^\mathcal{E}$ to be bounded on $X^r$ and $[(X^r)']^{\big(\frac{s}{r}\big)'}$. By Theorem~\ref{thm:mboundnos}, this could be a bigger class of spaces. However, in a hopeful note, we conjecture the following holds:
\begin{conjecture}\label{con:bfsboundconjecture1}
Let $X$ be an $r$-convex and $s$-concave quasi-Banach function space over $\Omega$ and suppose that $\mathcal{E}$ has the $A_1$ self-improvement property. Suppose that
\[
M^\mathcal{E}:[(X^r)']^{\big(\frac{s}{r}\big)'}\to[(X^r)']^{\big(\frac{s}{r}\big)'}\
\]
is bounded. Then the statements
\begin{enumerate}[(i)]
\item\label{it:mboundthm1} $M^\mathcal{E}:X_{r,s}\to X_{r,s}$ is bounded;
\item\label{it:mboundthm2} $M^\mathcal{E}:X^r\to X^r$ is bounded.
\end{enumerate}
are equivalent.
\end{conjecture}

As a matter of fact, consider the situation $\Omega=\R^d$, where $\mathcal{E}$ is the collection of cubes in $\R^d$. If $X=L^p_w(\R^d)$ for $p\in(r,s)$, then we have that $w\in A_{p,(r,s)}$ if and only if $M$ is bounded on either of the spaces
\[
X_{r,s}=L_{w^{\frac{1}{\frac{1}{r}-\frac{1}{s}}}}^{\frac{\frac{1}{r}-\frac{1}{s}}{\frac{1}{p}-\frac{1}{s}}}(\R^d),\quad (X_{r,s})'=L_{w^{-\frac{1}{\frac{1}{r}-\frac{1}{s}}}}^{\frac{\frac{1}{r}-\frac{1}{s}}{\frac{1}{r}-\frac{1}{p}}}(\R^d).
\]
But, as one can check, when $1\leq r<s\leq\infty$, then $w\in A_{p,(r,s)}$ is equivalent to the assertion that both $w\in A_{p,(r,\infty)}$ and $w^{-1}\in A_{p',(s',\infty)}$, i.e., $M$ is bounded on both
\[
X^r=L^{\frac{p}{r}}_{w^r}(\R^d),\quad (X')^{s'}=L^{\frac{p'}{s'}}_{w^{-s'}}(\R^d).
\]
This leads us to the following conjecture:
\begin{conjecture}\label{con:bfsboundconjecture2}
Let $X$ be an $r$-convex and $s$-concave Banach function space over $\Omega$ and suppose $\mathcal{E}$ has the $A_1$ self-improvement property. Then the following are equivalent:
\begin{enumerate}[(i)]
\item\label{it:mboundbfsthm1} $M^\mathcal{E}:X_{r,s}\to X_{r,s}$ and $M^\mathcal{E}:(X_{r,s})'\to (X_{r,s})'$ are bounded;
\item\label{it:mboundbfsthm2} $M^\mathcal{E}:X^r\to X^r$ and $M^\mathcal{E}:(X')^{s'}\to (X')^{s'}$ are bounded.
\end{enumerate}
are equivalent.
\end{conjecture}
The implication \ref{it:mboundbfsthm1}$\Rightarrow$\ref{it:mboundbfsthm2} follows from Proposition~\ref{prop:bfsrescaleddualequality} combined with Theorem~\ref{thm:mboundnos} and, hence, the conjecture is that \ref{it:mboundbfsthm2}$\Rightarrow$\ref{it:mboundbfsthm1}. We revisit these conjectures in Section~\ref{sec:sparsedom} below.

\section{Examples of quasi-Banach function spaces and mapping properties of the maximal operator}\label{sec:weightedbfsmbound}
In this section we provide several examples of quasi-Banach function spaces $X$, their $(r,s)$-rescaled Banach function space $X_{r,s}$, and boundedness of the maximal operator on these spaces.

Examples of Banach function spaces include Lorentz, Orlicz, and Lebesgue spaces. We moreover study the bounds we obtain on weighted variable Lebesgue spaces and weighted Morrey spaces in order to compare our extrapolation result to the ones that exist in the literature. For boundedness of the maximal operator on Musielak-Orlicz spaces under general conditions on the Young function we refer the reader to \cite{Ha15}.

Since there is a general interest for spaces that have inherent weights, we also introduce some notation regarding this. Given a weight $u$ and a quasi-Banach function space $X$, we define the  space $X(u)$ as the space of those $f\in L^0(\Omega)$ with
\[
\|f\|_{X(u)}:=\|fu\|_X<\infty.
\]
\begin{proposition}\label{prop:bfsu}
Let $X$ be a Banach function space over $\Omega$ and let $u$ be a weight. Then $X(u)$ is again a Banach function space and its K\"othe dual is given by
\[
(X(u))'=X'(u^{-1}).
\]
\end{proposition}
\begin{proof}
To see that $X(u)$ has the ideal property and the Fatou property, note that this follows from the fact that $X$ is isometrically isomorphic to $X(u)$ through the positive map $f\mapsto fu^{-1}$. For the saturation property, let $\rho>0$ be a weak order unit in $X$. Then $\rho u^{-1}$ is a weak order unit in $X(u)$. Hence, saturation follows from Proposition~\ref{prop:weakorderunit}.

For the second assertion, we compute
\begin{align*}
\|g\|_{(X(u))'}
&=\sup_{\|f\|_{X(u)}=1}\int_\Omega\!|fg|\,\mathrm{d}x
=\sup_{\|h\|_X=1}\int_\Omega\!|h||g|u^{-1}\,\mathrm{d}x\\
&=\|gu^{-1}\|_{X'}=\|g\|_{X'(u^{-1})},
\end{align*}
proving the result.
\end{proof}

A second way of introducing a weight to a space is as follows. Given a $\sigma$-finite measure space  $(\Omega,\mu)$ and a weight $v$, we can define a new measure on $\Omega$ through
\[
v(E):=\int_E\!v\,\mathrm{d}\mu.
\]
Noting that $(\Omega,v)$ is again $\sigma$-finite, we can now consider a space $X$ that is both a quasi-Banach function space over $(\Omega,\mu)$, as well as over $(\Omega,v)$. In this case we denote its K\"othe dual with respect to $(\Omega,\mu)$ by $X'$, and its K\"othe dual with respect to $(\Omega,v)$ as $X^\dag$. In this case we have the following relation:
\begin{proposition}\label{prop:bfsv}
Let $X$ be a Banach function space over both $(\Omega,\mu)$ and $(\Omega,v)$. Then
\[
X'=X^\dag(v^{-1}).
\]
In particular, an operator $T$ is bounded $T:X'\to X'$ if and only if the operator
\[
T_vg:=T(gv)v^{-1}
\]
is bounded $T_v:X^\dag\to X^\dag$, with
\[
\|T\|_{X'\to X'}=\|T_v\|_{X^\dag\to X^\dag}.
\]
\end{proposition}
\begin{proof}
We have
\[
\|g\|_{X^\dag(v^{-1})}=\Big\|\frac{g}{v}\Big\|_{X^\dag}=\sup_{\|f\|_{X}=1}\int_{\Omega}\!|f|\Big|\frac{g}{v}\Big|v\,\mathrm{d}\mu=\sup_{\|f\|_{X}=1}\int_{\Omega}\!|f||g|\,\mathrm{d}\mu=\|g\|_{X'},
\]
as desired. For the second assertion, note that by the previous identity we have
\begin{align*}
\|T_v\|_{X^\dag\to X^\dag}&
=\sup_{\|g\|_{X^\dag}=1}\|(T(gv)v^{-1}\|_{X^\dag}\\
&=\sup_{\|gv\|_{X'}=1}\|(T(gv)\|_{X'}\\
&=\sup_{\|h\|_{X'}=1}\|Th\|_{X'}\\
&=\|T\|_{X'\to X'}.
\end{align*}
This proves the result.
\end{proof}

Finally, for computing the $(r,s)$-rescaled Banach function space of an $r$-convex and $s$-concave quasi-Banach function space $X$, it is useful to introduce some notation. Given $0<r<s\leq\infty$, note that the interval $[\frac{1}{s},\frac{1}{r}]$ is mapped to $[0,1]$ through the affine transformation $\phi$ given by
\[
\phi(t):=\frac{t-\frac{1}{s}}{\frac{1}{r}-\frac{1}{s}}.
\]
For $p\in[r,s]$, we then set
\[
\frac{1}{p_{r,s}}:=\phi\Big(\frac{1}{p}\Big)=\frac{\frac{1}{p}-\frac{1}{s}}{\frac{1}{r}-\frac{1}{s}},\quad \frac{1}{p_{r,s}'}=1-\phi\Big(\frac{1}{p}\Big)=\frac{\frac{1}{r}-\frac{1}{p}}{\frac{1}{r}-\frac{1}{s}}.
\]
\subsection{Lorentz spaces}\label{subsec:lorentz}
Let $p,q\in(0,\infty]$ and let $v$ be a weight. Then the Lorentz space $ L_v^{p,q}(\Omega)$ is defined as the space of functions $f\in L^0(\Omega)$ for which
\[
\|f\|_{L^{p,q}_v(\Omega)}:=p^{\frac{1}{q}}\big\|\lambda \big(v^p(\{x\in\Omega:|f(x)|>\lambda\})\big)^{\frac{1}{p}}\big\|_{L^q((0,\infty);\frac{\mathrm{d}\lambda}{\lambda})}<\infty.
\]
The space $X=L^{p,q}_v(\Omega)=L^{p,q}(\Omega,v^p)$ is (possibly after switching to an equivalent norm) a Banach function space over both $(\Omega,\mu)$ and $(\Omega,v^p)$ for $p,q\in(1,\infty)$, and in this case we can compute its K\"othe dual $X'$ through the duality of Lorentz spaces in the measure space $(\Omega,v^p)$. Using Proposition~\ref{prop:bfsv}, this yields
\[
X'=[L^{p',q'}(\Omega,v^p)](v^{-p})=[L^{p',q'}_{v^{p-1}}(\Omega)](v^{-p}).
\]
For $r\in(0,\infty)$, we have $X^r=L^{\frac{p}{r},\frac{q}{r}}_{v^r}(\Omega)$. If $s\in(0,\infty]$ with $r<s$, then the space $X=L^{p,q}_v(\Omega)$ is $r$-convex when $p,q> r$, and is $s$-concave if $p,q<s$. In this case we have
\[
X_{r,s}=\big[L^{p_{r,s},q_{r,s}}_{v^{\frac{p}{p_{r,s}}}}(\Omega)\big](v^{\frac{1}{\frac{1}{r}-\frac{1}{s}}-\frac{p}{p_{r,s}}}),
\quad (X_{r,s})'=
\big[L^{p_{r,s}',q_{r,s}'}_{v^{\frac{p}{p_{r,s}'}}}(\Omega)\big](v^{-\frac{1}{\frac{1}{r}-\frac{1}{s}}-\frac{p}{p_{r,s}'}})
\]
Note that if $p=q$, then $X=L^p_v(\Omega)$ and
\[
X_{r,s}=L^{p_{r,s}}_{v^{\frac{1}{\frac{1}{r}-\frac{1}{s}}}}(\Omega),\quad
(X_{r,s})'=L^{p_{r,s}'}_{v^{-\frac{1}{\frac{1}{r}-\frac{1}{s}}}}(\Omega).
\]
Specializing to the case $\Omega=\R^d$ with the basis of cubes, we find that $M:X_{r,s}\to X_{r,s}$ and $M:(X_{r,s})'\to (X_{r,s})'$ is equivalent to
\[
v^{\frac{p_{r,s}}{\frac{1}{r}-\frac{1}{s}}}\in A_{p_{r,s}},
\]
which is equivalent to $v\in A_{p,(r,s)}$.

Fortunately in the general case, we still get these bounds:
\begin{theorem}\label{thm:mrsboundlorentz}
Suppose $\Omega=\R^d$ with the basis of cubes. Let $r\in(0,\infty)$, $s\in(0,\infty]$ with $r<s$, and $p,q\in(r,s)$. Then if $X=L^{p,q}_v(\R^d)$ for a weight $v\in A_{p,(r,s)}$, we have that
\[
M:X_{r,s}\to X_{r,s},\quad M:(X_{r,s})'\to (X_{r,s})'
\]
are bounded.
\end{theorem}
\begin{proof}
In \cite[Example~2.46]{CMM22} it is shown through Boyd's interpolation theorem that if $p,q\in(1,\infty)$, $(uv)^p\in A_p$ and $v^p\in A_\infty$, then
\begin{equation}\label{eq:bfslorentzaprs1}
M:[L^{p,q}_v(\R^d)](u)\to M:[L^{p,q}_v(\R^d)](u)
\end{equation}
and
\[
M_{v^p}:[L^{p,q}(\R^d,v^p)^\dag](u^{-1})\to M:[L^{p,q}(\R^d,v^p)^\dag](u^{-1})
\]
are bounded. Note that by Proposition~\ref{prop:bfsv}, the second bound is equivalent to the bound
\begin{equation}\label{eq:bfslorentzaprs2}
M:[L^{p',q'}_{v^{p-1}}(\Omega)](v^{-p}u^{-1})\to [L^{p',q'}_{v^{p-1}}(\Omega)](v^{-p}u^{-1}).
\end{equation}

We apply this result with $p$, $q$ replaced by $p_{r,s},q_{r,s}\in(1,\infty)$ respectively, and with $v$ replaced by $v^{\frac{p}{p_{r,s}}}$ and $u:=v^{\frac{1}{\frac{1}{r}-\frac{1}{s}}-\frac{p}{p_{r,s}}}$. Note that now \eqref{eq:bfslorentzaprs1} becomes $M:X_{r,s}\to X_{r,s}$ and \eqref{eq:bfslorentzaprs2} becomes $M:(X_{r,s})'\to (X_{r,s})'$. The conditions on the weight for these bounds to hold now become
\[
v^{\frac{p_{r,s}}{\frac{1}{r}-\frac{1}{s}}}=\big(v^{\frac{p}{p_{r,s}}}u\big)^{p_{r,s}}\in A_{p_{r,s}},
\]
which is equivalent to $v\in A_{p,(r,s)}$, and
\[
v^p=(v^{\frac{p}{p_{r,s}}})^{p_{r,s}}\in A_\infty,
\]
which is true, since $v\in A_{p,(r,s)}$ implies that $v^p\in A_{\frac{p}{r}}\subseteq A_\infty$ by Proposition~\ref{prop:jnreverseholder}. The assertion follows.
\end{proof}

We refer the reader to \cite{Ma04} for further convexity and concavity results of Lorentz spaces and their generalizations.
\subsection{Variable Lebesgue spaces}\label{subsec:variablelebesgue}
Let $(\Omega,|\cdot|)$ be a $\sigma$-finite measure space and let $v$ be a weight. For a measurable function $p:\Omega\to(0,\infty]$ we set
\[
\Omega_\infty:=\{x\in\Omega:p(x)=\infty\}
\]
and define the weighted \emph{variable Lebesgue space} $L_v^{p(\cdot)}(\Omega)$ as the space of functions $f\in L^0(\Omega)$ for which
\[
\|f\|_{L_v^{p(\cdot)}(\Omega)}=\inf\Big\{\lambda>0:\int_{\Omega\backslash\Omega_\infty}\!\left(\frac{|f(x)|v(x)}{\lambda}\right)^{p(x)}\,\mathrm{d}\mu(x)+\|fv\|_{L^\infty(\Omega_\infty)}\leq 1\Big\}.
\]
This is a quasi-Banach function space and, when $p(x)\geq 1$ for a.e. $x\in\Omega$, a Banach function space. In the latter case, the K\"othe dual is given by
\[
(L_v^{p(\cdot)}(\Omega))'=L_{v^{-1}}^{p'(\cdot)}(\Omega),
\]
where $p'(\cdot)$ is defined through
\[
\frac{1}{p'(x)}=1-\frac{1}{p(x)}.
\]
For $r\in(0,\infty)$, the $r$-concavification of $L_v^{p(\cdot)}(\Omega)$ is given by
\[
(L_v^{p(\cdot)}(\Omega))^r=L_{v^r}^{\frac{p(\cdot)}{r}}(\Omega).
\]
In conclusion, $L_v^{p(\cdot)}(\Omega)$ is $r$-convex and $s$-concave precisely when $p(x)\in[r,s]$ for a.e. $x\in\Omega$, and in this case we have
\[
X_{r,s}=L^{p_{r,s}(\cdot)}_{v^{\frac{1}{\frac{1}{r}-\frac{1}{s}}}}(\Omega),\quad (X_{r,s})'=L^{p'_{r,s}(\cdot)}_{v^{-\frac{1}{\frac{1}{r}-\frac{1}{s}}}}(\Omega)
\]
where
\[
\frac{1}{p_{r,s}(x)}=\phi\Big(\frac{1}{p(x)}\Big)
=\frac{\frac{1}{p(x)}-\frac{1}{s}}{\frac{1}{r}-\frac{1}{s}},\quad \frac{1}{p'_{r,s}(x)}=1-\frac{1}{p_{r,s}(x)}=\frac{\frac{1}{r}-\frac{1}{p(x)}}{\frac{1}{r}-\frac{1}{s}}.
\]
\begin{definition}
Let $\mathcal{E}$ be a basis of sets in $\Omega$. We say that $v\in A_{p(\cdot),(r,s)}(\mathcal{E})$ when
\[
\sup_{E\in\mathcal{E}}|E|^{-\big(\frac{1}{r}-\frac{1}{s}\big)}\|v\ind_Q\|_{L^{\frac{1}{\frac{1}{p(\cdot)}-\frac{1}{s}}}(\Omega)}\|v^{-1}\ind_Q\|_{L^{\frac{1}{\frac{1}{r}-\frac{1}{p(\cdot)}}}(\Omega)}<\infty.
\]
\end{definition}
Exactly as in the case where $p(\cdot)$ is constant, we have $v\in A_{p(\cdot),(r,s)}(\mathcal{E})$ if and only if
\[
v^{\frac{1}{\frac{1}{r}-\frac{1}{s}}}\in A_{p_{r,s}(\cdot),(1,\infty)}(\mathcal{E}).
\]
We now specialize to the case $\Omega=\R^d$ with the basis of cubes.
\begin{definition}
Let $p:\R^d\to[1,\infty)$ be a measurable function with $\|p\|_{L^\infty(\Omega)}<\infty$. We say that $p(\cdot)\in LH_0$ when there is a $C>0$ such that for a.e. $x,y\in\R^d$ with $|x-y|<\frac{1}{2}$ we have
\[
\left|\frac{1}{p(x)}-\frac{1}{p(y)}\right|\leq \frac{C}{-\log(|x-y|)}.
\]
Moreover, we say that $p(\cdot)\in LH_\infty$ if there is a $C>0$ and a $p_\infty\in[1,\infty)$ such that
\[
\left|\frac{1}{p(x)}-\frac{1}{p_\infty}\right|\leq\frac{C}{\log(e+|x|)}
\]
for a.e. $x\in\R^d$.
\end{definition}
We now have the following result:
\begin{theorem}\label{thm:mrsboundvariablelebesgue}
Suppose $\Omega=\R^d$ with the basis of cubes. Let $r\in(0,\infty)$, $s\in(0,\infty]$ with $r<s$, and $p:\Omega\to(r,s)$ a measurable function with the following properties:
\begin{itemize}
\item $r<\essinf p\leq\esssup p<s$;
\item $p_{r,s}\in LH_0\cap LH_\infty$.
\end{itemize}
Then if $X=L^{p(\cdot)}_v(\R^d)$ for a weight $v\in A_{p(\cdot),(r,s)}$, we have that
\[
M:X_{r,s}\to X_{r,s},\quad M:(X_{r,s})'\to (X_{r,s})'
\]
are bounded.
\end{theorem}
\begin{proof}
In \cite[Theorem~1.5]{CFN12} it was shown that if $p(\cdot):\R^d\to(1,\infty)$ satisfies $1<\essinf p\leq\esssup p<\infty$ and $p(\cdot)\in LH_0\cap LH_\infty$, then for any $v\in A_{p(\cdot),(1,\infty)}$ we have that
\[
M:L^{p(\cdot)}_v(\R^d)\to L^{p(\cdot)}_v(\R^d)
\]
is bounded.

Replacing $p(\cdot)$ by $p_{r,s}(\cdot)$ and $v$ by $v^{\frac{1}{\frac{1}{r}-\frac{1}{s}}}\in A_{p_{r,s}(\cdot),(1,\infty)}$, we conclude that $M$ is bounded on $L^{p_{r,s}(\cdot)}_{v^{\frac{1}{\frac{1}{r}-\frac{1}{s}}}}(\Omega)=X_{r,s}$.

For the dual result, note that also $p'_{r,s}\in LH_0\cap LH_\infty$ and $v^{-\frac{1}{\frac{1}{r}-\frac{1}{s}}}\in A_{p'_{r,s}(\cdot),(1,\infty)}$, which implies that $M$ is bounded on $L^{p'_{r,s}(\cdot)}_{v^{-\frac{1}{\frac{1}{r}-\frac{1}{s}}}}(\Omega)=(X_{r,s})'$. The assertion follows.
\end{proof}
The conditions for $p(\cdot)$ are sufficient, but not necessary, whereas the condition on the weight is necessary \cite{CFN12}, or see Lemma~\ref{lem:indE}. We refer the reader to \cite{DHHR17} for an overview.

\subsection{Morrey spaces}\label{subsec:morreyspaces}
We restrict ourselves to the case $\Omega=\R^d$ with the basis of cubes $\mathcal{Q}$. Let $p\in(0,\infty)$, $q\in[p,\infty)$, and let $v$ be a weight. Then the weighted Morrey space $\mathcal{L}^{p,q}_v(\R^d)$ is defined as the space of those $f\in L^0(\R^d)$ with
\[
\|f\|_{\mathcal{L}^{p,q}_v(\R^d)}:=\sup_Q |Q|^{\frac{1}{q}}\langle fv\rangle_{p,Q}=\sup_Q \left(\frac{1}{|Q|^{1-\frac{p}{q}}}\int_Q\!|fv|^p\,\mathrm{d}x\right)^{\frac{1}{p}}<\infty.
\]
This is a quasi-Banach function space and, as was established in \cite{ST15}, if $p\geq 1$, a Banach function space. As a matter of fact, they show that this space has the ideal property and the Fatou property, but fails the condition that $\ind_E\in(\mathcal{L}^{p,q}(\R^d))'$ for all sets $E\subseteq\R^d$ of finite measure. Fortunately, however, it does satisfy the property that $\ind_Q\in\mathcal{L}^{p,q}(\R^d)$ for all cubes $Q$ and, thus, it satisfies the saturation property, meaning it is a Banach function space in our sense.

Several characterizations of the K\"othe dual space are given in the literature at various levels of generality, e.g., see \cite{AX12, ST15, MST18, ZZ22}. If $p\in(1,\infty)$, $q\in[p,\infty)$, then we have
\[
(\mathcal{L}^{p,q}(\R^d))'=\mathcal{B}^{p',q'}(\R^d),
\]
where $\mathcal{B}^{p',q'}(\R^d)$ is a block space. To define this space, we say that a function $b\in L^0(\R^d)$ is a $(p',q')$-block if there is a cube $Q\in\mathcal{Q}$ such that $\supp b\subseteq Q$ and
\[
|Q|^{\frac{1}{q'}}\langle b\rangle_{p',Q}\leq 1.
\]
We say that $f\in\mathcal{B}^{p',q'}(\R^d)$ if there exists a sequence $(\lambda_n)_{n\in\N}\in\ell^1(\N)$ and a sequence $(b_n)_{n\in\N}$ of $(p',q')$-blocks such that
\[
f=\sum_{n\in\N}\lambda_nb_n.
\]
Moreover, we set
\[
\|f\|_{\mathcal{B}^{p',q'}(\R^d)}:=\inf\|(\lambda_n)_{n\in\N}\|_{\ell^1(\N)},
\]
where the infimum runs over all possible representations $f=\sum_{n\in\N}\lambda_nb_n$.

In \cite[Proposition~5.2]{DR20} it was shown that if $p\geq 1$ and the Hardy-Littlewood maximal operator $M$ satisfies the boundedness
\[
M:\mathcal{L}^{p,q}_v(\R^d)\to\mathcal{L}^{p,q}_v(\R^d),
\]
then we have $v^p\in A_{\frac{p}{q'}+1}$. Conversely, their extrapolation theorem \cite[Theorem~1.1]{DR20} establishes that $M$ satisfies the above bounds for the weights $v^p\in A_{\frac{p}{q'}+1}$ of the form
\[
v(x)=|x|^{\alpha d} w(x),
\]
where $w^p\in A_p\cap RH_{\frac{q}{p}}$, $\alpha\in[0,\frac{1}{p}-\frac{1}{q})$. We note that by Proposition~\ref{prop:jnreverseholder} this condition on $w$ is equivalent to
\[
w\in A_{p,(1,\frac{1}{\frac{1}{p}-\frac{1}{q}})}=A_{q,(\frac{1}{\frac{1}{q'}+\frac{1}{p}},\infty)}.
\]
In particular, we have
\begin{equation}\label{eq:mboundsweightedmorrey}
M:\mathcal{L}^{p,q}_v(\R^d)\to\mathcal{L}^{p,q}_v(\R^d)
\end{equation}
whenever $v(x)=|x|^{\beta d}$ with $\beta\in(-\frac{1}{q},\frac{1}{q'})$. We note that this coincides exactly with the condition $v^q\in A_q$ for this weight, and hence, the conjecture is that \eqref{eq:mboundsweightedmorrey} holds for all $v^q\in A_q$.

Since $\mathcal{L}_v^{p,q}(\R^d)=[\mathcal{L}^{p,q}(\R^d)](v)$, it follows from Proposition~\ref{prop:bfsu} that
\[
(\mathcal{L}_v^{p,q}(\R^d))'=[\mathcal{B}^{p',q'}(\R^d)](v^{-1})=:\mathcal{B}_{v^{-1}}^{p',q'}(\R^d).
\]
To the best of our knowledge, there are no known bounds for $M$ in weighted Block spaces. Hence, we prove some here.

\begin{theorem}\label{thm:mboundblockspaces}
Let $p\in(1,\infty)$, $q\in(1,p]$, $r\in[1,q)$ and $v^p\in A_{p(\frac{1}{r}-\frac{1}{q})+1}$. Then
\[
M_r:\mathcal{B}_{v}^{p,q}(\R^d)\to \mathcal{B}_{v}^{p,q}(\R^d)
\]
is bounded.
\end{theorem}
We need the following lemma:
\begin{lemma}\label{lem:morreyblock}
Let $b$ be a $(p,q)$-block and $v^p\in A_{p(\frac{1}{r}-\frac{1}{q})+1}$. Then
\[
\|M_r(bv^{-1})\|_{\mathcal{B}^{p,q}_v(\R^d)}\lesssim_{d,p,q,r,v}1.
\]
\end{lemma}
\begin{proof}
Suppose $Q$ is a cube so that $\supp b\subseteq Q$ and
\[
|Q|^{\frac{1}{q}}\langle b\rangle_{p,Q}\leq 1.
\]
Define $m_0:=\ind_{2Q}M_r(bv^{-1})$ and $m_k:=\ind_{2^{k+1}Q\backslash 2^k Q}M_r(bv^{-1})$ for $k\geq 1$ so that
\[
M_r(bv^{-1})=\sum_{k=0}^\infty m_k.
\]
We claim that there is a constant $C_{d,p,q,r,v}>0$ and an $t<\frac{1}{r}(\frac{p}{q'}+1)$ depending on $d$, $v$, $p$, $q$, $r$ such that
\[
C_{d,p,q,r,v} 2^{-(k+1)d(\frac{t-1}{p}-\frac{1}{q'})}m_k v
\]
is a $(p,q)$-block with respect to the cube $2^{k+1}Q$ for all integers $k\geq 0$. 

Indeed, for $k=0$ we have
\begin{align*}
\|m_0v\|_{L^p(\R^d)}&\leq\|M(bv^{-1})\|_{L^p_v(\R^d)}\leq\|M\|_{L^p_v(\R^d)\to L^p_v(\R^d)}\|bv^{-1}\|_{L^p_v(\R^d)}\\
&=\|M\|_{L^p_v(\R^d)\to L^p_v(\R^d)}\|b\|_{L^p(Q)}\leq\|M\|_{L^p_v(\R^d)\to L^p_v(\R^d)}|Q|^{\frac{1}{p}-\frac{1}{q}}\\
&=2^{\frac{d}{q}-\frac{d}{p}}\|M\|_{L^p_v(\R^d)\to L^p_v(\R^d)}|2Q|^{\frac{1}{p}-\frac{1}{q}}.
\end{align*}
We note here that $v^p\in A_{p(\frac{1}{r}-\frac{1}{q})+1}\subseteq A_{p(\frac{1}{r}-\frac{1}{p})+1}=A_{\frac{p}{r}}$ so that $\|M_r\|_{L^p_v(\R^d)\to L^p_v(\R^d)}<\infty$. 

For $k\geq 1$, if $x\in2^{k+1}Q\backslash 2^k Q$, since for any $Q'$ containing $x$ with $Q'\cap Q\neq\emptyset$ we must have at least $\ell(Q')\gtrsim 2^k\ell(Q)$, we have
\begin{equation}\label{eq:mboundblock1}
\begin{split}
M_r(bv^{-1})(x)
&\lesssim\left(\frac{1}{2^{kd}|Q|}\int_Q\!b^r v^{-r}\,\mathrm{d}x\right)^{\frac{1}{r}}
\leq 2^{-\frac{kd}{r}}|Q|^{-\frac{1}{r}}\| b\|_{L^p(Q)}\|v^{-1}\|_{L^{\frac{1}{\frac{1}{r}-\frac{1}{p}}}(Q)}\\
&\leq 2^{\frac{d}{r}}|2^{k+1}Q|^{-\frac{1}{r}}|Q|^{\frac{1}{p}-\frac{1}{q}}\|v^{-1}\|_{L^{\frac{1}{\frac{1}{r}-\frac{1}{p}}}(Q)}.
\end{split}
\end{equation}
Moreover, by self-improvement of the Muckenhoupt classes, we can pick $t<p(\frac{1}{r}-\frac{1}{q})+1$ depending on $d$, $v$, $p$, $q$, $r$ such that $v^p\in A_t$. Hence, using \cite[Proposition~7.1.5(8)]{Gr14a} with the cube $2^{k+1}Q$ and $f=\ind_Q$, we obtain
\[
\|v\|_{L^p(2^{k+1}Q)}\leq [v]_{A_t}^{\frac{1}{p}} 2^{(k+1)d\frac{t}{p}}\|v\|_{L^p(Q)}.
\]

Thus, combining this with \eqref{eq:mboundblock1} yields
\begin{align*}
\|m_kv\|_{L^p(\R^d)}
&\lesssim_{d,r}  |2^{k+1}Q|^{-\frac{1}{r}}|Q|^{\frac{1}{p}-\frac{1}{q}}\|v^{-1}\|_{L^{\frac{1}{\frac{1}{r}-\frac{1}{p}}}(Q)}\|v\|_{L^p(2^{k+1}Q)}\\
&\lesssim_{p,q,v} 2^{(k+1)d(\frac{t-1}{p}-(\frac{1}{r}-\frac{1}{q}))}|2^{k+1}Q|^{\frac{1}{p}-\frac{1}{q}}|Q|^{-\frac{1}{r}}\|v\|_{L^p(Q)}\|v^{-1}\|_{L^{\frac{1}{\frac{1}{r}-\frac{1}{p}}}(Q)}\\
&\leq[v]_{p,(r,\infty)}2^{(k+1)d(\frac{t-1}{p}-(\frac{1}{r}-\frac{1}{q}))}|2^{k+1}Q|^{\frac{1}{p}-\frac{1}{q}}
\end{align*}
as desired.

Since $t<p(\frac{1}{r}-\frac{1}{q})+1$ we have $\frac{t-1}{p}-(\frac{1}{r}-\frac{1}{q})<0$ and hence
\[
\sum_{k=1}^\infty2^{(k+1)d(\frac{t-1}{p}-(\frac{1}{r}-\frac{1}{q}))}\lesssim_{d,p,q,r,v}1,
\]
we conclude that $M_r(bv^{-1})v\in\mathcal{B}^{p,q}(\R^d)$ with
\[
\|M_r(bv^{-1})\|_{\mathcal{B}_{v}^{p,q}(\R^d)}\lesssim_{d,p,q,r,v} 1,
\]
as asserted.
\end{proof}
\begin{proof}[Proof of Theorem~\ref{thm:mboundblockspaces}]
Let $f\in\mathcal{B}^{p,q}_v(\R^d)$ and let $fv=\sum_{n\in\N} \lambda_n b_n$ be a representation. Then, by Lemma~\ref{lem:morreyblock}, we have
\[
\|M_rf\|_{\mathcal{B}^{p,q}_v(\R^d)}\leq \sum_{n\in\N}|\lambda_n|\|M_r(b_n v^{-1})\|_{\mathcal{B}^{p,q}_v(\R^d)}\lesssim_{d,p,q,r,v}\|(\lambda_n)_{n\in\N}\|_{\ell^1(\N)}.
\]
Thus, the result follows from taking an infimum over all possible representations of $f$.
\end{proof}

The Morrey space $\mathcal{L}^{p,q}_v(\R^d)$ is $r$-convex when $p\geq r$, with
\[
(\mathcal{L}^{p,q}_v(\R^d))^r=\mathcal{L}^{\frac{p}{r},\frac{q}{r}}_{v^r}(\R^d).
\]
However, when $p\neq q$, Morrey spaces are only $\infty$-concave and not $s$-concave for any $s<\infty$. This means that the Block space $\mathcal{B}^{p,q}_v(\R^d)$ is $1$-convex and not $r$-convex for any $r>1$. However, we do claim that for $r\in[1,p]$ we have
\[
\big[(\mathcal{B}^{p,q}_v(\R^d))^r\big]'=\mathcal{L}^{(\frac{p}{r})',(\frac{q}{r})'}_{v^{-r}}(\R^d).
\]
Hence, if the claim is true, for $p\in[r,s]$, $q\in[p,s)$, we have
\begin{equation}\label{eq:xrsblockmorrey}
X_{r,s}=\mathcal{L}^{p_{r,s},q_{r,s}}_{v^{\frac{1}{\frac{1}{r}-\frac{1}{s}}}}(\R^d),\quad (X_{r,s})'=\mathcal{B}^{p'_{r,s},q'_{r,s}}_{v^{-\frac{1}{\frac{1}{r}-\frac{1}{s}}}}(\R^d).
\end{equation}
We now prove the claim:
\begin{proposition}\label{prop:dualofpowerofblock}
Let $p\in[1,\infty)$, $q\in[p,\infty)$, $r\in[1,p]$, and $v$ a weight. Then
\[
\big[(\mathcal{B}^{p,q}_v(\R^d))^r\big]'=\mathcal{L}^{(\frac{p}{r})',(\frac{q}{r})'}_{v^{-r}}(\R^d)
\]
with equal norm.
\end{proposition}
\begin{proof}
Noting that $\mathcal{L}^{(\frac{p}{r})',(\frac{q}{r})'}_{v^{-r}}(\R^d)=\mathcal{L}^{(\frac{p}{r})',(\frac{q}{r})'}(\R^d)(v^{-r})$ and
\[
\big[(\mathcal{B}^{p,q}_v(\R^d))^r\big]'=\big[(\mathcal{B}^{p,q}(\R^d))^r(v^r)\big]'=\big[(\mathcal{B}^{p,q}(\R^d))^r\big]'(v^{-r}),
\]
we have reduced the problem to the case $v=1$.

Let $f\in\big[(\mathcal{B}^{p,q}_v(\R^d))^r\big]'$, let $Q$ be a cube, and let $g\in L^{\frac{p}{r}}(Q)$ be non-zero. Then 
\[
b:=|Q|^{\frac{1}{p}-\frac{1}{q}}\frac{|g|^{\frac{1}{r}}}{\|g\|^{\frac{1}{r}}_{L^{\frac{p}{r}}(Q)}}
\]
is a $(p,q)$-block in $Q$, so $f|b|^r\in L^1(Q)$ with
\[
|Q|^{\frac{r}{p}-\frac{r}{q}}\|fg\|_{L^1(Q)}=\|f|b|^r\|_{L^1(Q)}\|g\|_{L^{\frac{p}{r}}(Q)}\leq\|f\|_{\big[(\mathcal{B}^{p,q}(\R^d))^r\big]'}\|g\|_{L^{\frac{p}{r}}(Q)}.
\]
Hence
\[
|Q|^{\frac{r}{p}-\frac{r}{q}}\|f\|_{L^{(\frac{p}{r})'}(Q)}=\sup_{\|g\|_{L^{\frac{p}{r}}(Q)}=1}|Q|^{\frac{r}{p}-\frac{r}{q}}\|fg\|_{L^1(Q)}\leq\|f\|_{\big[(\mathcal{B}^{p,q}(\R^d))^r\big]'}.
\]
Taking a supremum over all cubes $Q$ we conclude that $f\in\mathcal{L}^{(\frac{p}{r})',(\frac{q}{r})'}(\R^d)$ with
\[
\|f\|_{\mathcal{L}^{(\frac{p}{r})',(\frac{q}{r})'}(\R^d)}\leq\|f\|_{\big[(\mathcal{B}^{p,q}(\R^d))^r\big]'}.
\]

For the converse, suppose $f\in\mathcal{L}^{(\frac{p}{r})',(\frac{q}{r})'}(\R^d)$. Let $g\in(\mathcal{B}^{p,q}(\R^d))^r$. Then $|g|^{\frac{1}{r}}\in\mathcal{B}^{p,q}(\R^d)$, so there is a representation $|g|^{\frac{1}{r}}=\sum_{n\in\N}\lambda_n b_n$ where $(\lambda_n)_{n\in\N}\in\ell^1(\N)$ and $b_n$ is a $(p,q)$-block over a cube $Q_n$. Let $h\in L^{r'}(\R^d)$. Then, since
\[
\frac{1}{r(\frac{p}{r})'}+\frac{1}{p}+\frac{1}{r}=\frac{1}{r}-\frac{1}{p}+\frac{1}{p}+\frac{1}{r'}=1,
\]
it follows from H\"older's inequality that
\begin{align*}
\int_{\R^d}\!|fg|^{\frac{1}{r}}|h|\,\mathrm{d}x&\leq\sum_{n\in\N}|\lambda_n|\int_{Q_n}\!|f|^{\frac{1}{r}}|b_n||h|\,\mathrm{d}x\\
&\leq\sum_{n\in\N}|\lambda_n|\|f\|_{L^{(\frac{p}{r})'}(Q_n)}^{\frac{1}{r}}\|b_n\|_{L^p(Q_n)}\|h\|_{L^{r'}(\R^d)}\\
&\leq\sum_{n\in\N}|\lambda_n|\big(|Q|^{\frac{r}{p}-\frac{r}{q}}\|f\|_{L^{(\frac{p}{r})'}(Q_n)}\big)^{\frac{1}{r}}\|h\|_{L^{r'}(\R^d)}\\
&\leq\|(\lambda_n)_{n\in\N}\|_{\ell^1(\N)}\|f\|^{\frac{1}{r}}_{\mathcal{L}^{(\frac{p}{r})',(\frac{q}{r})'}(\R^d)}\|h\|_{L^{r'}(\R^d)}.
\end{align*}
Hence, we have
\begin{align*}
\|fg\|_{L^1(\R^d)}&=\||fg|^{\frac{1}{r}}\|_{L^r(\R^d)}^r
=\Big(\sup_{\|h\|_{L^{r'}(\R^d)}=1}\int_{\R^d}\!|fg|^{\frac{1}{r}}|h|\,\mathrm{d}x\Big)^r\\
&\leq\|(\lambda_n)_{n\in\N}\|^r_{\ell^1(\N)}\|f\|_{\mathcal{L}^{(\frac{p}{r})',(\frac{q}{r})'}(\R^d)}.
\end{align*}
Taking an infimum over all possible representations of $|g|^{\frac{1}{r}}$, we conclude that
\[
\|fg\|_{L^1(\R^d)}\leq \||g|^{\frac{1}{r}}\|^r_{\mathcal{B}^{p,q}(\R^d)}\|f\|_{\mathcal{L}^{(\frac{p}{r})',(\frac{q}{r})'}(\R^d)}=\|g\|_{(\mathcal{B}^{p,q}(\R^d))^r}\|f\|_{\mathcal{L}^{(\frac{p}{r})',(\frac{q}{r})'}(\R^d)}
\]
and thus, $f\in\big[(\mathcal{B}^{p,q}(\R^d))^r\big]'$ with
\[
\|f\|_{\big[(\mathcal{B}^{p,q}(\R^d))^r\big]'}\leq\|f\|_{\mathcal{L}^{(\frac{p}{r})',(\frac{q}{r})'}(\R^d)}.
\]
The assertion follows.
\end{proof}

\begin{theorem}\label{thm:mrsboundmorrey}
Let $r\in(0,\infty)$, $s\in(0,\infty]$ with $r<s$. Let $p\in(r,s)$, $q\in[p,s)$, and $v\in A_{p,(r,\frac{1}{\frac{1}{p}-\frac{1}{q}+\frac{1}{s}})}$, i.e.,
\[
\sup_Q|Q|^{-\big(\frac{1}{r}-\frac{1}{s}-\frac{1}{p}+\frac{1}{q}\big)}\|v\ind_Q\|_{L^{\frac{1}{\frac{1}{q}-\frac{1}{s}}}(\R^d)}\|v^{-1}\ind_Q\|_{L^{\frac{1}{\frac{1}{r}-\frac{1}{p}}}(\R^d)}<\infty.
\]
Then, if $X=\mathcal{L}^{p,q}_v(\R^d)$,
\[
M:X_{r,s}\to X_{r,s},\quad M:(X_{r,s})'\to (X_{r,s})'
\]
are bounded. In particular, this is the case if $v(x)=|x|^{\beta d}$ and $\beta\in(-(\frac{1}{q}-\frac{1}{s}),\frac{1}{r}-\frac{1}{p})$. 
\end{theorem}
\begin{proof}
Note that we computed $X_{r,s}$ and $(X_{r,s})'$ in \eqref{eq:xrsblockmorrey}. As discussed above, it follows from taking $\alpha=0$ in \cite[Theorem~1.1]{DR20} that the first bound is true for
\[
v^{\frac{1}{\frac{1}{r}-\frac{1}{s}}}\in A_{p_{r,s},(1,\frac{1}{\frac{1}{p_{r,s}}-\frac{1}{q_{r,s}}})},
\]
which is equivalent to $v\in A_{p,(r,\frac{1}{\frac{1}{p}-\frac{1}{q}+\frac{1}{s}})}$. 

For the dual bound, note that by Theorem~\ref{thm:mboundblockspaces} we require that 
\[
v^{-\frac{1}{\frac{1}{r}-\frac{1}{s}}p_{r,s}'}\in A_{\frac{p'_{r,s}}{q_{r,s}}+1}.
\]
This is equivalent to $v^{-\frac{1}{\frac{1}{r}-\frac{1}{s}}}\in A_{p_{r,s}',(\frac{1}{\frac{1}{p_{r,s}'}+\frac{1}{q_{r,s}}},\infty)}$, or $v^{\frac{1}{\frac{1}{r}-\frac{1}{s}}}\in A_{p_{r,s},(1,\frac{1}{\frac{1}{p_{r,s}}-\frac{1}{q_{r,s}}})}$, which is precisely the condition that we imposed on $v$. This proves the assertion.

For the assertion about power weights, we use the fact that $v=|x|^{\beta d}\in A_{p,(r,s)}$ if and only if $\beta\in(-(\frac{1}{p}-\frac{1}{s}),\frac{1}{r}-\frac{1}{p})$, and replace $\frac{1}{s}$ by $\frac{1}{p}-\frac{1}{q}+\frac{1}{s}$. This proves the result.
\end{proof}
\begin{remark}
If we could prove that Theorem~\ref{thm:mboundblockspaces} holds for all weights of the form $v(x)=|x|^{\alpha d} w(x)$ with $w^p\in A_{\frac{p}{q'}+1}$ and $\alpha\in(-(\frac{1}{p}-\frac{1}{q}),0]$, then we could get all power weights $v(x)=|x|^{\beta d}$ with $\beta\in (-(\frac{1}{q}-\frac{1}{s}),\frac{1}{r}-\frac{1}{q})$. In this case we recover the bounds from \cite{DR20}. We leave this as an interesting open problem.
\end{remark}

\section{Abstract extrapolation}\label{sec:proofs}
In this section we state and prove our main results in the on diagonal case.
\subsection{Main abstract theorem}
\begin{theorem}\label{thm:mainabs}
Let $(\Omega,|\cdot|)$ be a $\sigma$-finite measure space with a basis of sets $\mathcal{E}$. Let $r_1,r_2\in[1,\infty)$, $q_1,q_2\in(0,\infty]$, with $\frac{1}{q}=\frac{1}{q_1}+\frac{1}{q_2}\in(0,\infty)$. Let $X_1,X_2$ be $q$-convex quasi-Banach function spaces over $\Omega$ and suppose the following conditions hold:
\begin{itemize}
\item If $q_2\neq \infty$, $X_2$ is $q$-convex and there exists a positive isometric isomorphism
\[
L_1:X_1^q\to (X_2^q)',
\]
and
\[
M_{r_1}:X_1^q\to (X_2^q)'
\]
is bounded;
\item If $q_1\neq \infty$, $X_1$ is $q$-convex and there exists a positive isometric isomorphism
\[
L_2:X_2^q\to (X_1^q)',
\]
and
\[
M_{r_2}:X_2^q\to (X_1^q)'
\]
is bounded.
\end{itemize}
Then for every $f_1\in X_1$, $f_2\in X_2$ there exist weights $W_1,W_2$ such that
\begin{equation}\label{eq:mainabs1}
\sup_{E\in\mathcal{E}}\langle W_1^{-1}\rangle_{r_1q_2,E}\langle W_2^{-1}\rangle_{r_2q_1,E}
\leq 2^{\frac{1}{q}}\|M_{r_1}^\mathcal{E}\|^{\frac{1}{q_2}}_{X_1^q\to (X_2^q)'}\|M_{r_2}^\mathcal{E}\|^{\frac{1}{q_1}}_{X_2^q\to (X_1^q)'}
\end{equation}
and
\begin{equation}\label{eq:mainabs2}
\|f_1W_1\|_{L^{q_1}(\Omega)}\|f_2W_2\|_{L^{q_2}(\Omega)}\leq 2^{\frac{1}{q}}\|f_1\|_{X_1}\|f_2\|_{X_2}.
\end{equation}

Moreover, suppose $X$ is a $q$-convex quasi-Banach function space satisfying:
\begin{itemize}
\item If $q_2\neq \infty$,
\[
M_{r_1}:X^q\to X^q
\]
is bounded;
\item If $q_1\neq\infty$, 
\[
M_{r_2}:(X^q)'\to (X_q)'
\]
is bounded.
\end{itemize}
Then for each $f\in X$, $g\in[(X^q)']^{\frac{1}{q}}$ there is a weight $W$ such that
\begin{equation}\label{eq:mainabs3}
\sup_{E\in\mathcal{E}}\langle W\rangle_{r_2q_1,E}\langle W^{-1}\rangle_{r_1q_2,E}
\leq 2^{\frac{1}{q}}\|M_{r_1}^\mathcal{E}\|^{\frac{1}{q_2}}_{X^q\to X^q}\|M_{r_2}^\mathcal{E}\|^{\frac{1}{q_1}}_{(X^q)'\to (X^q)'}
\end{equation}
and
\begin{equation}\label{eq:mainabs4}
\|f_1W\|_{L^{q_1}(\Omega)}\|f_2W^{-1}\|_{L^{q_2}(\Omega)}\leq 2^{\frac{1}{q}}\|f_1\|_{X}\|f_2\|_{[(X^q)']^{\frac{1}{q}}}.
\end{equation}
\end{theorem}
\begin{proof}
If $f_1=0$ or $f_2=0$, we can set $W_1=W_2=1$, where \eqref{eq:mainabs1} now follows from Proposition~\ref{prop:maxopnorm}. If both $f_1$ and $f_2$ are non-zero, we define
\[
S_j:X_j^q\to X_j^q \quad S_jf:=L_j^{-1}M_{r_j}^\mathcal{E} f,
\]
and set
\[
R_jf_j:=\sum_{k=0}^\infty\frac{S_j^k(|f_j|^q)}{2^k\|S_j\|^k_{X_j^q\to X_j^q}}
\]
for $j\in\{1,2\}$.  Note that by Proposition~\ref{prop:nonzeromax} we have $R_jf_j>0$ in $\Omega$ and hence,
\[
W_1:=(R_1f_1)^{-\frac{1}{q_2}}(L_2R_2f_2)^{\frac{1}{q_1}},\quad
W_2:=(L_1R_1f_1)^{\frac{1}{q_2}}(R_2f_2)^{-\frac{1}{q_1}},
\]
are well-defined weights. Since $L_j$ is positive,
\[
M_{r_j}^\mathcal{E} R_jf_j=L_jS_jR_jf_j\leq2\|S_j\|_{X_j^q\to (X_j^q)'}L_jR_jf_j
\]
so that for any $E\in\mathcal{E}$ we have
\begin{align*}
\langle &W_1^{-1}\rangle_{r_1q_2,E}\langle W_2^{-1}\rangle_{r_2q_1,E}\\
&\leq
\Big(\langle R_1f_1\rangle_{r_1,E}\langle (L_1R_1f_1)^{-1}\rangle_{\infty,E}\Big)^{\frac{1}{q_2}}
\Big(\langle R_2f_2\rangle_{r_2,E}\langle (L_2R_2f_2)^{-1}\rangle_{\infty,E}\Big)^{\frac{1}{q_1}}\\
&\leq 2^{\frac{1}{q}}
\|M_{r_1}^\mathcal{E}\|^{\frac{1}{q_2}}_{X_1^q\to (X_2^q)'}\|M_{r_2}^\mathcal{E}\|^{\frac{1}{q_1}}_{X_2^q\to (X_1^q)'},
\end{align*}
proving \eqref{eq:mainabs1}.

Since $|f_1|^q\leq R_1f_1$, we have $|f_1|W_1\leq|f_1|^{\frac{q}{q_1}}(L^2R_2f_2)^{\frac{1}{q_1}}$ so that
\begin{align*}
\|f_1W_1\|_{L^{q_1}(\Omega)}&
\leq\||f_1|^qL_2R_2f_2\|_{L^1(\Omega)}^{\frac{1}{q_1}}
\leq\||f_1|^q\|^{\frac{1}{q_1}}_{X_1^q}\|L_2R_2f_2\|^{\frac{1}{q_1}}_{(X_1^q)'}\\
&\leq\|f_1\|_{X_1}^{\frac{q}{q_1}}\|R_2f_2\|_{X_2^q}^{\frac{1}{q_1}}
\leq 2^{\frac{1}{q_1}}\|f_1\|_{X_1}^{\frac{q}{q_1}}\||f_2|^q\|_{X_2^q}^{\frac{1}{q_1}}\\
&=2^{\frac{1}{q_1}}\|f_1\|_{X_1}^{\frac{q}{q_1}}\|f_2\|_{X_2}^{\frac{q}{q_1}}.
\end{align*}
A similar estimate holds for $\|f_2W_2\|_{L^{q_2}(\Omega)}$, proving \eqref{eq:mainabs2}.

For the second assertion, set $X_1:=X$ and $X_2=[(X^q)']^{\frac{1}{q}}$. Then $X_1^q=(X_2^q)'$, so we can take $L_1$ and $L_2$ the identity operators. Setting $W:=W_1$, we note that now $W_2=W^{-1}$. Hence, \eqref{eq:mainabs3} and \eqref{eq:mainabs4} follow from \eqref{eq:mainabs1} and \eqref{eq:mainabs2} respectively. This proves the assertion.
\end{proof}

\begin{remark}\label{rem:splitfunctions}
Note that instead of \eqref{eq:mainabs2}, we actually get the stronger statement that for $j\in\{1,2\}$ we have $f_j\in L^{q_j}_{W_j}(\Omega)$ with
\[
\|f_j\|_{L^{q_j}_{W_j}(\Omega)}\leq 2^{\frac{1}{q_j}}\|f_1\|_{X_1}^{\frac{q}{q_1}}\|f_2\|_{X_2}^{\frac{q}{q_2}}.
\]
\end{remark}

The formulation of the statement simplifies when we are considering a single Banach function space $X$ and $q=1$. This gives us the following rather elegant result:
\begin{corollary}\label{cor:mainabs}
Let $p\in[1,\infty]$, let $(\Omega,|\cdot|)$ be a $\sigma$-finite measure space with a basis of sets $\mathcal{E}$, and let $X$ be a Banach function space over $\Omega$. Then for every $f\in X$, $g\in X'$ there exists a weight $w\in A_p(\mathcal{E})$ such that
\[
[w]^\mathcal{E}_p\leq 2\|M^\mathcal{E}\|_{X\to X}^{\frac{1}{p'}}\|M^\mathcal{E}\|_{X'\to X'}^{\frac{1}{p}},
\]
provided that
\begin{itemize}
\item If $p\neq 1$, 
\[
M^\mathcal{E}:X\to X
\]
is bounded;
\item If $p\neq\infty$,
\[
M^\mathcal{E}:X'\to X'
\]
is bounded.
\end{itemize}
Moreover, this weight $w$ can be chosen such that $f\in L^p_w(\Omega)$, $g\in L^{p'}_{w^{-1}}(\Omega)$ with
\[
\|f\|_{L^p_w(\Omega)}\|g\|_{L^{p'}_{w^{-1}}(\Omega)}\leq 2\|f\|_X\|g\|_{X'}.
\]
\end{corollary}
\begin{proof}
This follows from setting $q=1$, $q_1=p$, $q_2=p'$, $X_1=X$, $X_2=X'$ in Theorem~\ref{thm:mainabs}.
\end{proof}

If one unwinds the proof of this corollary, one can explain the main ideas fairly concisely. Indeed, if $f\in X$ and $M^\mathcal{E}$ is bounded in $X$, then one can use the Rubio de Francia algorithm on $f$ to obtain a weight $w_1\in A_1(\mathcal{E})$ with
\[
[w_1^{-1}]_\infty^\mathcal{E}=[w_1]_1^\mathcal{E}\leq 2\|M\|_{X\to X}.
\]
One can then use the same procedure for $g\in X'$ to get a weight $w_2\in A_1(\mathcal{E})$ with $[w_2]_1^\mathcal{E}\leq 2\|M^\mathcal{E}\|_{X'\to X'}$. Then it follows from \eqref{eq:wfullrangesym} that $w:=w_1^{-\frac{1}{p'}}w_2^{\frac{1}{p}}\in A_p(\mathcal{E})$ with the desired properties.

At this point it becomes reasonable to try to use Proposition~\ref{prop:limitedrangeweightinterpolation} to see if this same procedure works in the limited range setting. As one can check, it turns out that this is the case, which yields the following corollary:
\begin{corollary}
Let $r\in(0,\infty)$, $s\in(0,\infty]$ with $r<s$, $\frac{1}{r}-\frac{1}{s}\geq 1$, and let $p\in[r,s]$. Let $(\Omega,|\cdot|)$ be a $\sigma$-finite measure space with a basis of sets $\mathcal{E}$. If $X$ is a Banach function space over $\Omega$ and $M_{\frac{1}{\frac{1}{r}-\frac{1}{s}}}^\mathcal{E}:X\to X$ is bounded when $p>r$ and $M_{\frac{1}{\frac{1}{r}-\frac{1}{s}}}^\mathcal{E}:X'\to X'$ is bounded when $p<s$, then for every $f\in X$, $g\in X'$ there exists a weight $w\in A_{p,(r,s)}(\mathcal{E})$ such that
\[
[w]^\mathcal{E}_{p,(r,s)}\leq 2\|M^\mathcal{E}\|_{X^{\frac{1}{\frac{1}{r}-\frac{1}{s}}}\to X^{\frac{1}{\frac{1}{r}-\frac{1}{s}}}}^{\frac{1}{r}-\frac{1}{p}}\|M^\mathcal{E}\|_{(X')^{\frac{1}{\frac{1}{r}-\frac{1}{s}}}\to (X')^{\frac{1}{\frac{1}{r}-\frac{1}{s}}}}^{\frac{1}{p}-\frac{1}{s}}
\]
and $f\in L^{\frac{\frac{1}{r}-\frac{1}{s}}{\frac{1}{p}-\frac{1}{s}}}_w(\Omega)$, $g\in L^{\frac{\frac{1}{r}-\frac{1}{s}}{\frac{1}{r}-\frac{1}{p}}}_{w^{-1}}(\Omega)$ with
\begin{equation}\label{eq:cormainabslimrange1}
\|f\|_{L^{\frac{\frac{1}{r}-\frac{1}{s}}{\frac{1}{p}-\frac{1}{s}}}_w(\Omega)}\|g\|_{L^{\frac{\frac{1}{r}-\frac{1}{s}}{\frac{1}{r}-\frac{1}{p}}}_{w^{-1}}(\Omega)}\leq 2\|f\|_X\|g\|_{X'}.
\end{equation}
\end{corollary}
\begin{proof}
We set $\frac{1}{q_1}=\frac{\frac{1}{p}-\frac{1}{s}}{\frac{1}{r}-\frac{1}{s}}$, $\frac{1}{q_2}=\frac{\frac{1}{r}-\frac{1}{p}}{\frac{1}{r}-\frac{1}{s}}$ so that $q=1$, and take $\frac{1}{r_1}=\frac{1}{r_2}=\frac{1}{r}-\frac{1}{s}$, $X_1=X$, $X_2=X'$. Then it follows from Theorem~\ref{thm:mainabs} that for $f\in X$, $g\in X'$ there is a weight $w$ such that
\begin{align*}
[w]_{p,(r,s)}^\mathcal{E}&
=\sup_{E\in\mathcal{E}}\langle w\rangle_{r_2q_1,E}\langle w^{-1}\rangle_{r_1q_2,E}\\
&\leq 2\|M_{\frac{1}{\frac{1}{r}-\frac{1}{s}}}^\mathcal{E}\|_{X\to X}^{\frac{\frac{1}{r}-\frac{1}{p}}{\frac{1}{r}-\frac{1}{s}}}\|M_{\frac{1}{\frac{1}{r}-\frac{1}{s}}}^\mathcal{E}\|_{X'\to X'}^{\frac{\frac{1}{p}-\frac{1}{s}}{\frac{1}{r}-\frac{1}{s}}}\\
&=2\|M^\mathcal{E}\|_{X^{\frac{1}{\frac{1}{r}-\frac{1}{s}}}\to X^{\frac{1}{\frac{1}{r}-\frac{1}{s}}}}^{\frac{1}{r}-\frac{1}{p}}\|M^\mathcal{E}\|_{(X')^{\frac{1}{\frac{1}{r}-\frac{1}{s}}}\to (X')^{\frac{1}{\frac{1}{r}-\frac{1}{s}}}}^{\frac{1}{p}-\frac{1}{s}}
\end{align*}
by Proposition~\ref{prop:mqboundrescale}, and \eqref{eq:cormainabslimrange1} holds. This proves the result.
\end{proof}

Now, if one wants to show that an operator $T$ satisfying bounds in the limited weight classes $A_{p,(r,s)}(\mathcal{E})$ is bounded in an $r$-convex and $s$-concave quasi-Banach function space $X$, then, since $X$ might not be a Banach function space, one might first have to pass to the $r$-concavification $X^r$ before one can use a dualization argument. Indeed, we have
\begin{align*}
\|Tf\|_X
&=\||Tf|^r\|_{X^r}^{\frac{1}{r}}
=\sup_{\|g^r\|_{(X^r)'}=1}\||Tf|^r |g|^r\|_{L^1(\Omega)}^{\frac{1}{r}}\\
&=\sup_{\|g\|_{[(X^r)']^{\frac{1}{r}}}=1}\|(Tf)g\|_{L^r(\Omega)},
\end{align*}
Thus, to bound such an operator $T$, we need the following special corollary of Theorem~\ref{thm:mainabs} combined with the factorization result Corollary~\ref{cor:factorization}:
\begin{corollary}\label{cor:abstractaprsextrapolation}
Let $r\in(0,\infty)$, $s\in(0,\infty]$ with $r<s$ and let $p\in[r,s]$. Let $(\Omega,|\cdot|)$ be a $\sigma$-finite measure space with a basis of sets $\mathcal{E}$. Let $X$ be an $r$-convex and $s$-concave quasi-Banach function space over $\Omega$ satisfying:
\begin{itemize}
\item If $p\neq r$,
\[
M^\mathcal{E}:X_{r,s}\to X_{r,s}
\]
is bounded;
\item If $p\neq s$, 
\[
M^\mathcal{E}:(X_{r,s})'\to (X_{r,s})'
\]
is bounded.
\end{itemize}
Then for every $f\in X$, $g\in [(X^r)']^{\frac{1}{r}}$ there exists a weight $w\in A_{p,(r,s)}(\mathcal{E})$ such that
\[
[w]^\mathcal{E}_{p,(r,s)}\leq 2^{\frac{1}{r}-\frac{1}{s}}\|M^\mathcal{E}\|_{X_{r,s}\to X_{r,s}}^{\frac{1}{r}-\frac{1}{p}}\|M^\mathcal{E}\|_{(X_{r,s})'\to (X_{r,s})'}^{\frac{1}{p}-\frac{1}{s}}
\]
and $f\in L^p_w(\Omega)$, $g\in L^{\frac{1}{\frac{1}{r}-\frac{1}{p}}}_{w^{-1}}(\Omega)$ with
\[
\|f\|_{L^p_w(\Omega)}\|g\|_{L^{\frac{1}{\frac{1}{r}-\frac{1}{p}}}_{w^{-1}}(\Omega)}\leq 2^{\frac{1}{r}-\frac{1}{s}}\|f\|_X\|g\|_{[(X^r)']^{\frac{1}{r}}}.
\]
\end{corollary}
\begin{proof}
Since $f\in X$, it follows from Corollary~\ref{cor:factorization} that there exist $0\leq h\in (X_{r,s})^{\frac{1}{r}-\frac{1}{s}}$ and $0\leq k\in L^s(\Omega)$ so that $|f|=hk$ and
\begin{equation}\label{eq:abstractlimrange1}
\|f\|_X=\|h\|_{(X_{r,s})^{\frac{1}{r}-\frac{1}{s}}}\|k\|_{L^s(\Omega)}.
\end{equation}
We set $\frac{1}{q_1}=\frac{1}{p}-\frac{1}{s}$, $\frac{1}{q_2}=\frac{1}{p}-\frac{1}{r}$ so that $\frac{1}{q}=\frac{1}{r}-\frac{1}{s}$. Taking $X_1:=(X_{r,s})^{\frac{1}{r}-\frac{1}{s}}$, $X_2=(X_{r,s}')^{\frac{1}{r}-\frac{1}{s}}=[(X^r)']^{\frac{1}{r}}$ we have $X_2^q=X'=(X_1^q)'$ and hence, by using Theorem~\ref{thm:mainabs} with $f_1=h$, $f_2=g$, there is a weight $w\in A_{p,(r,s)}(\mathcal{E})$ with
\[
[w]_{p,(r,s)}^\mathcal{E}\leq 2^{\frac{1}{r}-\frac{1}{s}}\|M^\mathcal{E}\|_{X_{r,s}\to X_{r,s}}^{\frac{1}{r}-\frac{1}{p}}\|M^\mathcal{E}\|_{(X_{r,s})'\to (X_{r,s})'}^{\frac{1}{p}-\frac{1}{s}}
\]
and
\[
\|h\|_{L_w^{\frac{1}{\frac{1}{p}-\frac{1}{s}}}(\Omega)}\|g\|_{L_{w^{-1}}^{\frac{1}{\frac{1}{r}-\frac{1}{p}}}(\Omega)}\leq 2^{\frac{1}{r}-\frac{1}{s}}\|h\|_{(X_{r,s})^{\frac{1}{r}-\frac{1}{s}}}\|g\|_{[(X^r)']^{\frac{1}{r}}}.
\]
Hence, by \eqref{eq:abstractlimrange1} we have
\begin{align*}
\|f\|_{L^p_w(\Omega)}\|g\|_{L^{\frac{1}{\frac{1}{r}-\frac{1}{p}}}_{w^{-1}}(\Omega)}
&\leq\|h\|_{L_w^{\frac{1}{\frac{1}{p}-\frac{1}{s}}}(\Omega)}\|k\|_{L^s(\Omega)}\|g\|_{L_{w^{-1}}^{\frac{1}{\frac{1}{r}-\frac{1}{p}}}(\Omega)}\\
&\leq 2^{\frac{1}{r}-\frac{1}{s}}\|h\|_{(X_{r,s})^{\frac{1}{r}-\frac{1}{s}}}\|k\|_{L^s(\Omega)}\|g\|_{[(X^r)']^{\frac{1}{r}}}\\
&=2^{\frac{1}{r}-\frac{1}{s}}\|f\|_X\|g\|_{[(X^r)']^{\frac{1}{r}}}.
\end{align*}
The assertion follows.
\end{proof}
We give the full argument to extrapolate bounds for an operator in the limited range setting in Theorem~\ref{thm:aprsextrapolation} below.

\subsection{Extrapolation formulated in a dual statement in favor of extrapolation pairs}\label{subsection:extrapolationpairssuck}
Usually in the literature, extrapolation theorems are formulated in terms of pairs of functions, rather than in the dual statement appearing in Theorem~\ref{thm:mainabs}. In this subsection we explain why the dual statement is more favorable. For simplicity, we consider the case where $\Omega=\R^d$ with the basis of cubes and $q=1$.

Typically, for the statement in terms of extrapolation pairs, one considers a family of pairs $\mathcal{F}\subseteq L^0(\R^d)\times L^0(\R^d)$. Then one assumes that there is an increasing function $\phi:[0,\infty)\to[0,\infty)$ such that for all $(f,g)\in\mathcal{F}$ and all $w^p\in A_p$ (which we recall is equivalent to $[w]_p<\infty$) we have
\begin{equation}\label{eq:extrapolationpairsin}
\|f\|_{L^p_w(\R^d)}\leq\phi([w]_p)\|g\|_{L^p_w(\R^d)}.
\end{equation}
At this point one wishes to deduce that for an appropriate class of Banach function spaces $X$, there is a constant $C_X>0$ for which
\begin{equation}\label{eq:extrapolationpairsout}
\|f\|_X\leq C_X\|g\|_X
\end{equation}
for all $(f,g)\in\mathcal{F}$. As a matter of fact, using Corollary~\ref{cor:mainabs}, we can show that this holds with
\[
C_X=2\phi(2\|M\|_{X\to X}^{\frac{1}{p'}}\|M\|_{X'\to X'}^{\frac{1}{p}}).
\]

Now, suppose we wish to use this result to extrapolate bounds for a given operator. More precisely, suppose $T$ is an operator that is well-defined on all functions $f\in L^p_w(\R^d)$ with $w^p\in A_p$ (i.e, $T$ is defined on $\bigcup_{w\in A_p}L^p_w(\R^d)$). Moreover, assume that there is an increasing function $\phi:[0,\infty)\to[0,\infty)$ such that for all $w^p\in A_p$ and $f\in L^p_w(\R^d)$ we have
\begin{equation}\label{eq:extrapolationpairsinT}
\|Tf\|_{L^p_w(\R^d)}\leq\phi([w]_p)\|f\|_{L^p_w(\R^d)}.
\end{equation}
The conclusion we wish to reach then is that for all appropriate Banach function spaces $X$, (i.e., the ones for which $M$ is bounded on $X$ and $X'$,) $Tf$ is well-defined for any $f\in X$ and
\begin{equation}\label{eq:extrapolationpairsoutT}
\|Tf\|_X\leq C_X\|f\|_X.
\end{equation}

The idea is to now consider the pairs of functions
\[
\mathcal{F}=\{(Tf,f):f\in S\},
\]
where $S\subseteq L^0(\R^d)$ is an appropriate class of functions. A condition that $S$ has to satisfy is that \eqref{eq:extrapolationpairsin} holds for the pair $(Tf,f)$ whenever $f\in S$. By \eqref{eq:extrapolationpairsinT}, this holds whenever $f\in L^p_w(\R^d)$ for some $w\in A_p$. However, since $S$ cannot depend on the choice of the weight $w\in A_p$, this means that we must have
\[
S\subseteq\bigcap_{w\in A_p} L^p_w(\R^d).
\]
This intersection is not empty, since, for example, it contains the space of bounded functions of compact support $L^\infty_c(\R^d)$. Indeed, if $w^p\in A_p$, then $w\in L^p_{\loc}(\R^d)$ and hence, if $f\in L^\infty_c(\R^d)$, then
\[
\|f\|_{L^p_w(\R^d)}\leq\|f\|_{L^\infty(\R^d)}\|w\|_{L^p(\supp f)}<\infty.
\]
Choosing $S=L^\infty_c(\R^d)$, the conclusion \eqref{eq:extrapolationpairsout} now becomes
\[
\|Tf\|_X\leq C_X\|f\|_X
\]
for all appropriate $X$ for all $f\in L^\infty_c(\R^d)$. To conclude \eqref{eq:extrapolationpairsoutT}, one would have to extend the operator $T$ to all of $X$ through a density argument. This requires some regularity of the operator $T$, e.g., it is (sub)linear, and it requires $L^\infty_c(\R^d)$ to be a dense subspace of $X$. While the former condition can be undesirable when we want to consider $T$ where $T$ is not necessarily an operator, the second condition is undesirable in that this means we would only get partial results for certain spaces. For example, $M$ is bounded on both $X=L^{p,\infty}(\R^d)$ and $X'=L^{p',1}(\R^d)$, but the space $L^\infty_c(\R^d)$ is not dense in $X$, and hence, we would not get bounds on the whole space $X$.

Another way is to take $S$ large enough so that it includes pairs $(Tf,f)$, where the left-hand side is still well-defined, but the right-hand side could be infinite. In this case one would have to delve back into the proof of the extrapolation theorem to see if the function $f$ having infinite norm in certain spaces does not lead to any problems. At this point one would also have to be careful whether $Tf$ is contained in $X$. These issues can usually be solved by an appropriate limiting argument.

We conclude that the formulation of extrapolation in terms of pairs of functions when applied to bounds for operators results in either a sub optimal result, or in a cumbersome argument. This is especially unfortunate since \emph{the conclusion \eqref{eq:extrapolationpairsoutT} follows from \eqref{eq:extrapolationpairsinT} without any limiting argument or additional assumptions on $T$ or $X$.}

As a matter of fact, in this case the formulation in terms of pairs of functions does a disservice to the original argument, and is more complicated without any benefit. Indeed, if $X$ is a Banach function space with $M$ bounded on $X$ and $X'$, then Corollary~\ref{cor:mainabs} tells us that if $f\in X$, then there exists a weight $w^p\in A_p$ such that $f\in L^p_w(\R^d)$ (i.e., $X\subseteq \bigcup_{w\in A_p}L^p_w(\R^d)$). Since $T$ is already defined on $L^p_w(\R^d)$, this means that $Tf$ is also well-defined, without needing any kind of extension argument.

For completeness, we give the full (rather short) argument in general spaces here.
\begin{theorem}\label{thm:apextrapolation}
Let $p\in[1,\infty]$. Let $(\Omega,|\cdot|)$ be a $\sigma$-finite measure space with a basis of sets $\mathcal{E}$. Suppose $T$ is an operator that is well-defined on all functions $f\in L^p_w(\Omega)$ with $w\in A_p(\mathcal{E})$. Moreover, assume that there is an increasing function $\phi:[0,\infty)\to[0,\infty)$ such that for all $w\in A_p(\mathcal{E})$ and $f\in L^p_w(\Omega)$ we have
\begin{equation}\label{eq:extrapolationpairsinTthm}
\|Tf\|_{L^p_w(\Omega)}\leq\phi([w]^\mathcal{E}_p)\|f\|_{L^p_w(\Omega)}.
\end{equation}

Let $X$ be a Banach function space over $\Omega$. Then $Tf$ is well-defined for any $f\in X$ with
\begin{equation}\label{eq:extrapolationpairsoutTthm}
\|Tf\|_X\leq 2\phi(2\|M^\mathcal{E}\|_{X\to X}^{\frac{1}{p'}}\|M^\mathcal{E}\|_{X'\to X'}^{\frac{1}{p}})\|f\|_X,
\end{equation}
provided that the following holds:
\begin{itemize}
\item If $p\neq 1$, 
\[
M^\mathcal{E}:X\to X
\]
is bounded.
\item If $p\neq\infty$,
\[
M^\mathcal{E}:X'\to X'
\]
is bounded.
\end{itemize}

\end{theorem}
\begin{proof}
Fix $f\in X$ and $g\in X'$. Then by Corollary~\ref{cor:mainabs}, there exists a weight $w\in A_p(\mathcal{E})$ with
\[
[w]_p\leq 2\|M^\mathcal{E}\|_{X\to X}^{\frac{1}{p'}}\|M^\mathcal{E}\|_{X'\to X'}^{\frac{1}{p}}
\]
and $f\in L^p_w(\Omega)$, $g\in L^{p'}_{w^{-1}}(\Omega)$ with
\[
\|f\|_{L^p_w(\Omega)}\|g\|_{L^{p'}_{w^{-1}}(\Omega)}\leq 2\|f\|_X\|g\|_{X'}.
\]
Thus, by \eqref{eq:extrapolationpairsinTthm} we have
\begin{align*}
\|(Tf)g\|_{L^1(\Omega)}&\leq\phi([w]_p)\|f\|_{L^p_w(\Omega)}\|g\|_{L^{p'}_{w^{-1}}(\Omega)}\\
&\leq 2\phi(2\|M^\mathcal{E}\|_{X\to X}^{\frac{1}{p'}}\|M^\mathcal{E}\|_{X'\to X'}^{\frac{1}{p}})\|f\|_X\|g\|_{X'}.
\end{align*}
Since
\[
\|Tf\|_X=\sup_{\|g\|_{X'}=1}\|(Tf)g\|_{L^1(\Omega)},
\]
this proves \eqref{eq:extrapolationpairsoutTthm}.
\end{proof}

However, it could still be the case that one requires the utility given by pairs of weights. In this case we offer an alternative solution. We start with a set $V$, not even necessarily of functions, and consider a map $S:V\to L^0(\Omega)$. The statement of the result then becomes that if $T$ is well-defined for all $f\in V$ and $w\in A_p(\mathcal{E})$ for which $Sf\in L^p_w(\Omega)$, (i.e., $T$ is defined on $\bigcup_{w\in A_p(\mathcal{E})}S^{-1}(L^p_w(\Omega))$,) with
\[
\|Tf\|_{L^p_w(\Omega)}\leq\phi([w]_{p})\|Sf\|_{L^p_w(\Omega)},
\]
then $Tf$ is well-defined for any $f\in V$ with $Sf\in X$, and
\[
\|Tf\|_X\leq 2\phi(2\|M^\mathcal{E}\|_{X\to X}^{\frac{1}{p'}}\|M^\mathcal{E}\|_{X'\to X'}^{\frac{1}{p}})\|f\|_X.
\]

We give the complete argument in the limited range case here:
\begin{theorem}\label{thm:aprsextrapolation}
Let $r\in(0,\infty)$, $s\in(0,\infty]$ with $r<s$ and let $p\in[r,s]$. Let $(\Omega,|\cdot|)$ be a $\sigma$-finite measure space with a basis of sets $\mathcal{E}$. Let $V$ be a set and let $S:V\to L^0(\Omega)$ be a map. Moreover, suppose
\[
T:\bigcup_{w\in A_{p,(r,s)}(\mathcal{E})}S^{-1}(L^p_w(\Omega))\to L^0(\Omega)
\]
is a map for which there is an increasing function $\phi:[0,\infty)\to[0,\infty)$ such that for all $w\in A_{p,(r,s)}(\mathcal{E})$ and $f\in V$ with $Sf\in L^p_w(\Omega)$ we have
\begin{equation}\label{eq:aprsextrapolationin}
\|Tf\|_{L^p_w(\Omega)}\leq\phi([w]^\mathcal{E}_{p,(r,s)})\|Sf\|_{L^p_w(\Omega)}.
\end{equation}

Let $X$ be an $r$-convex and $s$-concave quasi-Banach function spaces over $(\Omega,|\cdot|)$ for which the following conditions hold:
\begin{itemize}
\item If $p\neq r$,
\[
M^\mathcal{E}:X_{r,s}\to X_{r,s}
\]
is bounded;
\item If $p\neq s$, 
\[
M^\mathcal{E}:(X_{r,s})'\to (X_{r,s})'
\]
is bounded.
\end{itemize}
Then $Tf$ is well-defined for any $f\in V$ with $Sf\in X$ and
\begin{equation}\label{eq:aprsextrapolationout}
\|Tf\|_X\leq 2^{\frac{1}{r}-\frac{1}{s}}\phi(2^{\frac{1}{r}-\frac{1}{s}}\|M^\mathcal{E}\|_{X_{r,s}\to X_{r,s}}^{\frac{1}{r}-\frac{1}{p}}\|M^\mathcal{E}\|_{(X_{r,s})'\to (X_{r,s})'}^{\frac{1}{p}-\frac{1}{s}})\|Sf\|_X.
\end{equation}
\end{theorem}
\begin{proof}
Let $f\in V$ with $Sf\in X$ and $g\in [(X^r)']^{\frac{1}{r}}$. Then by Corollary~\ref{cor:abstractaprsextrapolation} there is a weight $w\in A_{p,(r,s)}(\mathcal{E})$ such that
\[
[w]^\mathcal{E}_{p,(r,s)}\leq 2^{\frac{1}{r}-\frac{1}{s}}\|M^\mathcal{E}\|_{X_{r,s}\to X_{r,s}}^{\frac{1}{r}-\frac{1}{p}}\|M^\mathcal{E}\|_{(X_{r,s})'\to (X_{r,s})'}^{\frac{1}{p}-\frac{1}{s}}
\]
and $Sf\in L^p_w(\Omega)$, $g\in L^{\frac{1}{\frac{1}{r}-\frac{1}{p}}}_{w^{-1}}(\Omega)$ with
\[
\|Sf\|_{L^p_w(\Omega)}\|g\|_{L^{\frac{1}{\frac{1}{r}-\frac{1}{p}}}_{w^{-1}}(\Omega)}\leq 2^{\frac{1}{r}-\frac{1}{s}}\|Sf\|_X\|g\|_{[(X^r)']^{\frac{1}{r}}}.
\]
Then by \eqref{eq:aprsextrapolationin} we have
\begin{align*}
\|(Tf)g\|_{L^r(\Omega)}&\leq\phi([w]_{p,(r,s)})\|Sf\|_{L^p_w(\Omega)}\|g\|_{L^{\frac{1}{\frac{1}{r}-\frac{1}{p}}}_{w^{-1}}(\Omega)}\\
&\leq 2^{\frac{1}{r}-\frac{1}{s}}\phi(2^{\frac{1}{r}-\frac{1}{s}}\|M^\mathcal{E}\|_{X_{r,s}\to X_{r,s}}^{\frac{1}{r}-\frac{1}{p}}\|M^\mathcal{E}\|_{(X_{r,s})'\to (X_{r,s})'}^{\frac{1}{p}-\frac{1}{s}})\|Sf\|_X\|g\|_{[(X^r)']^{\frac{1}{r}}}.
\end{align*}
Thus, the result follows from the fact that
\begin{align*}
\|Tf\|_X
&=\||Tf|^r\|_{X^r}^{\frac{1}{r}}
=\sup_{\|g^r\|_{(X^r)'}=1}\||Tf|^r |g|^r\|_{L^1(\Omega)}^{\frac{1}{r}}\\
&=\sup_{\|g\|_{[(X^r)']^{\frac{1}{r}}}=1}\|(Tf)g\|_{L^r(\Omega)}.\qedhere
\end{align*}
\end{proof}

When $V=L^0(\Omega)$ and $S$ is the identity operator, we recover the limited range extrapolation theorem for operators $T$.

\begin{remark}\label{rem:extrapolationpairs}
We point out that writing the result this way is actually more general than writing things in terms of extrapolation pairs. Indeed, if one has a family $\mathcal{F}\subseteq L^0(\Omega)\times L^0(\Omega)$, then one can take $V=\mathcal{F}$, $S(f,g):=g$, and $T(f,g):=f$. The assertion \eqref{eq:aprsextrapolationout} can then be formulated as saying that such in inequality holds for $(f,g)\in\mathcal{F}$, provided that $g\in X$, i.e., provided that the right-hand side is finite, which is how this is often written in the literature.
\end{remark}

\subsection{Vector-valued estimates through extrapolation}\label{subsection:fubini}
When $(\Omega_1,|\cdot|_1)$ and $(\Omega_2,|\cdot|_2)$ are two $\sigma$-finite measure spaces and $X$ and $Y$ are quasi-Banach function spaces over $\Omega_1$ and $\Omega_2$ respectively, then one can define a new quasi-Banach function space over the measure space $\Omega_1\times\Omega_2$ equipped with the product measure by considering the functions $f\in L^0(\Omega_1\times\Omega_2)$ so that for each $x\in\Omega_1$ we have $f(x,\cdot)\in Y$ and the function $h_f(x):=\|f(x,\cdot)\|_Y$ is in $X$. This space is denoted as $X(Y)$, and has quasi-norm
\[
\|f\|_{X(Y)}:=\|h_f\|_X.
\]
Sometimes this space is also referred to as a mixed norm space.

Suppose again we have $\Omega_1=\R^d$ with the Lebesgue measure and write $\Omega:=\Omega_2$. For $p\in[1,\infty]$, suppose we are given an operator $T:\bigcup_{w^p\in A_p}L^p_w(\R^d)\to L^0(\R^d)$. If we define the set $U\subseteq L^0(\Omega_1\times\Omega_2)$ by
\[
U:=\Big\{f\in L^0(\R^d\times\Omega):f(\cdot,y)\in\bigcup_{w^p\in A_p}L^p_w(\R^d)\quad\text{for all $y\in\Omega_2$}\Big\},
\]
then we can define a new operator $\widetilde{T}:U\to L^0(\R^d\times\Omega)$ through
\[
\widetilde{T}f(x,y):=T(f(\cdot,y))(x).
\]
A natural question to ask is if we can extend the conclusion of the extrapolation theorem to give us bounds of the form
\begin{equation}\label{eq:vvboundsbfs}
\|\widetilde{T}f\|_{X(Y)}\leq C_{X,Y}\|f\|_{X(Y)}.
\end{equation}
A classical case of this is where $Y=\ell^p(\N)$, in which case we have $\Omega=\N$, equipped with the counting measure. For $n\in\N$, if we write $f_n=f(\cdot,n)$, then the space $U$ is the space of sequences $(f_n)_{n\in\N}$ where $f_n\in L^p_w(\R^d)$ for some $w^p\in A_p$ for all $n\in\N$, and $\widetilde{T}$ is the operator that applies $T$ pointwise, i.e.,
\[
\widetilde{T}(f_n)_{n\in\N}=(Tf_n)_{n\in\N}.
\]
In this case, the extrapolation theorem gives us bounds of the form \eqref{eq:vvboundsbfs} more or less for free. For a weight $w$ and $X=L^p_w(\R^d)$, we write $X(Y)=L^p_w(\R^d;\ell^p(\N))$ and $Y(X)=\ell^p(\N;L^p_w(\R^d))$. The thing that freely allows us to obtain bounds in this space is the fact that
\begin{equation}\label{eq:fubinibfs}
\big\|(f_n)_{n\in\N}\big\|_{L^p_w(\R^d;\ell^p(\N))}=\big\|(\|f_n\|_{L_w^p(\R^d)})_{n\in\N}\big\|_{\ell^p(\N)}
\end{equation}
by Fubini's theorem. We present the argument here.
\begin{theorem}
Let $(\Omega,|\cdot|)$ be a $\sigma$-finite measure space with a basis of sets $\mathcal{E}$. Let $p\in[1,\infty]$, let $V$ be a set, and let $S:V\to L^0(\Omega)$ be a map. Moreover, suppose
\[
T:\bigcup_{w\in A_p(\mathcal{E})}S^{-1}(L^p_w(\Omega))\to L^0(\Omega)
\]
is a map for which there is an increasing function $\phi:[0,\infty)\to[0,\infty)$ such that for all $w\in A_p(\mathcal{E})$ and $f\in V$ with $Sf\in L^p_w(\Omega)$ we have
\begin{equation}\label{eq:extrapolationpairsinTthmvv}
\|Tf\|_{L^p_w(\Omega)}\leq\phi([w]^\mathcal{E}_p)\|Sf\|_{L^p_w(\Omega)}.
\end{equation}
Let $X$ be a Banach function space over $\Omega$. Then $Tf$ is well-defined for any $f\in V$ with $Sf\in X$ and for all sequences $(f_n)_{n\in\N}$ in $V$ with $(Sf_n)_{n\in\N}\in X(\ell^p(\N))$ we have
\begin{equation}\label{eq:extrapolationpairsoutTthmvv}
\big\|\|(Tf_n)_{n\in\N}\|_{\ell^p(\N)}\big\|_X\leq 2\phi(2\|M^\mathcal{E}\|_{X\to X}^{\frac{1}{p'}}\|M^\mathcal{E}\|_{X'\to X'}^{\frac{1}{p}})\big\|\|(Sf_n)_{n\in\N}\|_{\ell^p(\N)}\big\|_X,
\end{equation}
provided that the following holds:
\begin{itemize}
\item If $p\neq 1$, 
\[
M^\mathcal{E}:X\to X
\]
is bounded.
\item If $p\neq\infty$,
\[
M^\mathcal{E}:X'\to X'
\]
is bounded.
\end{itemize}

\end{theorem}
\begin{proof}
Define
\[
V:=\bigcup_{w\in A_p(\mathcal{E})} L^p_w(\Omega;\ell^p(\N))
\]
and define $\mathcal{S}:V\to L^0(\Omega)$,
\[
\mathcal{S}f(x):=h_f(x)=\big\|(Sf_n(x))_{n\in\N}\big\|_{\ell^p(\N)}.
\]
Note that now
\[
\bigcup_{w\in A_p(\mathcal{E})}S^{-1}(L^p_w(\Omega))=V,
\]
and we define $\mathcal{T}:V\to L^0(\Omega)$ by
\[
\mathcal{T}f(x):=\big\|(Tf_n(x))_{n\in\N}\big\|_{\ell^p(\N)}.
\]
By combining \eqref{eq:fubinibfs} with $\R^d$ replaced by $\Omega$ with \eqref{eq:extrapolationpairsinTthmvv}, we obtain
\begin{align*}
\|\mathcal{T}f\|_{L^p_w(\Omega)}&=\big\|(\|Tf_n\|_{L_w^p(\Omega)})_{n\in\N}\big\|_{\ell^p(\N)}\\
&\leq\phi([w]^\mathcal{E}_p)\big\|(\|Sf_n\|_{L_w^p(\Omega)})_{n\in\N}\big\|_{\ell^p(\N)}\\
&=\phi([w]^\mathcal{E}_p)\|\mathcal{S}f\|_{L^p_w(\Omega)}
\end{align*}
for all $f\in V$. Now, it follows from Theorem~\ref{thm:aprsextrapolation} that
then $\mathcal{T}f$ is well-defined for any $f\in V$ with $\mathcal{S}f\in X$, and
\[
\|\mathcal{T}f\|_X\leq 2\phi(2\|M^\mathcal{E}\|_{X\to X}^{\frac{1}{p'}}\|M^\mathcal{E}\|_{X'\to X'}^{\frac{1}{p}})\|\mathcal{S}f\|_X.
\]
The conclusion follows if we can show that $\{f\in V:\mathcal{S}f\in X\}=X(\ell^p(\N))$. The inclusion ``$\subseteq$'' is immediate. For the converse inclusion, pick $f\in X(\ell^p(\N))$. Then $h_f\in X$ and hence, there is a $w\in A_p(\mathcal{E})$ such that $h_f\in L^p_w(\Omega)$. But this means that $f\in V$, proving the result.
\end{proof}

\begin{remark}\label{rem:vvq}
In the case of $\Omega=\R^d$ with the basis of cubes, by extrapolation, the initial assumption \eqref{eq:extrapolationpairsinTthmvv} is true with a similar estimate in $L^q_w(\R^d)$ for all $q\in(1,\infty)$. Hence, by running the same argument again with $q$ replaced by $p$, the conclusion \eqref{eq:extrapolationpairsoutTthmvv} actually also holds with $\ell^p(\N)$ replaced by $\ell^q(\N)$ for all $q\in (1,\infty)$.
\end{remark}

\begin{remark}\label{rem:vvanylebesgue}
One can replace the measure space $\N$ by any $\sigma$-finite measure space $\Omega_1$ to obtain bounds in $X(L^p(\Omega_1))$, with the exact same argument. This is because the Fubini argument \eqref{eq:fubinibfs} still holds in this setting, i.e., we have $L^p_w(\Omega;L^p(\Omega_1))=L^p(\Omega_1;L^p_w(\Omega))$ for any such space.
\end{remark}

\begin{remark}
Having Remark~\ref{rem:vvq} and Remark~\ref{rem:vvanylebesgue} in mind, one can actually iterate the argument and obtain bounds in the space $X(Y)$ for any $Y$ given by an iterated mixed norm Lebesgue space with exponents in $(1,\infty)$, i.e., of the form
\[
Y=Y_1(Y_2(\ldots(Y_N))
\]
with $Y_n=L^{q_n}(\Omega_n)$ for a $\sigma$-finite measure space $\Omega_n$ and $q_n\in(1,\infty)$ for all $n\in\{1,\ldots,N\}$. The details are left to the interested reader.
\end{remark}

At this point one may wonder if this argument works for any Banach function space $Y$ that is not a Lebesgue space. However, it turns out that this is not the case. Indeed, by the Kolmogorov-Nagumo Theorem, see \cite[Theorem~3.1]{BBS02}, if a Banach function space $Y$ has the property that $L^p(\R^d;Y)=Y(L^p(\R^d))$ with equivalent norm, then one must have that $Y=L^p_v(\Omega)$ for some weight $v$. As a matter of fact, the full result says that if $X(Y)=Y(X)$ for any pair of spaces, then they must both be a weighted Lebesgue space, and hence, there is no hope in using the Fubini argument in any other Banach function space.

Fortunately there is a rich theory for obtaining bounds in the spaces $L^p_w(\R^d;Y)$ through different means. In \cite[Theorem~5]{Ru85} Rubio de Francia showed that one can obtain bounds in these spaces as long as $Y$ satisfies the so-called $\UMD$ property. This result was extended to the limited range setting $0<r<s\leq\infty$, where the assumption on the space becomes that $Y$ must be a $r$-convex and $s$-concave quasi-Banach function space, and $Y_{r,s}$ has the $\UMD$ property. Indeed, this was first done for $r\neq 1$, $s=\infty$ in \cite{ALV17} and for general $0<r<s\leq\infty$ in \cite{LN19}. As a matter of fact, the weaker notion of $Y\in\UMD_{r,s}$, which is equivalent to assuming that $Y$ is order-continuous and $Y^r$ and $(Y_{r,s})'$ have the so-called Hardy-Littlewood property, was introduced in \cite{LN22} and it was shown in \cite{Nidiss} that the vector-valued extrapolation theorem still holds for such spaces.

As a matter of fact, using the result from this work we now obtain the following result:
\begin{theorem}\label{thm:vectorvaluedmain}
Let $r\in(0,\infty)$, $s\in(0,\infty]$ with $r<s$. Suppose $T$ is a (sub)linear operator such that for all $p\in(r,s)$ it is well-defined on all functions $f\in L^p_w(\R^d)$ with $w\in A_{p,(r,s)}$ and, moreover, there is an increasing function $\phi_p:[0,\infty)\to[0,\infty)$ such that for all $w\in A_{p,(r,s)}$ and $f\in L^p_w(\R^d)$ we have
\[
\|Tf\|_{L^p_w(\R^d)}\leq\phi([w]_p)\|f\|_{L^p_w(\R^d)}.
\]
Let $Y$ be an order-continuous $r$-convex and $s$-concave quasi-Banach function space over a $\sigma$-finite measure space $(\Omega,|\cdot|)$ such that for all simple functions $f\in L^\infty_c(\R^d;Y)$ the function
\[
\widetilde{T}f(x,y):=T(f(\cdot,y))(x)
\]
is strongly measurable. Suppose $Y^r$ and $(Y_{r,s})'$ have the Hardy-Littlewood property, then for each $r$-convex and $s$-concave quasi-Banach function space $X$ such that 
\[
M:X_{r,s}\to X_{r,s},\quad M:(X_{r,s})'\to (X_{r,s})'
\]
are bounded, there is an increasing function $\phi_{Y,r,s}:[0,\infty)\times[0,\infty)\to[0,\infty)$ such that $\widetilde{T}:X(Y)\to X(Y)$ is bounded with
\[
\|\widetilde{T}\|_{X(Y)\to X(Y)}\leq\phi_{Y,r,s}(\|M\|_{X_{r,s}\to X_{r,s}},\|M\|_{(X_{r,s})'\to (X_{r,s})'}).
\]
\end{theorem}
We refer the reader to \cite{LN22, Nidiss} for the relevant definitions and an overview of the theory.
\begin{proof}
Fix $p\in(r,s)$. By \cite[Theorem 9.1.1]{Nidiss} there is an increasing function $\phi_{Y,p,r,s}:[0,\infty)\to[0,\infty)$ such that for all $w\in A_{p,(r,s)}$, $\widetilde{T}:L^p_w(\R^d;Y)\to L^p_w(\R^d;Y)$ is bounded with
\[
\|\widetilde{T}f\|_{L^p_w(\R^d;Y)}\leq\phi_{Y,p,r,s}([w]_{p,(r,s)})\|f\|_{L^p_w(\R^d;Y)}.
\]
Then it follows from Theorem~\ref{thm:aprsextrapolation} with $V=\bigcup_{w\in A_{p,(r,s)}}L_w^p(\Omega;Y)$ applied with $Tf$ replaced by $\mathcal{T}f(x):=\|\widetilde{T}f(x,\cdot)\|_Y$ and  $Sf(x):=\|f(x,\cdot)\|_Y$. That for all $f\in X(Y)$ we have
\begin{align*}
\|\widetilde{T}f\|_{X(Y)}&=\|\mathcal{T}f\|_X\leq 2^{\frac{1}{r}-\frac{1}{s}}\phi_{Y,p,r,s}(2^{\frac{1}{r}-\frac{1}{s}}\|M\|_{X_{r,s}\to X_{r,s}}^{\frac{1}{r}-\frac{1}{p
}}\|M\|_{(X_{r,s})'\to (X_{r,s})'}^{\frac{1}{p}
-\frac{1}{s}})\|Sf\|_X\\
&=2^{\frac{1}{r}-\frac{1}{s}}\phi_{Y,p,r,s}(2^{\frac{1}{r}-\frac{1}{s}}\|M\|_{X_{r,s}\to X_{r,s}}^{\frac{1}{r}-\frac{1}{p
}}\|M\|_{(X_{r,s})'\to (X_{r,s})'}^{\frac{1}{p}
-\frac{1}{s}})\|f\|_{X(Y)}.
\end{align*}
The assertion follows by taking $p=1+\frac{r}{s'}$.
\end{proof}

\begin{remark}
The extrapolation arguments from \cite{LN19, Nidiss} are not sharp, since key ingredients of the proof are versions of several self-improvement results related to the Hardy-Littlewood property, for which a quantitative result did not exist in the literature. However, with Theorem~\ref{thm:maximalselfimprove} in hand it should be possible to prove sharp versions of these results. We leave the details to the interested reader.
\end{remark}

\subsection{Sharp estimates in terms of the operator norms of the maximal operator}
Sharp extrapolation in terms of the weight constant has been used in various places throughout the literature, and was first formally presented in \cite{DGPP05}. However, in our current framework, the question of finding sharp estimates in terms of the operator norms of $M$ on $X$ and $X'$ becomes relevant in an intrinsic sense. When extrapolating the initial estimate
\[
\|Tf\|_{L^p_w(\R^d)}\leq\phi([w]_p)\|f\|_{L^p_w(\R^d)}.
\]
Our arguments allow us to deduce that for an appropriate class of Banach function spaces $X$, there is a constant $C_X>0$ for which
\[
\|Tf\|_X\leq C_X\|f\|_X
\]
for all $(f,g)\in\mathcal{F}$. In Theorem~\ref{thm:apextrapolation} we have shown that this holds with
\begin{equation}\label{eq:sharpxext}
C_X=2\phi(2\|M\|_{X\to X}^{\frac{1}{p'}}\|M\|_{X'\to X'}^{\frac{1}{p}}).
\end{equation}
However, using classical techniques, e.g., see \cite[Theorem~3.1]{Du11}, one can even show that if $X=L^q_v(\R^d)$ with $v^q\in A_q$, then we can take
\begin{equation}\label{eq:sharplpext}
C_X=\begin{cases}
2^{1-\frac{q}{p}}\phi\big([v]^{\frac{q}{p}}_q\|M\|^{1-\frac{q}{p}}_{L^q_v(\R^d)\to L^q_v(\R^d)}\big) & \text{if $q\leq p$};\\
2^{1-\frac{q'}{p'}}\phi\big([v]^{\frac{q'}{p'}}_q\|M\|^{1-\frac{q'}{p'}}_{L^{q'}_{v^{-1}}(\R^d)\to L^{q'}_{v^{-1}}(\R^d)}\big) & \text{if $q\geq p$}.
\end{cases}
\end{equation}
Using Buckley's bound at this stage allows us to recover the sharp extrapolation theorem in terms of the weight constant $[v]_q$.

The idea here is that, for example, if one wants to extrapolate to $q<p$, then one can use the fact that $M$ is bounded on $L^p_v(\R^d)$ to use the Rubio de Francia algorithm to obtain a weight $u\in A_1$. Then, since $v^q\in A_q$, one can consider the weight
\[
w:=u^{-\big(1-\frac{q}{p}\big)}v^{\frac{q}{p}},
\]
which satisfies $w^p\in A_p$. Plugging this weight into the initial bound then precisely gives us \eqref{eq:sharplpext}. The case $q>p$ is treated analogously, this time using the Rubio de Francia algorithm in $L^{q'}_{v^{-1}}(\R^d)$.

In essence, we are using two facts of the space $X=L^q_v(\R^d)$ that allow us to give a better bound than \eqref{eq:sharpxext}. The first is the fact that either $q>p$ or $q<p$ allows us to give our extrapolation a sense of direction. We are either  `extrapolating up' or `extrapolating down' to $L^q_v(\R^d)$. The second is the fact that we have the weight $v$ in the space to work with. This means that once we have used the Rubio de Francia algorithm once, we do not need to use it a second time to get a weight $w^p\in A_p$. Indeed, in the above example of $q<p$, we note that the weight $u$ obtained from the Rubio de Francia algorithm satisfies $[u^{-1}]_\infty=[u]_1<\infty$. Since also $v^q\in A_q$, and $p\in[q,\infty]$, this means we can create a weight $w$ by interpolating $u^{-1}$ and $v$, which is exactly what we did.

The first idea of extrapolating up or down, can be realized for general Banach function spaces as well. Say, our initial bound is at $p\in(1,\infty)$ and we again take $X=L^q_v(\R^d)$. Our main remark now, is that $q\geq p$ if and only if $X$ is $p$-convex and $q\leq p$ if and only if $X$ is $p$-concave. Indeed, $X$ is $p$-convex exactly when
\[
X^p=L^{\frac{q}{p}}_{v^p}(\R^d)
\]
is a Banach function space, which is true exactly when $\frac{q}{p}\geq 1$, and $X$ is $p$-concave exactly when
\[
(X')^{p'}=L^{\frac{q'}{p'}}_{v^{-p'}}(\R^d)
\]
is a Banach function space, which happens precisely when $\frac{q'}{p'}\geq 1$, or $p\geq q$. In conclusion, the idea of extrapolating up or down can be rephrased as considering Banach function spaces $X$ that are either $p$-convex or $p$-concave.

As a side remark, we point out that the only spaces that are both $p$-convex and $p$-concave are the spaces of the form $X=L^p_w(\R^d)$ for some weight $w$, and hence, it makes sense to have these spaces as the ones in which we have our initial estimate.

By applying our limited range extrapolation theorem for $r=p$ and $s=p$ respectively, we can obtain the following result:
\begin{theorem}\label{thm:pconcorpconvext}
Let $p\in(0,\infty]$. Let $(\Omega,|\cdot|)$ be a $\sigma$-finite measure space with a basis of sets $\mathcal{E}$.

Let $V$ be a set and let $S:V\to L^0(\Omega)$ be a map. Moreover, suppose
\[
T:\bigcup_{w\in W}S^{-1}(L^p_w(\Omega))\to L^0(\Omega)
\]
is a map for which there is an increasing function $\phi:[0,\infty)\to[0,\infty)$ such that we have one of the following properties:
\begin{enumerate}[(i)]
\item\label{it:pconcorpconvin1} $W:=\{w:w^p\in A_1(\mathcal{E})\}$ and $\|Tf\|_{L^p_w(\Omega)}\leq\phi([w^p]^\mathcal{E}_1)\|Sf\|_{L^p_w(\Omega)}$ for all $w^p\in A_1(\mathcal{E})$ and $f\in V$ with $Sf\in L^p_w(\Omega)$;
\item\label{it:pconcorpconvin2} $W:=\{w:w^{-p'}\in A_1(\mathcal{E})\}$ and $\|Tf\|_{L^p_w(\Omega)}\leq\phi([w^{-p'}]^\mathcal{E}_1)\|Sf\|_{L^p_w(\Omega)}$ for all $w^{-p'}\in A_1(\mathcal{E})$ and $f\in V$ with $Sf\in L^p_w(\Omega)$.
\end{enumerate}
Let $X$ be a quasi-Banach function space over $(\Omega,|\cdot|)$ for which one of the following conditions hold:
\begin{enumerate}[(I)]
\item\label{it:pconcorpconv1} We have $p<\infty$, $X$ is $p$-convex, and $M^\mathcal{E}:(X^p)'\to (X^p)'$ is bounded;
\item\label{it:pconcorpconv2} We have $p>1$, $X$ is a $p$-concave Banach function space, and $M^\mathcal{E}:[(X')^{p'}]'\to [(X')^{p'}]'$ is bounded.
\end{enumerate}
Then for all $f\in V$ with $Sf\in X$ we have
\begin{equation}\label{eq:pconcorpconvout}
\|Tf\|_X\leq \begin{cases}
2^{\frac{1}{p}}\phi(2\|M^\mathcal{E}\|_{(X^p)'\to (X^p)'})\|Sf\|_X & \text{if \ref{it:pconcorpconvin1} and \ref{it:pconcorpconv1} are satisfied;}\\
2^{\frac{1}{p'}}\phi(2\|M^\mathcal{E}\|_{[(X')^{p'}]'\to [(X')^{p'}]'})\|Sf\|_X & \text{if \ref{it:pconcorpconvin2} and \ref{it:pconcorpconv2} are satisfied.}
\end{cases}
\end{equation}
\end{theorem}
\begin{remark}
Note that if
\[
\|Tf\|_{L^p_w(\Omega)}\leq\phi([w]^\mathcal{E}_p)\|Sf\|_{L^p_w(\Omega)}
\]
for $p\in(1,\infty)$ and all weights $w\in A_p(\mathcal{E})$, then both \ref{it:pconcorpconvin1} and \ref{it:pconcorpconvin2} are satisfied. Indeed, one readily checks that
\[
[w]^\mathcal{E}_p\leq\min\big(([w^p]^\mathcal{E}_1)^{\frac{1}{p}},([w^{-p'}]^\mathcal{E}_1)^{\frac{1}{p'}}\big).
\]
\end{remark}

\begin{proof}[Proof of Theorem~\ref{thm:pconcorpconvext}]
The first case follows from noting that $[w]_{p,(p,\infty)}^\mathcal{E}=\big([w^p]_1^\mathcal{E}\big)^{\frac{1}{p}}$ and applying Theorem~\ref{thm:aprsextrapolation} with $r=p$, $s=\infty$ and the second case follows from noting that $[w]_{p,(1,p)}^\mathcal{E}=\big([w^{-p'}]_1^\mathcal{E}\big)^{\frac{1}{p'}}$ and applying Theorem~\ref{thm:aprsextrapolation} with $r=1$, $s=p$.
\end{proof}

Note that we can also extract limited range versions of this result from Theorem~\ref{thm:aprsextrapolation} by taking different $r$ or $s$ in their respective cases.

The first case of this result recovers the Banach function space extrapolation result from \cite{CMP11}, while the second result provides a dual version. We point out that extrapolation from the point of view of the dual weight belonging to $A_1$ was first established in \cite{HMS88}, and has since appeared in works such as \cite{Ni19, CPR19, NR23, LOR22}.

Unfortunately this result is far from satisfactory. Even though we have now reduced our result to one that looks more like \eqref{eq:sharplpext}, we get a much less precise result. Indeed, plugging $X=L^q_v(\R^d)$ back into the above result, yields that for $q\geq p$ we need $M$ to be bounded on
\[
(X^p)'=L^{\big(\frac{q}{p}\big)'}_{v^{-p}}(\R^d),
\]
which happens if and only if $v^{q}\in A_{\frac{q}{p}}$, meaning that we do not get bounds for the general case $v^q\in A_q$. A similar thing happens in the case $q\leq p$. The reason for this is because in the above result we only took the convexity and the concavity of the space $X$ into account, but not the possibility of a weight appearing in the space. Since we are interested in intrinsic properties of the space $X$ without specializing to a weight in the space, this leaves us with the problem of whether we can replace the boundedness of $M:(X^p)'\to (X^p)'$ or $M:[(X')^{p'}]'\to [(X')^{p'}]'$ with the boundedness of $M:X'\to X'$ or $M:X\to X$ respectively. We discuss the results we obtain when the space $X$ has an explicit weight, which is the case for example when $X=L^q_v(\R^d)$, in Subsection~\ref{subsec:weightedbfs}. We conclude this subsection with a possible intrinsic method of obtaining weights, without reaching any further results.

Again, similarly to the case \ref{it:pconcorpconv1} above, we let $p\in(1,\infty)$ and assume that $X$ is a $p$-convex quasi-Banach function space, but this time we assume that
\[
M:X'\to X'
\]
is bounded.

Fixing $g\in X'$, we set
\[
u:=\sum_{k=0}^\infty\frac{M^k g}{2^k\|M\|_{X'\to X'}^k},
\]
so that $u\in X'$ with $\|u\|_{X'}\leq 2\|g\|_{X'}$. By applying Theorem~\ref{thm:spacesplitting} with $Y=X$, $q=p$, we find that
\begin{equation}\label{eq:sharpextsplitting}
X'=\big[(X^p)'\big]^{\frac{1}{p}}\cdot L^{p'}(\R^d)
\end{equation}
and there exist $0<w\in\big[(X^p)'\big]^{\frac{1}{p}}$, $0<k\in L^{p'}(\R^d)$ such that $u=wk$ and
\[
\|u\|_{X'}=\|w\|_{\big[(X^p)'\big]^{\frac{1}{p}}}\|k\|_{L^{p'}(\R^d)}.
\]
Note that for any $f\in X$ we now have
\[
\|fw\|_{L^p(\R^d)}
=\||f|^p w^p\|_{L^1(\R^d)}^{\frac{1}{p}}
\leq\||f|^p\|^{\frac{1}{p}}_{X^p}\|w^p\|^{\frac{1}{p}}_{(X^p)'}
=\|f\|_{X}\|w\|_{\big[(X^p)'\big]^{\frac{1}{p}}},
\]
and, since $|g|\leq u=wk$, this yields
\begin{align*}
\|fw\|_{L^p(\R^d)}\|gw^{-1}\|_{L^{p'}(\R^d)}
&\leq\|f\|_X\|\|w\|_{\big[(X^p)'\big]^{\frac{1}{p}}}\|k\|_{L^{p'}(\R^d)}\\
&=\|f\|_X\|u\|_{X'}\\
&\leq 2\|f\|_X\|g\|_{X'}.
\end{align*}
We conclude that the weight $w$ satisfies the second key inequality in Corollary~\ref{cor:mainabs} and it remains to check that $w^p\in A_p$. Unfortunately, this is the step where it is unclear to proceed at this level of generality.

In the case $X=L^q_v(\R^d)$, the factorization \eqref{eq:sharpextsplitting} takes the form
\[
L^{q'}_{v^{-1}}(\R^d)=L^{\frac{1}{\frac{1}{p}-\frac{1}{q}}}_{v^{-1}}(\R^d)\cdot L^{p'}(\R^d),
\]
and to factor $u=wk\in L^{q'}_{v^{-1}}(\R^d)$ one can take
\[
w:=u^{1-\frac{q'}{p'}}v^{\frac{q'}{p'}}\in L^{\frac{1}{\frac{1}{p}-\frac{1}{q}}}_{v^{-1}}(\R^d),\quad k:=\big(uv^{-1}\big)^{\frac{q'}{p'}}\in L^{p'}(\R^d).
\]
Hence, if $v^q\in A_q$, we have $w^p\in A_p$ by interpolation of weights. Unfortunately it seems that $X=L^q_v(\R^d)$ is the only choice of space where $w$ can have this particular form and, hence, it appears to be impossible to use this strategy in any other space.

\subsection{Extrapolation in \texorpdfstring{$A_\infty$}{Fujii-Wilson classes}}
In the classical setting of $\R^d$ with the basis of cubes, the $A_\infty$ class is defined as
\[
A_\infty=\bigcup_{p\in[1,\infty)}A_p.
\]
Moreover, we have $w\in A_\infty$ if and only if it has a so-called finite Fujii-Wilson $A_\infty$ constant:
\[
[w]_{FW}:=\sup_{Q}\frac{\langle M(w\ind_Q)\rangle_{1,Q}}{\langle w\rangle_{1,Q}}<\infty.
\]
We observe that for weights in this class, we have the following result:
\begin{proposition}\label{prop:ainftymbound}
Let $p\in(0,\infty)$. Then $w^p\in A_\infty$, if and only if there exists an $r\in(0,p]$ such that
\begin{equation}\label{eq:ainftymbound}
 M:L^{\big(\frac{p}{r}\big)'}_{w^{-r}}(\R^d)\to L^{\big(\frac{p}{r}\big)'}_{w^{-r}}(\R^d)
\end{equation}
is bounded.
\end{proposition}
\begin{proof}
Our first observation is that the bound \eqref{eq:ainftymbound} is equivalent to the condition
\[
[w]_{p,(r,\infty)}=[w^r]^{\frac{1}{r}}_{\frac{p}{r}}=[w^{-r}]^{\frac{1}{r}}_{\big(\frac{p}{r}\big)'}<\infty,
\]
and in this case we have $w^p=(w^r)^{\frac{p}{r}}\in A_{\frac{p}{r}}$, so $w^p\in A_\infty$.

Conversely, if $w^p\in A_\infty$, there is a $q\in[1,\infty)$ such that $w^p\in A_q$. Setting $r:=\frac{p}{q}\in(0,p]$, we have
\[
[w]_{p,(r,\infty)}=[w^p]^{\frac{1}{p}}_{A_q}<\infty,
\]
proving that \eqref{eq:ainftymbound} holds. The assertion follows.
\end{proof}

For a $\sigma$-finite measure space $(\Omega,|\cdot|)$ with a basis of sets $\mathcal{E}$, we can define
\[
[w]^\mathcal{E}_{FW}:=\sup_{E\in\mathcal{E}}\frac{\langle M^\mathcal{E}_E(w)\rangle_{1,E}}{\langle w\rangle_{1,E}},
\]
where
\[
M^\mathcal{E}_E(w):=\sup_{\substack{F\in\mathcal{E}\\ F\subseteq E}}\langle w\rangle_{1,F}\ind_F.
\]

Noting that $[w]^\mathcal{E}_{FW}\leq [w]_1^\mathcal{E}$, Proposition~\ref{prop:ainftymbound} now motivates the following corollary of Theorem~\ref{thm:pconcorpconvext}:
\begin{theorem}
Let $V$ be a set and let $S:V\to L^0(\R^d)$ be a map. Moreover, let
\[
T:\bigcup_{p\in(0,\infty)}\bigcup_{w^p\in A_{FW}(\mathcal{E})}S^{-1}(L^p_w(\Omega))\to L^0(\Omega)
\]
be a map for which for each $p\in(0,\infty)$ there is an increasing function $\phi_p:[0,\infty)\to[0,\infty)$ such that
\begin{equation}\label{eq:ainftyextrapin}
\|Tf\|_{L^p_w(\Omega)}\leq\phi_p([w^p]^\mathcal{E}_{FW})\|Sf\|_{L^p_w(\Omega)}
\end{equation}
for all $w^p\in A_{FW}(\mathcal{E})$ and $f$ such that $Sf\in L^p_w(\Omega)$.

Then for all quasi-Banach function spaces $X$ over $\Omega$ that are $r$-convex for some $r\in(0,\infty)$ and
\[
 M^\mathcal{E}:(X^r)'\to (X^r)'
\]
is bounded, we have
\[
\|Tf\|_X\leq 2^{\frac{1}{r}}\phi_r(2\|M^{\mathcal{E}}\|_{(X^r)'\to (X^r)'})\|Sf\|_X
\]
for all $f$ with $Sf\in X$.
\end{theorem}
\begin{proof}
Since
\[
[w^p]^{\mathcal{E}}_{FW}\leq[w^p]^\mathcal{E}_{1},
\]
this follows from the first case in Theorem~\ref{thm:pconcorpconvext} with $p=r$.
\end{proof}
While this result may seem rather cheap, it is important to note that in the case $\Omega=\R^d$ with the basis of cubes, the condition \eqref{eq:ainftyextrapin} only needs to hold for a single $p\in(0,\infty)$. Indeed, by the $A_\infty$ extrapolation result \cite{CMP04}, it then holds for all $p\in (0,\infty)$, allowing us to apply our theorem. Thus, in this case, we have obtained an extension of the $A_\infty$ extrapolation theorem.

We also point out that this version is an extension of the $A_\infty$ extrapolation theorems for rearrangement invariant Banach function spaces from \cite{CGMP06, CMP11}.

\subsection{Weak-type estimates in weak Banach function spaces}
\begin{definition}
Let $(\Omega,|\cdot|)$ be a $\sigma$-finite measure space and let $X$ be a quasi-Banach function space over $\Omega$. Then we define the weak space $wX$ as the space of those $f\in L^0(\Omega)$ for which
\[
\|f\|_{wX}:=\sup_{\lambda>0}\lambda\|\ind_{\{x\in\Omega:|f(x)|>\lambda\}}\|_X<\infty.
\]
\end{definition}
\begin{proposition}\label{prop:weakxinclusion}
If $X$ is a quasi-Banach function space, then so is $wX$. Moreover, we have $X\subseteq wX$ with
\[
\|f\|_{wX}\leq\|f\|_X.
\]
\end{proposition}
If $X$ is a Banach function space, then $wX$ is not necessarily also a Banach function space. Indeed, consider $X=L^1(\R^d)$. Then $wX=L^{1,\infty}(\R^d)$, which is not a Banach function space.
\begin{proof}[Proof of Proposition~\ref{prop:weakxinclusion}]
Note that since
\[
\lambda\ind_{\{x\in\Omega:|f(x)|>\lambda\}}\leq|f|,
\]
it follows from the ideal property of $X$ that $X\subseteq wX$ with
\[
\|f\|_{wX}\leq\|f\|_X.
\]
Next, we prove the ideal and Fatou properties. Let $0\leq f_n\uparrow f$ a.e. with $\sup_{n\in\N}\|f_n\|_{wX}<\infty$. Then, for a fixed $\lambda>0$, the functions
\[
h_n:=\lambda\ind_{\{x\in\Omega:|f_n(x)|>\lambda\}}
\]
satisfy $0\leq h_n\uparrow \lambda\ind_{\{x\in\Omega:|f(x)|>\lambda\}}$ and
\[
\sup_{n\in\N}\|h_n\|_X\leq\sup_{n\in\N}\|f_n\|_{wX}<\infty.
\]
Hence, by the Fatou property of $X$, we have $\lambda\ind_{\{x\in\Omega:|f(x)|>\lambda\}}\in X$
with
\[
\lambda\|\ind_{\{x\in\Omega:|f(x)|>\lambda\}}\|_X=\sup_{n\in\N}\lambda\|\ind_{\{x\in\Omega:|f_n(x)|>\lambda\}}\|_X.
\]
Taking a supremum over $\lambda>0$ proves the result.

For the saturation property, note that this follows from the fact that $X$ has the saturation property and $X\subseteq wX$. This proves the assertion.
\end{proof}
Lemma~\ref{lem:indE} still holds if we assume the weaker statement that $M^\mathcal{E}:X\to wX$ is bounded.
\begin{proposition}
Let $\mathcal{E}$ be a basis of sets in $\Omega$ and suppose that
\[
M^\mathcal{E}:X\to wX
\]
is bounded. Then $\ind_E\in X$ and $\ind_E\in X'$ with
\[
\|\ind_E\|_X\|\ind_E\|_{X'}\leq\|M^\mathcal{E}\|_{X\to wX}|E|
\]
for all $E\in\mathcal{E}$.
\end{proposition}
\begin{proof}
By Proposition~\ref{prop:weakorderunit} there is a weak order unit $\rho\in X$. Fix $E\in\mathcal{E}$. Then by the definition of $M^\mathcal{E}$ we have
\[
\langle\rho\rangle_{1,E}\ind_E\leq M^\mathcal{E}\rho\in wX.
\]
Since $\langle\rho\rangle_{1,E}>0$, it follows from the ideal property that $\ind_E\in wX$. Now, note that for any $f\in X$ we have
\begin{align*}
\int_\Omega\!\ind_E |f|\,\mathrm{d}x
&=\langle f\rangle_{1,E}|E|=\frac{|E|}{\|\ind_E\|_{wX}}\|\langle f\rangle_{1,E}\ind_E\|_{wX}\\
&\leq\|M^\mathcal{E}\|_{X\to wX}\frac{|E|}{\|\ind_E\|_{wX}}\|f\|_X.
\end{align*}
Hence, $\ind_E\in X'$ with
\[
\|\ind_E\|_{X'}\leq\|M^\mathcal{E}\|_{X\to wX}\frac{|E|}{\|\ind_E\|_{wX}}.
\]
The result now follows from the observation that
\[
\ind_{\{x\in\Omega:\ind_E(x)>\lambda\}}=\ind_E
\]
when $\lambda\in(0,1)$, and equals $0$ when $\lambda\geq 1$, so that
\[
\|\ind_E\|_{wX}=\sup_{\lambda\in(0,1)}\lambda\|\ind_E\|_X=\|\ind_E\|_X.
\]
\end{proof}

Extrapolating from weak type estimates can be deduced from the strong type case using a trick that can already be found in \cite{GM04}.
\begin{theorem}
Let $r\in(0,\infty)$, $s\in(0,\infty]$ with $r<s$ and let $p\in[r,s]$. Let $(\Omega,|\cdot|)$ be a $\sigma$-finite measure space with a basis of sets $\mathcal{E}$. Let $V$ be a set and let $S:V\to L^0(\Omega)$ be a map. Moreover, suppose
\[
T:\bigcup_{w\in A_{p,(r,s)}(\mathcal{E})}S^{-1}(L^p_w(\Omega))\to L^0(\Omega)
\]
is a map for which there is an increasing function $\phi:[0,\infty)\to[0,\infty)$ such that for all $w\in A_{p,(r,s)}(\mathcal{E})$ and $f\in V$ with $Sf\in L^p_w(\Omega)$ we have
\begin{equation}\label{eq:strongtoweakin}
\|Tf\|_{L^{p,\infty}_w(\Omega)}\leq\phi([w]^\mathcal{E}_{p,(r,s)})\|Sf\|_{L^p_w(\Omega)}.
\end{equation}

Let $X$ be an $r$-convex and $s$-concave quasi-Banach function spaces over $(\Omega,|\cdot|)$ for which the following conditions hold:
\begin{itemize}
\item If $p\neq r$,
\[
M^\mathcal{E}:X_{r,s}\to X_{r,s}
\]
is bounded;
\item If $p\neq s$, 
\[
M^\mathcal{E}:(X_{r,s})'\to (X_{r,s})'
\]
is bounded.
\end{itemize}
Then $Tf$ is well-defined for any $f\in V$ with $Sf\in X$ with
\begin{equation}\label{eq:strongtoweakout}
\|Tf\|_{wX}\leq 2^{\frac{1}{r}-\frac{1}{s}}\phi(2^{\frac{1}{r}-\frac{1}{s}}\|M^\mathcal{E}\|_{X_{r,s}\to X_{r,s}}^{\frac{1}{r}-\frac{1}{p}}\|M^\mathcal{E}\|_{(X_{r,s})'\to (X_{r,s})'}^{\frac{1}{p}-\frac{1}{s}})\|Sf\|_{X}.
\end{equation}
\end{theorem}
\begin{proof}
Fix $\lambda>0$ and define
\[
T_\lambda:\bigcup_{w\in A_{p,(r,s)}(\mathcal{E})}S^{-1}(L^p_w(\Omega))\to L^0(\Omega)
\]
by
\[
T_\lambda f:=\lambda\ind_{\{x\in\Omega:|Tf(x)|>\lambda\}}.
\]
Then, by \eqref{eq:strongtoweakin}, we have
\[
\|T_\lambda f\|_{L^p_w(\R^d)}\leq\|Tf\|_{L^{p,\infty}_w(\R^d)}\leq\phi([w]^\mathcal{E}_{p,(r,s)})\|Sf\|_{L^p_w(\Omega)}
\]
Thus, by Theorem~\ref{thm:aprsextrapolation} we have
\begin{align*}
\lambda&\|\ind_{\{x\in\Omega:|Tf(x)|>\lambda\}}\|_X
=\|T_\lambda f\|_{X}\\
&\leq 2^{\frac{1}{r}-\frac{1}{s}}\phi(2^{\frac{1}{r}-\frac{1}{s}}\|M^\mathcal{E}\|_{X_{r,s}\to X_{r,s}}^{\frac{1}{r}-\frac{1}{p}}\|M^\mathcal{E}\|_{(X_{r,s})'\to (X_{r,s})'}^{\frac{1}{p}-\frac{1}{s}})\|Sf\|_{X}.
\end{align*}
The result follows by taking a supremum over $\lambda>0$.
\end{proof}

\subsection{Weighted Banach function spaces}\label{subsec:weightedbfs}
In the paper \cite{CMM22}, various extrapolation theorems in weighted Banach function spaces are presented. In this section we compare our results to theirs.

The weighted Banach function spaces that are considered in this work are defined as follows. Consider a non-atomic $\sigma$-finite measure space $(\Omega,\mu)$. Then, for a weight $v\in L^0(\Omega,\mu)$ consider the new measure $v=v\,\mathrm{d}\mu$ on $\Omega$. They then define a Banach function space over $(\Omega, v)$ as a space $X_v\subseteq L^0(\Omega,v)$ with a norm $\|\cdot\|_{X_v}$ satisfying the ideal property, the Fatou property, and the following property:
\begin{enumerate}[(i)]
\item\label{it:cmmsaturation} If $E\subseteq\Omega$ is a measurable set with $v(E)<\infty$, then $\ind_E\in X_v$ and there is a constant $C_E>0$ such that $\int_E\!|f|v\,\mathrm{d}\mu\leq C_E\|f\|_{X_v}$ for all $f\in X_v$.
\end{enumerate}
Noting that $L^0(\Omega,v)=L^0(\Omega,\mu)$, we claim that such a space is also a Banach function space in our sense over the (unweighted) measure space $(\Omega,\mu)$.
\begin{proposition}\label{prop:cmmbfscomp}
The weighted Banach function spaces $X_v$ in the sense of \cite{CMM22} are Banach function spaces in our sense over both the weighted measure space $(\Omega,v)$ and the unweighted measure space $(\Omega,\mu)$.
\end{proposition}
\begin{proof}
Note that since the measure $v$ and $\mu$ have the same zero-sets, both the notions of the ideal property and the Fatou property do not depends on whether one considers the measure $v$ or $\mu$. Hence, it remains to check the saturation property. We first check that this holds in $(\Omega,v)$. If $E\subseteq\Omega$ is a measurable set with $v(E)>0$, then, since $(\Omega,v)$ is again $\sigma$-finite measure space (see \cite[Page~10]{CMM22}), there is a set $U\subseteq\Omega$ with $0<v(U\cap E)<\infty$. Setting $F:=U\cap E\subseteq E$, it follows from \ref{it:cmmsaturation} that $\ind_F\in X$, proving the result.

As for the saturation property in $(\Omega,\mu)$, let $E\subseteq \Omega$ be a measurable set with $\mu(E)>0$. Since $(\Omega,\mu)$ is $\sigma$-finite, there is a measurable set $U\subseteq\Omega$ with $\mu(U)<\infty$ and $\mu(U\cap E)>0$. Since we then also have $v(U\cap E)>0$, we can use the saturation property from $(\Omega,v)$ to find an $F\subseteq E\cap U$ with $v(F)>0$ and $\ind_F\in X$. The assertion follows.
\end{proof}
Note that we did not require the second part of property \ref{it:cmmsaturation} to obtain the saturation property. We conclude that the Banach function spaces we consider in this work include the ones considered in \cite{CMM22}.

Conversely, one might wonder if the spaces we are considering are more general. As it turns out, this is indeed the case. Setting $u=v=1$, property \ref{it:cmmsaturation} is equivalent to saying that for every set $E\subseteq\Omega$ with $|E|<\infty$ we have $\ind_E\in X$ and $\ind_E\in X'$ (where $C_E=\|\ind_E\|_{X'}$). If one assumes that $M^\mathcal{E}:X\to X$ is bounded, then by Lemma~\ref{lem:indE} this is the case for all $E\in\mathcal{E}$, however one cannot conclude anything for general sets of finite measure. As a matter of fact, it is shown in \cite[Example~3.3]{ST15} that the Morrey space $\mathcal{L}^{p,q}(\R^d)$ does not satisfy the second property in \ref{it:cmmsaturation}, i.e., that $\ind_E\in (\mathcal{L}^{p,q}(\R^d))'$ for all sets $E\subseteq\R^d$ of finite measure. Hence, Morrey spaces do not fall under the framework of \cite{CMM22}, whereas we obtain several result for these spaces.

Given a space $Y$ that is a Banach function space in our sense over both $(\Omega,\mu)$ and $(\Omega,v)$, we denote the K\"othe dual with respect to $(\Omega,\mu)$ by $Y'$ and with respect to $(\Omega,v)$ by $Y^\dag$.

The following result was proven in \cite{CMM22}:
\begin{theorem}[{\cite[Theorem~3.1]{CMM22}}]
Let $(\Omega,\mu)$ be a non-atomic $\sigma$-finite measure space with a basis of sets $\mathcal{E}$ and let $u,v\in L^0(\Omega)$ be weights. Let $p\in[1,\infty)$ and let $\mathcal{F}\subseteq L^0(\Omega)_+\times L^0(\Omega)_+$ be such that there is an increasing function $\phi:[1,\infty)\to[1,\infty)$ so that for all $f,g\in\mathcal{F}$ we have
\[
\|f\|_{L^p_w(\Omega)}\leq\phi([w]_p^\mathcal{E})\|g\|_{L^p_w(\Omega)}.
\]
Then for every weighted Banach function space $X_v$ for which $M^\mathcal{E}:X_v(u)\to X_v(u)$ is bounded when $p>1$ and $(M^\mathcal{E})_v:X_v^\dag(u^{-1})\to X_v^\dag(u^{-1})$ is bounded, we have
\begin{equation}\label{eq:cmmconc1}
\|f\|_{X_v(u)}\leq C_0\|g\|_{X_v(u)}
\end{equation}
for all $(f,g)\in\mathcal{F}$ and
\begin{equation}\label{eq:cmmconc2}
\big\|\|(f_n)_{n\in\N}\|_{\ell^p(\N)}\big\|_{X_v(u)}\leq C_0\big\|\|(g_n)_{n\in\N}\|_{\ell^p(\N)}\big\|_{X_v(u)}
\end{equation}
for all sequences $(f_n,g_n)_{n\in\N}$ in $\mathcal{F}$,
with
\[
C_0=2^{3+\frac{2}{p'}}\phi\big(2\|M^\mathcal{E}\|_{X_v(u)\to X_v(u)}^{\frac{1}{p'}}\|(M^\mathcal{E})_v\|^{\frac{1}{p}}_{X_v^\dag(u^{-1})\to X_v^\dag(u^{-1})}\big).
\]
Moreover, if $\mathcal{E}$ is a Muckenhoupt basis, i.e., a basis for which for every $q\in(1,\infty)$ and $w\in A_q(\mathcal{E})$ we have that $M^\mathcal{E}:L^q_w(\Omega)\to L^q_w(\Omega)$ is bounded, then \eqref{eq:cmmconc2} holds with $p$ replaced by any $q\in(1,\infty)$ and $C_0$ replaced by a constant depending on $q$.
\end{theorem}
As we explained in Remark~\ref{rem:extrapolationpairs}, we can also obtain this formulated in terms of pairs of functions for our results.

In this case, by Proposition~\ref{prop:cmmbfscomp} we completely recover this result and extend it to a more general class of Banach function spaces. Indeed we can take $(\Omega,|\cdot|)=(\Omega,\mu)$ and $X=X_v(u)$. Then by Proposition~\ref{prop:bfsu} and Proposition~\ref{prop:bfsv} we have
\[
\|M^\mathcal{E}\|_{X'\to X'}=\|(M^\mathcal{E})_v\|_{X_v^\dag(u^{-1})\to X_v^\dag(u^{-1})},
\]
so by Theorem~\ref{thm:apextrapolation} we recover \eqref{eq:cmmconc1}, as a matter of fact, with the $2^{3+\frac{2}{p'}}$ in $C_0$ replaced by the smaller constant $2$. Obtaining \eqref{eq:cmmconc2} follows from the Fubini trick explained in Subsection~\ref{subsection:fubini} and the final assertion about the Muckenhoupt basis follows exactly as explained in Remark~\ref{rem:vvq}.

In \cite[Theorem~3.67]{CMM22} a limited range variant of Banach function space extrapolation is presented. This result is an extension of \cite[Proposition~5.9]{CW17} and the conclusion is formulated in a similar way. They show that for extrapolation pairs $(f,g)\in\mathcal{F}$ under the initial assumption
\begin{equation}\label{eq:cmmlr1}
\|f\|_{L^p_w(\Omega)}\leq\phi([w]^\mathcal{E}_{p,(r,s)})\|g\|_{L^p_w(\Omega)}
\end{equation}
for $p\in[r,s]$, $1\leq r<s<\infty$, one can extrapolate to a limited class of weighted Banach function spaces $X_v(u)$, as long as one assumes that $\mathcal{E}$ is a Muckenhoupt basis. The reason they require $\mathcal{E}$ to be a Muckenhoupt basis is because they first have to extrapolate \eqref{eq:cmmlr1} to a very precise $p^\ast\in(r,s)$ under which $u$, $v$, and $X_v$ satisfy various properties.

Our limited range extrapolation theorem, Theorem~\ref{thm:aprsextrapolation}, does not require $\mathcal{E}$ to be a Muckenhoupt basis, since we get the full result we require from our initial exponent $p\in[r,s]$. We also do not require our exponents to be in the Banach range $1\leq r<s<\infty$, since our results are for general quasi-Banach spaces. In the case where $X=L^{p(\cdot)}_v(\R^d)$ is a variable Lebesgue space, \cite[Theorem~3.67]{CMM22} reduces to \cite[Proposition~5.9]{CW17}. In this latter work it was shown that their extrapolation result only allows for a small range of weights within the class $v\in A_{p(\cdot),(r,s)}$, see \cite[Example~2.17]{CW17}, whereas we obtain the entire class, see Theorem~\ref{thm:mrsboundvariablelebesgue}.

Next, we discuss some of the result in the paper \cite{CO22}, which considers the same weighted Banach function spaces as \cite{CMM22}. Many of the results in this work rely on the properties of the weights $u$ and $v$, and are therefore outside of the scope of the intrinsic study of Banach function spaces in this current work. However, to illustrate how we could recover some of the results from \cite{CO22} using our intrinsic methods, we give an alternative approach to one of their results. For this, for a weight $v$ we define the weighted maximal operator in $\R^d$ by
\[
M^v f(x)=\sup_{Q\ni x}\frac{1}{v(Q)}\int_Q\!|f|v\,\mathrm{d}x,
\]
which is the maximal operator on $(\R^d,v\mathrm{d}x)$ with respect to the basis of cubes.
\begin{theorem}[{\cite[Theorem~3.3]{CO22}}]
Let $v\in A_1$ and let $X_v$ be a weighted Banach function space over $\R^d$. Let $p\in(0,\infty)$ and suppose $\mathcal{F}$ is an extrapolation family for which there is an increasing function $\phi:[1,\infty)\to[1,\infty)$ such that
\[
\|f\|_{L^p_w(\R^d)}\leq\phi([w^p]_{A_1})\|g\|_{L^p_w(\R^d)}
\]
for all $(f,g)\in\mathcal{F}$ and $w^p\in A_1$.

Then if $M^v:(X_v)^\dag(u^{-p})\to (X_v)^\dag(u^{-p})$ is bounded, we have
\[
\|f\|_{(X_v)^{\frac{1}{p}}(u)}\leq 4^{\frac{1}{p}}\phi(2\|M^v\|_{(X_v)^\dag(u^{-p})\to (X_v)^\dag(u^{-p})}[v]_{A_1})\|g\|_{(X_v)^{\frac{1}{p}}(u)}
\]
for all $(f,g)\in\mathcal{F}$.
\end{theorem}

We consider the pairs $(|Tf|,|Sf|)$ and we show how we can recover this from the first case in Theorem~\ref{thm:pconcorpconvext}. Indeed, we choose $X=(X_v)^{\frac{1}{p}}(u)=(X_v(u^p))^{\frac{1}{p}}$, which is a $p$-convex quasi-Banach function space, to obtain
\begin{equation}\label{eq:a1comparisonresult}
\|(Tf)u\|_{(X_v)^{\frac{1}{p}}}\leq 2^{\frac{1}{p}}\phi(2\|M\|_{(X_v(u^p))'\to (X_v(u^p))'})\|(Sf)u\|_{(X_v)^{\frac{1}{p}}},
\end{equation}
whenever $M:(X_v(u^p))'\to (X_v(u^p))'$ is bounded. We now need the following lemma:
\begin{lemma}
Let $v\in A_1$ and suppose $M^v:(X_v)^\dag(u^{-p})\to (X_v)^\dag(u^{-p})$ is bounded. Then $M:(X_v(u^p))'\to (X_v(u^p))'$ is bounded with
\[
\|M\|_{(X_v(u^p))'\to (X_v(u^p))'}\leq[v]_{A_1}\|M^v\|_{(X_v)^\dag(u^{-p})\to (X_v)^\dag(u^{-p})}.
\]
\end{lemma}
\begin{proof}
By Proposition~\ref{prop:bfsu} and Proposition~\ref{prop:bfsv} we have
\begin{equation}\label{eq:mveqa1}
\|M\|_{(X_v(u^p))'\to (X_v(u^p))'}=\|M_v\|_{X_v^\dag(u^{-p})\to X_v^\dag(u^{-p})}.
\end{equation}
Using the equality $\langle f v\rangle_{1,Q}=\langle f\rangle^v_{1,Q}\langle v\rangle_{1,Q}$, where the superscript $v$ denotes taking an average with respect to the measure $v\,\mathrm{d}x$, we have
\[
M_vf=M(fv)v^{-1}\leq (M^v f)(Mv)v^{-1}\leq [v]_{A_1}M^v f.
\]
Combining this with \eqref{eq:mveqa1} proves the result.
\end{proof}
Combining the above lemma with \eqref{eq:a1comparisonresult} now recovers \cite[Theorem~3.3]{CO22}.

As a final remark, since with the same proof as above we have $Mf\leq [v^{-1}]_{A_1}(M^{v^{-1}})_v f$, where $(M^{v^{-1}})_v f=M^{v^{-1}}(fv)v^{-1}$, we can use the second case in Theorem~\ref{thm:pconcorpconvext} to obtain the following dual result:
\begin{theorem}
Let $v^{-1}\in A_1$ and let $X_v$ be a weighted Banach function space over $\R^d$. Let $p\in(1,\infty]$, let $V$ be a set and let $S:V\to L^0(\Omega)$ be a map. Moreover, suppose
\[
T:\bigcup_{w^{-p'}\in A_1}S^{-1}(L^p_w(\Omega))\to L^0(\Omega)
\]
is a map for which there is an increasing function $\phi:[1,\infty)\to[1,\infty)$ such that
\[
\|Tf\|_{L^p_w(\R^d)}\leq\phi([w^{-p'}]_{A_1})\|Sf\|_{L^p_w(\R^d)}
\]
for all $f\in L^p_w(\R^d)$ and all $w^{-p'}\in A_1$.

Then if $(M^{v^{-1}})_v:X_v(u^{p'})\to X_v(u^{p'})$ is bounded, we have
\[
\|(Tf)u\|_{\big([(X_v)']^{\frac{1}{p'}}\big)'}\leq 2^{\frac{1}{p'}}\phi(2\|(M^{v^{-1}})_v\|_{X_v(u^{p'})\to X_v(u^{p'})}[v^{-1}]_{A_1})\|(Sf)u\|_{\big([(X_v)']^{\frac{1}{p'}}\big)'}
\]
for all $f\in V$ with $Sf\in \big([(X_v)']^{\frac{1}{p'}}\big)'(u)$.
\end{theorem}

\section{Main theorems for operators in the off-diagonal and two-weight setting}\label{sec:mainthm}

In this section we prove a generalized version of Theorem~\ref{thm:A} which extends our previously obtained results to an off-diagonal and two-weight setting, and then obtain several corollaries.

\begin{theorem}\label{thm:mainop}
Let $(\Omega,|\cdot|)$ be a $\sigma$-finite measure space and let $\mathcal{E}$ be a basis of sets in $\Omega$. Let $\alpha\in\R$ and suppose $r_1,r_2\in(0,\infty)$, $s_1\in(0,\infty]$, $\frac{1}{s_2}\in\R$ satisfy $\frac{1}{r_1}>\frac{1}{s_1}$, $\frac{1}{r_2}>\frac{1}{s_2}$
and
\[
\frac{1}{r_2}-\frac{1}{r_1}=\frac{1}{s_2}-\frac{1}{s_1}=\alpha.
\]
Let $p_1,p_2\in(0,\infty]$ with $\frac{1}{p_1}\in[\frac{1}{s_1},\frac{1}{r_1}]$, $\frac{1}{p_2}\in[\frac{1}{s_2},\frac{1}{r_2}]$ and $\frac{1}{p_2}-\frac{1}{p_1}=\alpha$. Let $V$ be a set and let $S:V\to L^0(\Omega)$ be a map. Moreover, suppose that we are in one of the following situations:
\begin{enumerate}[(i)]
\item\label{it:oneweightassumption}
We have a map
\[
T:\bigcup_{w\in A_{\vec{p},(\vec{r},\vec{s})}(\mathcal{E})}S^{-1}(L^{p_1}_w(\Omega))\to L^0(\Omega)
\]
for which there is an increasing function $\phi:[0,\infty)\to[0,\infty)$ such for all weights $w\in A_{\vec{p},(\vec{r},\vec{s})}(\mathcal{E})$ we have
\[
\|Tf\|_{L^{p_2}_w(\Omega)}\leq\phi([w]^\mathcal{E}_{\vec{p},(\vec{r},\vec{s})})\|Sf\|_{L^{p_1}_w(\Omega)}
\]
for all $f\in V$ with $Sf\in L^{p_1}_w(\Omega)$.
\item\label{it:twoweightassumption} We have a map
\[
T:\bigcup_{(w_1,w_2)\in A_{\vec{p},(\vec{r},\vec{s})}(\mathcal{E})}S^{-1}(L^{p_1}_{w_1}(\Omega))\to L^0(\Omega)
\]
for which there is an increasing function $\phi:[0,\infty)\to[0,\infty)$ such for all weights $(w_1,w_2)\in A_{\vec{p},(\vec{r},\vec{s})}(\mathcal{E})$ we have
\[
\|Tf\|_{L^{p_2}_{w_2}(\Omega)}\leq\phi([w_1,w_2]^\mathcal{E}_{\vec{p},(\vec{r},\vec{s})})\|Sf\|_{L^{p_1}_{w_1}(\Omega)}
\]
for all $f\in V$ with $Sf\in L^{p_1}_{w_1}(\Omega)$.
\end{enumerate}
Then $T$ satisfies the respective conclusions:
\begin{enumerate}[(I)]
\item\label{it:oneweight} Let $X$, $Y$ be a pair of quasi-Banach function spaces over $\Omega$ that are respectively $r_1$- and $r_2$-convex, $s_1$- and $s_2$-concave, satisfy $[(X^{r_1})']^{\big(\frac{s_1}{r_1}\big)'}=[(Y^{r_2})']^{\big(\frac{s_2}{r_2}\big)'}$, and both of the following conditions hold:
\begin{itemize}
\item If $p_1\neq r_1$,
\[
M^\mathcal{E}:X_{r_1,s_1}\to X_{r_1,s_1}
\]
is bounded;
\item If $p_1\neq s_1$,
\[
M^\mathcal{E}:(X_{r_1,s_1})'\to (X_{r_1,s_1})'
\]
is bounded.
\end{itemize}
Then $Tf$ is well-defined for all $f\in V$ with $Sf\in X$ and
\[
\|Tf\|_{Y}\leq 2^{\frac{1}{r_1}-\frac{1}{s_1}}\phi\big(2^{\frac{1}{r_1}-\frac{1}{s_1}}\|M^\mathcal{E}\|_{X_{r_1,s_1}\to X_{r_1,s_1}}^{\frac{1}{r_1}-\frac{1}{p_1}}\|M^\mathcal{E}\|_{(X_{r_1,s_1})'\to (X_{r_1,s_1})'}^{\frac{1}{p_2}-\frac{1}{s_2}}\big)\|Sf\|_X.
\]
\item\label{it:twoweights} Let $X$, $Y$ be a pair of quasi-Banach function spaces over $\Omega$ that are respectively $r_1$- and $r_2$-convex, $s_1$- and $s_2$-concave, and both of the following conditions hold:
\begin{itemize}
\item If $p_1\neq r_1$, there exists a positive isometric linear isomorphism
\[
L_1:X_{r_1,s_1}\to Y_{r_2,s_2},
\]
and
\[
M^\mathcal{E}:X_{r_1,s_1}\to Y_{r_2,s_2}
\]
is bounded;
\item If $p_2\neq s_2$, there exists a positive isometric linear isomorphism
\[
L_2:(Y_{r_2,s_2})'\to (X_{r_1,s_1})',
\]
and
\[
M^\mathcal{E}:(X_{r_1,s_1})'\to (X_{r_1,s_1})'
\]
is bounded.
\end{itemize}
Then $Tf$ is well-defined for all $f\in V$ with $Sf\in X$ and
\[
\|Tf\|_{Y}\leq 2^{\frac{1}{r_1}-\frac{1}{s_1}}\phi\big(2^{\frac{1}{r_1}-\frac{1}{s_1}}\|M^\mathcal{E}\|_{X_{r_1,s_1}\to Y_{r_2,s_2}}^{\frac{1}{r_1}-\frac{1}{p_1}}\|M^\mathcal{E}\|_{(Y_{r_2,s_2})'\to (X_{r_1,s_1})'}^{\frac{1}{p_2}-\frac{1}{s_2}}\big)\|Sf\|_X.
\]
\end{enumerate}
\end{theorem}

\begin{remark}\label{remark:offdiagonalbfsequality}
The condition $[(X^{r_1})']^{\big(\frac{s_1}{r_1}\big)'}=[(Y^{r_2})']^{\big(\frac{s_2}{r_2}\big)'}$ in conclusion \ref{it:oneweight} is equivalent to $X_{r_1,s_1}=Y_{r_2,s_2}$. One can also rewrite this condition as $\big[(X^{r_1})'\big]^{\frac{1}{r_1}}=\big[(Y^{r_2})'\big]^{\frac{1}{r_2}}$, or
\[
Y=\Big[\Big(\big[(X^{r_1})'\big]^{\frac{r_2}{r_1}}\Big)'\Big]^{\frac{1}{r_2}}.
\]
We note that this relation is satisfied when $X=L^{p_1}_w(\Omega)$ and $Y=L^{p_2}_w(\Omega)$. Indeed, since
\[
\frac{1}{r_1}-\frac{1}{p_1}=\frac{1}{r_2}-\alpha-\left(\frac{1}{p_2}-\alpha\right)=\frac{1}{r_2}-\frac{1}{p_2},
\]
we have
\[
\big[(X^{r_1})'\big]^{\frac{1}{r_1}}=L^{\frac{1}{\frac{1}{r_1}-\frac{1}{p_1}}}_{w^{-1}}(\Omega)=L^{\frac{1}{\frac{1}{r_2}-\frac{1}{p_2}}}_{w^{-1}}(\Omega)=\big[(Y^{r_2})'\big]^{\frac{1}{r_2}}.
\]
\end{remark}

\begin{remark}
In \ref{it:twoweights} if we additionally assume that $X_{r_1,s_1}$ and $Y_{r_2,s_2}$ are order-continuous, then we only require the map $L_1$, since in this case we have $(X_{r_1,s_1})'=(X_{r_1,s_1})^\ast$, $(Y_{r_2,s_2})'=(Y_{r_2,s_2})^\ast$, and can thus take $L_2:=L_1^\ast$.
\end{remark}

\begin{proof}[Proof of Theorem~\ref{thm:mainop}]
We need only prove \ref{it:twoweights}, since, by Theorem~\ref{thm:mainabs}, \ref{it:oneweight} follows from applying \ref{it:twoweights} with $L_1$ and $L_2$ equal to the identity operators.

Let $f\in V$ with $Sf\in X$ and $g\in \big[(Y^{r_2})'\big]^{\frac{1}{r_2}}$. By Corollary~\ref{cor:factorization}, there are $0\leq h\in (X_{r_1,s_1})^{\frac{1}{r_1}-\frac{1}{s_1}}$, $0\leq k\in L^{s_1}(\Omega)$ with
\[
|Sf|=hk,\quad \|Sf\|_{X}=\|h\|_{(X_{r_1,s_1})^{\frac{1}{r_1}-\frac{1}{s_1}}}\|k\|_{L^{s_1}(\Omega)}.
\]
We apply Theorem~\ref{thm:mainabs} with $\frac{1}{q_1}=\frac{1}{p_1}-\frac{1}{s_1}$, $\frac{1}{q_2}=\frac{1}{r_2}-\frac{1}{p_2}$, $\frac{1}{q}=\frac{1}{r_1}-\frac{1}{s_1}$, $X_1=(X_{r_1,s_1})^{\frac{1}{r_1}-\frac{1}{s_1}}$, $X_2=\big[(Y^{r_2})'\big]^{\frac{1}{r_2}}$, $f_1=h$, $f_2=g$, to obtain weights $W_1$, $W_2$ satisfying \eqref{eq:mainabs1} and \eqref{eq:mainabs2}. Setting $w_1:=W_1$,  $w_2:=W_2^{-1}$, we obtain

\begin{align*}
&\|(Tf)g\|_{L^{r_2}(\R^d)}\leq\phi([w_1,w_2]^\mathcal{E}_{\vec{p},(\vec{r},\vec{s})})\|Sf\|_{L^{p_1}_{w_1}(\Omega)}\|g\|_{L^{\frac{1}{\frac{1}{r_2}-\frac{1}{p_2}}}_{w_2^{-1}}(\Omega)}\\
&\leq \phi([w_1,w_2]^\mathcal{E}_{\vec{p},(\vec{r},\vec{s})})\|h\|_{L^{\frac{1}{\frac{1}{p_1}-\frac{1}{s_1}}}_{w_1}(\Omega)}\|k\|_{L^{s_1}(\R^d)}\|g\|_{L^{\frac{1}{\frac{1}{r_2}-\frac{1}{p_2}}}_{w_2^{-1}}(\Omega)}\\
&=\phi\Big(\sup_{E\in\mathcal{E}}\langle W_1^{-1}\rangle_{q_2,E}\langle W_2^{-1}\rangle_{q_1,E}\Big)\|f_1W_1\|_{L^{q_1}(\Omega)}\|f_2W_2\|_{L^{q_2}(\Omega)}\|k\|_{L^{s_1}(\Omega)}\\
&\leq2^{\frac{1}{r_1}-\frac{1}{s_1}}\phi\big(2^{\frac{1}{r_1}-\frac{1}{s_1}}\|M^\mathcal{E}\|_{X_{r_1,s_1}\to Y_{r_2,s_2}}^{\frac{1}{r_1}-\frac{1}{p_1}}\|M^\mathcal{E}\|_{(Y_{r_2,s_2})'\to (X_{r_1,s_1})'}^{\frac{1}{p_2}-\frac{1}{s_2}}\big)
\|h\|_{X_1}\|k\|_{L^{s_1}(\Omega)}\|g\|_{\big[(Y^{r_2})'\big]^{\frac{1}{r_2}}}.
\end{align*}
Since $\|h\|_{X_1}\|k\|_{L^{s_1}(\Omega)}=\|Sf\|_X$, the result now follows from the fact that
\[
\|Tf\|_{ Y}=\sup_{ \|g\|_{\big[(Y^{r_2})'\big]^{\frac{1}{r_2}}}=1}\|(Tf)g\|_{L^{r_2}(\Omega)}.
\]
\end{proof}

Next, we prove Theorem~\ref{thm:A}.
\begin{proof}[Proof of Theorem~\ref{thm:A}]
The first part of the theorem follows from setting $V=L^0(\R^d)$ and $Sf:=f$ in the case \ref{it:oneweightassumption}$\Rightarrow$\ref{it:oneweight}.

For the result in Lorentz spaces, note that if $X=L^{p_1,q_1}_v(\R^d)$, $Y=\big[L^{p_2,q_2}_{v^{\frac{p_1}{p_2}}}(\R^d)\big](v^{1-\frac{p_1}{p_2}})$, then, as computed in Subsection~\ref{subsec:lorentz}, we have
\[
Y=\Big[\Big(\big[(X^{r_1})'\big]^{\frac{r_2}{r_1}}\Big)'\Big]^{\frac{1}{r_2}}.
\]
and hence, the result follows from Remark~\ref{remark:offdiagonalbfsequality} and Theorem~\ref{thm:mrsboundlorentz}.

For the remaining two results, we first make the observation that for any $p_1\in[r_1,s_1]$, $\frac{1}{p_2}\in[\frac{1}{s_2},\frac{1}{r_2}]$ we have
\[
\frac{1}{(p_1)_{r_1,s_1}}=\frac{\frac{1}{p_1}-\frac{1}{s_1}}{\frac{1}{r_1}-\frac{1}{s_1}}=\frac{\frac{1}{p_2}-\frac{1}{s_2}}{\frac{1}{r_2}-\frac{1}{s_2}}=\frac{1}{(p_2)_{r_2,s_2}}.
\]

Hence, if $X=L^{p_1(\cdot)}_v(\R^d)$, $Y=L^{p_2(\cdot)}_v(\R^d)$, then, from the computation in Section~\ref{subsec:variablelebesgue},
\[
X_{r_1,s_1}=L^{(p_1(\cdot))_{r_1,s_1}}_{v^{\frac{1}{\frac{1}{r_1}-\frac{1}{s_1}}}}(\R^d)=L^{(p_2(\cdot))_{r_2,s_2}}_{v^{\frac{1}{\frac{1}{r_2}-\frac{1}{s_2}}}}(\R^d)=Y_{r_2,s_2}
\]
and if $X=\mathcal{L}_v^{p_1,q_1}(\R^d)$, $Y=\mathcal{L}^{p_2,q_2}_v(\R^d)$, then, from the computation in Section~\ref{subsec:morreyspaces},
\[
X_{r_1,s_1}=\mathcal{L}^{(p_1)_{r,s},(q_1)_{r,s}}_{v^{\frac{1}{\frac{1}{r_1}-\frac{1}{s_1}}}}(\R^d)=\mathcal{L}^{(p_2)_{r,s},(q_2)_{r,s}}_{v^{\frac{1}{\frac{1}{r_2}-\frac{1}{s_2}}}}(\R^d)=Y_{r_2,s_2}.
\]
Thus, the result in variable Lebesgue spaces follows Theorem~\ref{thm:mrsboundvariablelebesgue} and the result for Morrey spaces follows from Theorem~\ref{thm:mrsboundmorrey}, where we note that we have to restrict to the case $\alpha=0$, $s_1=s_2=\infty$, since Morrey spaces are only $\infty$-concave. This proves the assertion.
\end{proof}

Since the abstraction in which Theorem~\ref{thm:mainop} is stated is cumbersome to unwind, we use the remainder of this section to extract more palpable versions of this result in specific situations, and, moreover, we provide examples and applications.

\subsection{Off-diagonal extrapolation in the one-weight setting}
The full-range case of the off-diagonal extrapolation theorem in Banach function spaces follows from setting $r_1=1$, $s_2=\infty$ in \ref{it:oneweight}. Writing $w\in A_{p_1,p_2}(\mathcal{E})$ when
\[
[w]^\mathcal{E}_{p_1,p_2}:=[w]^\mathcal{E}_{(p_1,p_2),\big((1,r_2),(s_1,\infty)\big)}=\sup_{E\in\mathcal{E}}\langle w^{-1}\rangle_{p_1',E}\langle w\rangle_{p_2,E}<\infty,
\]
this yields the following:

\begin{corollary}\label{cor:fullrangemain}
Let $(\Omega,|\cdot|)$ be a $\sigma$-finite measure space and let $\mathcal{E}$ be a basis of sets in $\Omega$. Let $\alpha\in[0,1)$ and let $p_1\in[1,\infty]$, $p_2\in(0,\infty]$ with $\frac{1}{p_1}-\frac{1}{p_2}=\alpha$. Let $V$ be a set and let $S:V\to L^0(\Omega)$ be a map. Moreover, suppose
\[
T:\bigcup_{w\in A_{p_1,p_2}(\mathcal{E})}S^{-1}(L^{p_1}_w(\Omega))\to L^0(\Omega)
\]
is a map for which there is an increasing function $\phi:[0,\infty)\to[0,\infty)$ such that
\[
\|Tf\|_{L^{p_2}_w(\Omega)}\leq\phi([w]^\mathcal{E}_{p_1,p_2})\|Sf\|_{L^{p_1}_w(\Omega)}.
\]
for all $w\in A_{p_1,p_2}(\mathcal{E})$ and $f\in V$ with $Sf\in L^{p_1}_w(\Omega)\cap V$.

Let $X$ be a $\frac{1}{\alpha}$-concave Banach function spaces over $\Omega$ such that the following holds:
\begin{itemize}
\item If $p_1\neq 1$,
\[
M^\mathcal{E}:\big[(X')^{\frac{1}{1-\alpha}}\big]'\to \big[(X')^{\frac{1}{1-\alpha}}\big]'
\]
is bounded;
\item If $p_2\neq \infty$,
\[
M^\mathcal{E}:(X')^{\frac{1}{1-\alpha}}\to (X')^{\frac{1}{1-\alpha}}
\]
is bounded,
\end{itemize}
Then $Tf$ is well-defined for all $f\in V$ with $Sf\in X$ and
\begin{align*}
\|&Tf\|_{\big(\big[(X')^{\frac{1}{1-\alpha}}\big]'\big)^{1-\alpha}}\\
&\leq 2^{1-\alpha} \phi\big(2^{1-\alpha}\|M^\mathcal{E}\|_{\big[(X')^{\frac{1}{1-\alpha}}\big]'\to \big[(X')^{\frac{1}{1-\alpha}}\big]'}^{\frac{1}{p_1'}}\|M^\mathcal{E}\|_{(X')^{\frac{1}{1-\alpha}}\to (X')^{\frac{1}{1-\alpha}}}^{\frac{1}{p_2}}\big)\|Sf\|_X.
\end{align*}
\end{corollary}

\begin{remark}\label{rem:gamma}
The condition $\alpha\geq 0$ coincides with the condition $p_2\geq p_1$. Unfortunately, it is unclear how to proceed in the case $p_2<p_1$, i.e., when $\alpha<0$. We note that in this case we would have $\frac{1}{s_1}=\alpha<0$, but in the proof of our extrapolation theorem we require $\frac{1}{s_1}\geq 0$ to use the factorization result Corollary~\ref{cor:factorization}. Fortunately, this is merely a theoretical quandary; as we will see below, in the application of the Riesz potential we are in the situation $p_2\geq p_1$ and our result applies without a problem.
\end{remark}

\begin{remark}
The condition $\frac{1}{p_1}-\frac{1}{p_2}=\alpha$ restricts the range of $p_1$ and $p_2$. Indeed this condition is equivalent to
\[
\frac{1}{p_1'}+\frac{1}{p_2}=1-\alpha.
\]
From this we deduce that $\frac{1}{p_1'}\leq 1-\alpha$, or $p_1\in[1,\frac{1}{\alpha}]$, and $\frac{1}{p_2}\leq 1-\alpha$, or $p_2\in[\frac{1}{1-\alpha},\infty]$. As a matter of fact, the fact that $\frac{1}{r_1}=1$ and $\frac{1}{s_2}=0$ means that $\frac{1}{r_2}=1-\alpha$ and $\frac{1}{s_1}=\alpha$, which is reflected in these ranges.
\end{remark}

For $\alpha=0$, Corollary~\ref{cor:fullrangemain} recovers Theorem~\ref{thm:apextrapolation}. By setting $\alpha=0$ in \ref{it:oneweight}, we recover Theorem~\ref{thm:aprsextrapolation}.

\bigskip

Let $\lambda\in(0,d)$. In $\R^d$, the typical examples of operators that satisfy off-diagonal bounds are the Riesz potential
\[
I_\lambda f(x):=\int_{\R^d}\!\frac{f(y)}{|x-y|^{d-\lambda}}\,\mathrm{d}y,
\]
and the associated maximal operator
\[
M_\lambda f(x):=\sup_Q|Q|^{\frac{\lambda}{d}}\langle f\rangle_{1,Q}\ind_Q(x).
\]
As was shown in \cite{MW74}, these operators satisfy the bounds
\[
I_\lambda, M_\lambda:L^{p_1}_w(\R^d)\to L^{p_2}_w(\R^d)
\]
for $p_1\in(1,\frac{d}{\lambda})$, $p_2\in(\frac{1}{1-\frac{d}{\lambda}},\infty)$ satisfying
\[
\frac{1}{p_1}-\frac{1}{p_2}=\frac{\lambda}{d},
\]
and all $w\in A_{p_1,p_2}$.

Hence, we obtain the following result:
\begin{corollary}
Let $X$ be a $\frac{d}{\lambda}$-concave Banach function space over $\R^d$ for which
\[
M:\big[(X')^{\frac{1}{1-\frac{\lambda}{d}}}\big]'\to\big[(X')^{\frac{1}{1-\frac{\lambda}{d}}}\big]',
\quad M:(X')^{\frac{1}{1-\frac{\lambda}{d}}}\to (X')^{\frac{1}{1-\frac{\lambda}{d}}}
\]
are bounded. Then
\[
I_\lambda,M_\lambda:X\to\big(\big[(X')^{\frac{1}{1-\frac{\lambda}{d}}}\big]'\big)^{1-\frac{\lambda}{d}}
\]
are bounded. In particular, we have:
\begin{itemize}
\item Lorentz spaces: For all $p_1,q_1\in(1,\frac{d}{\lambda})$, $p_2$ and $q_2$ satisfying
\[
\frac{1}{p_1}-\frac{1}{p_2}=\frac{1}{q_1}-\frac{1}{q_2}=\frac{\lambda}{d},
\]
and weights $v\in A_{p_1,p_2}$, we have
\[
I_\lambda,M_\lambda:L^{p_1,q_1}_v(\R^d)\to \big[L^{p_2,q_2}_{v^{\frac{p_1}{p_2}}}(\R^d)\big](v^{1-\frac{p_1}{p_2}}).
\]
\item Variable Lebesgue spaces: For all $p_1:\R^d\to(1,\infty)$ satisfying
\[
1<\essinf p_1\leq\esssup p_1< \frac{d}{\lambda}
\]
and $p_1(\cdot)\in LH_0\cap LH_\infty$, $p_2\in L^0(\R^d)$ satisfying
\[
\frac{1}{p_1(x)}-\frac{1}{p_2(x)}=\frac{\lambda}{d},
\]
and all weights $v$ satisfying
\[
\sup_Q|Q|^{-\big(1-\frac{\lambda}{d}\big)}\|v^{-1}\ind_Q\|_{L^{p_1'(\cdot)}(\R^d)}\|v\ind_Q\|_{L^{p_2(\cdot)}(\R^d)}<\infty,
\]
we have
\[
I_\lambda,M_\lambda:L^{p_1(\cdot)}_v(\R^d)\to L^{p_2(\cdot)}_v(\R^d).
\]
\end{itemize}
\end{corollary}

We note that in weighted variable Lebesgue spaces, our result contains a larger class of weights than the one obtained in \cite[Corollary~3.7]{CW17}.

Notably, we have not included Morrey spaces in the above result. This is because we are not in the situation $\alpha=0$ and $s_1=s_2=\infty$, and hence, we are hindered by the fact that Morrey spaces are not concave. However, with the bounds we have obtained in Block spaces we can still obtain bounds for the Riesz potential as a consequence of sparse domination:
\begin{theorem}\label{thm:rieszmorrey}
For all $p_1\in(1,\frac{d}{\lambda})$, $q_1\in[p_1,\frac{d}{\lambda})$, $p_2$ and $q_2$ satisfying
\[
\frac{1}{p_1}-\frac{1}{p_2}=\frac{1}{q_1}-\frac{1}{q_2}=\frac{\lambda}{d},
\]
and weights $v$ satisfying
\[
\sup_Q|Q|^{-\big(1-\frac{\lambda}{d}-\frac{1}{p_1}+\frac{1}{q_1}\big)}\|v\ind_Q\|_{L^{q_2}(\R^d)}\|v^{-1}\ind_Q\|_{L^{p_1'}(\R^d)}<\infty,
\]
we have
\[
\|I_\lambda f\|_{\mathcal{L}_v^{p_2,q_2}(\R^d)}\lesssim_d\|f\|_{\mathcal{L}_v^{p_1,q_1}(\R^d)}
\]
for all $f\in L^{q_1}_v(\R^d)$.
\end{theorem}
For this we require a sparse form domination result involving the bisublinear maximal operator
\[
M_{r,s'}(f,g):=\sup_Q\langle f\rangle_{r,Q}\langle g\rangle_{s',Q}\ind_Q.
\]
\begin{lemma}\label{lem:rieszmorreysparse}
Let $\lambda\in(0,d)$ and $\frac{1}{r_2}:=1-\frac{\lambda}{d}$, $\frac{1}{s_1}:=\frac{\lambda}{d}$. Then for all $f\in L^\infty_c(\R^d)$, $g\in L^{s_1'}_{\loc}(\R^d)$ we have
\[
\|(I_\lambda f)g\|_{L^{r_2}(\R^d)}\lesssim_d\|M_{1,s_1'}(f,g)\|_{L^1(\R^d)}.
\]
\end{lemma}
\begin{proof}
We use the fact that there exists a sparse collection $\mathcal{S}$, i.e., a collection of cubes $Q$ satisfying the property that there are pairwise disjoint $E_Q\subseteq Q$ with $|Q|\lesssim_d|E_Q|$, such that
\[
|I_\lambda f|\lesssim_d\sum_{Q\in\mathcal{S}}|Q|^{\frac{\lambda}{d}}\langle f\rangle_{1,Q}\ind_Q
\]
for all $f\in L^\infty_c(\R^d)$, see \cite{Cr17}.

Let $g\in L^{s_1'}_{\loc}(\R^d)$. Then
\begin{align*}
\|(I_\lambda f)g\|_{L^{r_2}(\R^d)}
&=\sup_{\|h\|_{L^{\frac{d}{\lambda}}(\R^d)}=1}\|(I_\lambda)gh\|_{L^1(\R^d)}\\
&\leq\sup_{\|h\|_{L^{\frac{d}{\lambda}}(\R^d)}=1}\sum_{Q\in\mathcal{S}}|Q|^{\frac{\lambda}{d}}\langle f\rangle_{1,Q}\langle gh\rangle_{1,Q}|Q|\\
&\leq\sum_{Q\in\mathcal{S}}\langle f\rangle_{1,Q}\langle g\rangle_{s_1',Q}|Q|
\lesssim_d\sum_{Q\in\mathcal{S}}\langle f\rangle_{1,Q}\langle g\rangle_{s_1',Q}|E_Q|\\
&\leq\sum_{Q\in\mathcal{S}}\int_{E_Q}M_{1,s_1'}(f,g)\,\mathrm{d}x\leq\|M_{1,s_1'}(f,g)\|_{L^1(\R^d)}.
\end{align*}
The assertion follows.
\end{proof}

\begin{proof}[Proof of Theorem~\ref{thm:rieszmorrey}]
Since the space $L^{q_1}_v(\R^d)$ is order-continuous, the estimate
\[
\|(I_\lambda f)g\|_{L^{r_2}(\R^d)}\lesssim_d\|M_{1,s_1'}(f,g)\|_{L^1(\R^d)}
\]
from Lemma~\ref{lem:rieszmorreysparse} holds for all $f\in L^{q_1}_v(\R^d)$, $g\in L^{s_1'}(\R^d)$. Hence, for all $r_2$-convex quasi-Banach function space $Y$, all Banach function space $X$, and all $f\in L^{q_1}_v(\R^d)\cap X$, $g\in L^{s_2'}_{\loc}(\R^d)$ we have
\begin{equation}\label{eq:rieszmorrey1}
\begin{split}
\|I_\lambda f\|_Y&=\sup_{\|g\|_{[(Y^{r_2})']^{\frac{1}{r_2}}}=1}\|(I_\lambda f)g\|_{L^{r_2}(\R^d)}\\
&\lesssim_d\sup_{\|g\|_{[(Y^{r_2})']^{\frac{1}{r_2}}}=1}\|M f\|_X\|M_{s_1'}g\|_{X'}\\
&\leq\sup_{\|g\|_{[(Y^{r_2})']^{\frac{1}{r_2}}}=1}\|M\|_{X\to X}\|M_{s_1'}\|_{X'\to X'}\|f\|_X\|g\|_{X'}
\end{split}
\end{equation}
provided that 
\begin{equation}\label{eq:rieszmorrey2}
M:X\to X,\quad M_{s_1'}:X'\to X'
\end{equation}
are bounded.

Now take $X=\mathcal{L}^{p_1,q_1}_v(\R^d)$, $Y=\mathcal{L}^{p_2,q_2}_v(\R^d)$. Then we have
\[
X'=\mathcal{B}^{p_1',q_1'}_{v^{-1}}(\R^d),\quad [(Y^{r_2})']^{\frac{1}{r_2}}=(\mathcal{B}^{(\frac{p_2}{r_2})',(\frac{q_2}{r_2})'}_{v^{-r_2}}(\R^d))^{\frac{1}{r_2}}.
\]
Since $\frac{1}{r_2}-\frac{1}{p_2}=\frac{1}{p_1'}$, $\frac{1}{r_2}-\frac{1}{q_2}=\frac{1}{q_1'}$, it follows from Proposition~\ref{prop:dualofpowerofblock} that
\[
\big([(Y^{r_2})']^{\frac{1}{r_2}}\big)''=\big(\mathcal{L}^{p_1,q_1}_v(\R^d)\big)'=X'.
\]
Thus, using Proposition~\ref{prop:envelopeembedding} we have
\[
\|g\|_{X'}=\|g\|_{\big([(Y^{r_2})']^{\frac{1}{r_2}}\big)''}\leq\|g\|_{[(Y^{r_2})']^{\frac{1}{r_2}}}.
\]
Combining this with \eqref{eq:rieszmorrey1} yields
\[
\|I_\lambda f\|_Y\lesssim_d\|M\|_{X\to X}\|M_{s_1'}\|_{X'\to X'}\|f\|_X.
\]
Hence, it remains to check \eqref{eq:rieszmorrey2}.

By Theorem~\ref{thm:mrsboundmorrey} we have $M:X\to X$ whenever $v\in A_{p_1,(1,\frac{1}{\frac{1}{p_1}-\frac{1}{q_1}})}$. Since we have $v\in A_{p_1,(1,\frac{1}{\frac{1}{p_1}-\frac{1}{q_1}+\frac{1}{s_1}})}\subseteq A_{p_1,(1,\frac{1}{\frac{1}{p_1}-\frac{1}{q_1}})}$, this bound is satisfied.

For the bound $M_{s_1'}:X'\to X'$, note that by Theorem~\ref{thm:mboundblockspaces} this holds when $v^{-p_1'}\in A_{p_1'(\frac{1}{s_1'}-\frac{1}{q_1'})+1}$. This is equivalent to $v^{-1}\in A_{p_1',(\frac{1}{\frac{1}{q_2}+\frac{1}{p_1'}},\infty)}$, which is precisely our assumed condition. The assertion follows.
\end{proof}

\subsection{Endpoint extrapolation in the two-weight setting}
Setting $\alpha=0$ in \ref{it:twoweights}, we obtain the following upper endpoint two-weight extrapolation result:

\begin{corollary}\label{cor:twoweightendpointextrapolation}
Let $(\Omega,|\cdot|)$ be a $\sigma$-finite measure space and let $\mathcal{E}$ be a basis of sets in $\Omega$. Let $r\in(0,\infty)$, $s\in(0,\infty]$ with $r<s$. Let $V$ be a set and let $S:V\to L^0(\Omega)$ be a map. Moreover, suppose
\[
T:\bigcup_{(w_1,w_2)\in A_{s,(r,s)}(\mathcal{E})}S^{-1}(L^s_w(\Omega))\to L^0(\Omega)
\]
is a map for which there is an increasing function $\phi:[0,\infty)\to[0,\infty)$ such that
\[
\|Tf\|_{L^s_{w_2}(\R^d)}\leq\phi([w_1,w_2]^\mathcal{E}_{s,(r,s)})\|Sf\|_{L^s_{w_1}(\R^d)}
\]
for all $(w_1,w_2)\in A_{s,(r,s)}(\mathcal{E})$ and $f\in V$ with $Sf\in L^s_{w_1}(\R^d)$.

Let $X$, $Y$ be an $r$-convex and $s$-concave pair of quasi-Banach function space for which there is an isometric positive linear isomorphism
\[
L:X_{r,s}\to Y_{r,s}.
\]
Then $Tf$ is well-defined for all $f\in V$ with $Sf\in X$, with
\[
\|Tf\|_{Y}\leq 2^{\frac{1}{r}-\frac{1}{s}}\phi\big(2^{\frac{1}{r}-\frac{1}{s}}\|M\|_{X_{r,s}\to Y_{r,s}}^{\frac{1}{r}-\frac{1}{s}}\big)\|Sf\|_X.
\]
\end{corollary}
\begin{remark}\label{rem:asrsmineq}
We note that $(w_1,w_2)\in A_{s,(r,s)}(\mathcal{E})$ if and only if there is a $C>0$ such that
\[
M^\mathcal{E}_{\frac{1}{\frac{1}{r}-\frac{1}{s}}}(w_1^{-1})\leq C w_2^{-1},
\]
with the smallest possible constant being given by $C^{\frac{1}{r}-\frac{1}{s}}=[w_1,w_2]^\mathcal{E}_{s,(r,s)}$.
\end{remark}

\begin{example}
As a typical example of a space $Y$ for which there exists a positive isometric isomorphism $L:X\to Y$, we consider the case where $u$ is a weight, and $Y=X(u)$, which is the space of those $f\in L^0(\Omega)$ for which $fu\in X$, equipped with the norm
\[
\|f\|_{X(u)}:=\|fu\|_X.
\]
If $X$ is a Banach function space, then so is $X(u)$. Indeed, the ideal property and the Fatou property are readily checked. For the saturation property, note that if $\rho>0$ is a weak order unit in $X$, then $\rho u^{-1}>0$ is a weak order unit in $X(u)$.

In this case the map
\[
L:X\to X(u),\quad Lf:=fu^{-1}
\]
is an isometric positive linear isomorphism.

For example, if $v_1$ and $v_2$ are weights, we can take $X=L^q_{v_1}(\R^d)$. Then choosing the weight $u:=\frac{v_2}{v_1}$ yields $X(u)=L^q_{v_2}(\R^d)$.
\end{example}

\begin{example}
Let $\Omega=\R^d$ with the Lebesgue measure and let $\mathcal{E}=\mathcal{Q}$ be the collection of cubes in $\R^d$. Let $p\in(r,s)$ and let $v$, $w$ be weights. Setting $X=L^p_w(\R^d)$, $Y=L^p_v(\R^d)$, we note that
\[
X_{r,s}=L^{p_{r,s}}_{w^{\frac{1}{\frac{1}{r}-\frac{1}{s}}}}(\R^d),\quad Y_{r,s}=L^{p_{r,s}}_{v^{\frac{1}{\frac{1}{r}-\frac{1}{s}}}}(\R^d).
\]

As in the above example, noting that
\[
L:X_{r,s}\to Y_{r,s}, \quad Lf:=\Big(\frac{w}{v}\Big)^{\frac{1}{\frac{1}{r}-\frac{1}{s}}} f,\\
\]
is an isometric positive linear isomorphism, we conclude that $T$ as above satisfies
\begin{equation}\label{eq:upperexample1}
\|T\|_{L^p_w(\R^d)\to L^p_v(\R^d)}\leq 2^{\frac{1}{r}-\frac{1}{s}}\phi\big(2^{\frac{1}{r}-\frac{1}{s}}\|M\|_{X_{r,s}\to Y_{r,s}}^{\frac{1}{r}-\frac{1}{s}}\big),
\end{equation}
provided that $M:L^{p_{r,s}}_{w^{\frac{1}{\frac{1}{r}-\frac{1}{s}}}}(\R^d)\to L^{p_{r,s}}_{v^{\frac{1}{\frac{1}{r}-\frac{1}{s}}}}(\R^d)$ is bounded. As a matter of fact, in this case we have the Sawyer testing condition
\[
\|M\|_{X_{r,s}\to Y_{r,s}}^{p_{r,s}}\eqsim_{p,r,s,d} \sup_Q\frac{\langle M(w^{-p_{r,s}'}\ind_Q)^{p_{r,s}}v^{p_{r,s}}\rangle_{1,Q}}{\langle w^{-p_{r,s}'}\rangle_{1,Q}}.
\]
Since
\[
[w^{\frac{1}{\frac{1}{r}-\frac{1}{s}}},v^{\frac{1}{\frac{1}{r}-\frac{1}{s}}}]^{\frac{1}{r}-\frac{1}{s}}_{p_{r,s},(1,\infty)}=[w,v]_{p,(r,s)},
\]
a sufficient condition for this to be bounded is $(w,v)\in A_{p,(r,s)}$ and $w^{-\frac{1}{\frac{1}{r}-\frac{1}{p}}}\in A_\infty$, with
\[
\|M\|^{\frac{1}{r}-\frac{1}{s}}_{X_{r,s}\to Y_{r,s}}\leq C_{p,r,s,d}[w,v]_{p,(r,s)}[w^{-\frac{1}{\frac{1}{r}-\frac{1}{p}}}]_{A_\infty}^{\frac{1}{p}-\frac{1}{s}}.
\]
Hence, in this case, by \eqref{eq:upperexample1}, we have
\[
\|T\|_{L^p_w(\R^d)\to L^p_v(\R^d)}\lesssim_{r,s}\phi\big(C_{p,r,s,d}[w,v]_{p,(r,s)}[w^{-\frac{1}{\frac{1}{r}-\frac{1}{p}}}]_{A_\infty}^{\frac{1}{p}-\frac{1}{s}}\big).
\]
\end{example}

In the classical setting of $\Omega=\R^d$, with $\mathcal{E}=\mathcal{Q}$ is the collection of cubes in $\R^d$ and the full-range case $r=1$, $s=\infty$, a $\BMO$ extrapolation result in Banach function spaces was proven by the author and Rey \cite{NR23}. Letting $M^\sharp f:=\sup_Q\inf_{c}\langle f-c\rangle_{1,Q}\ind_Q$ denote the sharp maximal operator, for a weight $w\in L^\infty_{\loc}(\R^d)$ we can then define the space $\BMO_w(\R^d)$ through
\[
\|f\|_{\BMO_w(\R^d)}:=\|(M^\sharp f)w\|_{L^\infty(\R^d)}.
\]
The following result was proven in \cite{NR23}:
\begin{theorem}[{\cite[Theorem~3.1]{NR23}}]
Suppose
\[
T:\bigcup_{w^{-1}\in A_1} \{f\in L^\infty_w(\R^d):\text{$f$ has compact support.}\}\to L^0(\R^d)
\]
is a map for which there exists an increasing function $\phi:[0,\infty)\to[0,\infty)$ such that
\begin{equation}\label{eq:bmowin}
\|Tf\|_{\BMO_w(\R^d)}\leq\phi([w^{-1}]_1)\|f\|_{L^\infty_w(\R^d)}.
\end{equation}
Then for all Banach function spaces $X$ over $\R^d$ for which 
\[
M:X\to X,\quad M:X'\to X'
\]
are bounded, $Tf$ is well-defined for all $f\in X$ of compact support with
\[
\|Tf\|_X\leq 2C_X\phi(2\|M\|_{X\to X})\|f\|_X,
\]
where $C_X$ is the optimal constant in the Fefferman-Stein inequality
\[
\|f\|_X\leq C_X\|M^\sharp f\|_X.
\]
Moreover, if $T$ is (sub)linear and $X$ is order-continuous, then $T$ extends to a bounded map $T:X\to X$ with
\[
\|T\|_{X\to X}\leq 2C_X\phi(2\|M\|_{X\to X}).
\]
\end{theorem}
Let $\Omega=\R^d$ and let $\mathcal{E}$ be the basis of cubes. Note that by Remark~\ref{rem:asrsmineq}, the condition $(w_1,w_2)\in A_{\infty,(1,\infty)}$ is equivalent to there being a constant $C>0$ such that
\[
M(w_1^{-1})\leq Cw_2^{-1}.
\]
In turn, this is equivalent to $(w_2^{-1},w_1^{-1})\in A_1:=A_{1,(1,\infty)}$, where the optimal constant coincides with $C=[w_2^{-1},w_1^{-1}]_1$. Thus, Corollary~\ref{cor:twoweightendpointextrapolation} allows us to obtain the following two-weight version of the above theorem:
\begin{theorem}
Suppose
\[
T:\bigcup_{(w_2^{-1},w_1^{-1})\in A_1}\{f\in L^\infty_{w_1}(\R^d):\text{$f$ has compact support}\}\to L^0(\R^d)
\]
is a map for which there is an increasing function $\phi:[0,\infty)\to[0,\infty)$ such that
\[
\|Tf\|_{\BMO_{w_2}(\R^d)}\leq\phi([w_2^{-1},w_1^{-1}]_1)\|f\|_{L^\infty_{w_1}(\R^d)}
\]
for all $(w_2^{-1},w_1^{-1})\in A_1$ and compactly supported $f\in L^\infty_{w_1}(\R^d)$.

Let $X$, $Y$ be a pair of Banach function space for which $M:X\to Y$ is bounded, there is an isometric positive isometric isomorphism
\[
L:X\to Y,
\]
and for which $Y$ satisfies the Fefferman-Stein inequality
\[
\|f\|_Y\leq C_Y\|M^\sharp f\|_Y.
\]
Then $Tf$ is well-defined for all $f\in X$ with compact support, with
\[
\|Tf\|_{Y}\leq 2C_Y\phi(2\|M\|_{X\to Y})\|f\|_X.
\]
Moreover, if $T$ is (sub)linear and $X$ is order-continuous, then $T$ extends to a bounded map $T:X\to Y$ with
\[
\|T\|_{X\to Y}\leq 2C_X\phi(2\|M\|_{X\to Y}).
\]
\end{theorem}
\begin{proof}
Setting $r=1$, $s=\infty$, replacing $Tf$ by $M^\sharp(Tf)$, and taking $V=\{f\in L^0(\R^d):\text{$f$ has compact support}\}$ with $Sf=f$ in Corollary~\ref{cor:twoweightendpointextrapolation}, we conclude that
\[
\|M^\sharp(Tf)\|_Y\leq 2\phi(2\|M\|_{X\to Y})\|f\|_X
\]
for all $f\in X$ with compact support.

The final assertion follows from the fact that if $X$ is order-continuous, then $L^\infty_c(\R^d)\cap X$ is dense in $X$. Indeed, for $f\in X$, it follows from the ideal property of $X$ that $f_n:=f\ind_{\{|f|\leq n\}}\ind_{\{|x|\leq n\}}\in X\cap L^\infty_c(\R^d)$. If $X$ is order-continuous, then, since $|f-f_n|=|f|\ind_{\{|f|>n\}\cup\{|x|>n\}}\downarrow 0$, we have $\|f-f_n\|_X\to 0$, as asserted.
\end{proof}

\section{Sparse domination and multilinear extraplation in Banach function spaces}\label{sec:sparsedom}

In this section we prove a multilinear extrapolation result in Banach function spaces, motivated by the bounds that can be obtained from sparse domination.

\subsection{Sparse domination in the linear setting}

Many operators in the literature satisfy the $A_2$ bound
\begin{equation}\label{eq:TfullrangesA2}
\|Tf\|_{L^2_w(\R^d)}\leq C_d[w]_2^2\|f\|_{L^2_w(\R^d)}
\end{equation}
and, since this was shown to hold for Calder\'on-Zygmund operator by Hyt\"onen \cite{Hy12}, the idea to use sparse domination to prove weighted bounds was conceived of in \cite{Le13a} and further generalized in works such as \cite{Le16, LO20, Lo21, LLO22}. As a matter of fact, by now, the standard method of proving that an operator satisfies weighted bounds is through sparse domination.

One of the equivalent ways of writing sparse domination of an operator $T$ in form is through the estimate
\begin{equation}\label{eq:Tfullrangesd}
\|(Tf)g\|_{L^1(\R^d)}\leq C_T\|M_{1,1}(f,g)\|_{L^1(\R^d)}.
\end{equation}
Here, $M_{1,1}$ is the bisublinear Hardy-Littlewood maximal operator
\[
M_{1,1}(f,g):=\sup_Q\langle f\rangle_{1,Q}\langle g\rangle_{1,Q}\ind_Q.
\]
Whenever an operator satisfies \eqref{eq:Tfullrangesd}, it also satisfies \eqref{eq:TfullrangesA2}.

Using Theorem~\ref{thm:apextrapolation} with $p=2$ on the estimate \eqref{eq:TfullrangesA2}, we conclude that for a Banach function space $X$ for which $M:X\to X$ and $M:X'\to X'$ are bounded we have
\[
\|Tf\|_X\leq 8C_d\|M\|_{X\to X}\|M\|_{X'\to X'}\|f\|_X.
\]
We also point out that we can obtain this same bound directly from \eqref{eq:Tfullrangesd}. Indeed, we have $M_{1,1}(f,g)\leq (Mf)(Mg)$ so that
\begin{equation}\label{eq:sparsebfsfullrangeright}
\|M_{1,1}(f,g)\|_{L^1(\R^d)}\leq\|Mf\|_X\|Mg\|_{X'}\leq \|M\|_{X\to X}\|M\|_{X'\to X'}\|f\|_X\|g\|_{X'}
\end{equation}
and hence, \eqref{eq:Tfullrangesd} implies that
\begin{align*}
\|Tf\|_X
&=\sup_{\|g\|_{X'}=1}\|(Tf)g\|_{L^1(\R^d)}
\leq C_T\sup_{\|g\|_{X'}=1}\|M_{1,1}(f,g)\|_{L^1(\R^d)}\\
&\leq C_T\|M\|_{X\to X}\|M\|_{X'\to X'}\|f\|_X.
\end{align*}
As it turns out, one cannot use sparse domination to obtain bounds in Banach function spaces $X$ beyond the ones for which $M:X\to X$ and $M:X'\to X'$ are bounded. Indeed, we have the following:

\begin{proposition}\label{prop:sparseboundsfullrangeequiv}
Let $X$ be a Banach function space over $\R^d$. Then the following assertions are equivalent:
\begin{enumerate}[(i)]
\item\label{it:sparseboundfullrangeequiv1} $M:X\to X$ and $M:X'\to X'$ are bounded;
\item\label{it:sparseboundfullrangeequiv2} $M_{1,1}:X\times X'\to L^1(\R^d)$ is bounded.
\end{enumerate}
Moreover, in this case we have
\begin{equation}\label{eq:sparsebfsfullrange}
\max\big(\|M\|_{X\to X},\|M\|_{X'\to X'}\big)\lesssim_d\|M_{1,1}\|_{X\times X'\to L^1(\R^d)}\leq\|M\|_{X\to X}\|M\|_{X'\to X'}.
\end{equation}
\end{proposition}
\begin{proof}
The implication \ref{it:sparseboundfullrangeequiv1}$\Rightarrow$\ref{it:sparseboundfullrangeequiv2} as well as the second inequality in \eqref{eq:sparsebfsfullrange} follow from \eqref{eq:sparsebfsfullrangeright}.

The implication \ref{it:sparseboundfullrangeequiv2}$\Rightarrow$\ref{it:sparseboundfullrangeequiv1} follows from the fact that $M$ itself satisfies the sparse domination
\[
\|(Mf)g\|_{L^1(\R^d)}\lesssim_d \|M_{1,1}(f,g)\|_{L^1(\R^d)}
\]
so that
\[
\|Mf\|_X=\sup_{\|g\|_{X'}=1}\|(Mf)g\|_{L^1(\R^d)}\lesssim_d\|M_{1,1}\|_{X\times X'\to L^1(\R^d)}\|f\|_X
\]
and
\[
\|Mg\|_{X'}=\sup_{\|f\|_X=1}\|f(Mg)\|_{L^1(\R^d)}\lesssim_d\|M_{1,1}\|_{X\times X'\to L^1(\R^d)}\|g\|_{X'}.
\]
This proves the result.
\end{proof}
Since, generally, sparse domination yields sharper bounds than extrapolation, a natural question to ask at this point is if it is possible to obtain a bound that is better than
\[
\|T\|_{X\to X}\lesssim_d C_T\|M\|_{X\to X}\|M\|_{X'\to X'}
\]
from sparse domination. Say, we want to find the optimal increasing function $\phi$ for which
\[
\|T\|_{X\to X}\lesssim_d C_T\phi\Big(\max\big(\|M\|_{X\to X},\|M\|_{X'\to X'}\big)\Big)
\]
for any $T$ satisfying \eqref{eq:Tfullrangesd}. It was shown by Lerner, Li, and Ombrosi\footnote{Personal communication with S. Ombrosi.} that for the Hilbert transform $T=H$ and the choice $X=L^2_w(\R)$ for $w^2\in A_2$ we must have
\[
\phi(t)\gtrsim t(1+\log t).
\]
Hence, we conjecture the following:
\begin{conjecture}
Suppose $T$ satisfies \eqref{eq:Tfullrangesd}. Then for any Banach function space $X$ for which $M:X\to X$ and $M:X'\to X'$ are bounded we have
\[
\|T\|_{X\to X}\lesssim_d C_T\max\big(\|M\|_{X\to X},\|M\|_{X'\to X'}\big)\Big(1+\log\big(\max\big(\|M\|_{X\to X},\|M\|_{X'\to X'}\big)\big)\Big).
\]
\end{conjecture}

In \cite{BFP16} it was shown that a large class of operators, beyond the scope of Calder\'on-Zygmund operators, satisfy the sparse domination
\begin{equation}\label{eq:Tlimrangesd}
\|(Tf)g\|_{L^1(\R^d)}\leq C_T\|M_{r,s'}(f,g)\|_{L^1(\R^d)}
\end{equation}
for $1\leq r<s\leq\infty$, where
\[
M_{r,s'}(f,g):=\sup_Q\langle f\rangle_{r,Q}\langle g\rangle_{s',Q}\ind_Q.
\]
Moreover, they showed that this bound implies that
\[
T:L^p_w(\R^d)\to L^p_w(\R^d)
\]
for all $w\in A_{p,(r,s)}$.

Much like before, one can show that \eqref{eq:Tlimrangesd} implies that $T:X\to X$ precisely when $X$ is a Banach function space for which
\[
M:X^r\to X^r,\quad M:(X')^{s'}\to (X')^{s'}
\]
are bounded. This is the case when
\begin{equation}\label{eq:conjxrsbound}
M:X_{r,s}\to X_{r,s},\quad M:(X_{r,s})'\to (X_{r,s})'
\end{equation}
are bounded. It is not clear if the converse also holds, which we conjectured in Conjecture~\ref{con:bfsboundconjecture2}.

It was conjectured in \cite{LN22} that if \eqref{eq:Tlimrangesd} holds, then we also have
\[
\|(Tf)g\|_{L^r(\R^d)}\lesssim_d C_T\|M_{r,\frac{1}{\frac{1}{r}-\frac{1}{s}}}(f,g)\|_{L^r(\R^d)}.
\]
Then, using the fact that for an $r$-convex quasi-Banach function space $X$ we have
\[
\|Tf\|_X=\sup_{\|g\|_{[(X^r)']^{\frac{1}{r}}}=1}\|(Tf)g\|_{L^r(\R^d)},
\]
we find that this implies that $T:X\to X$ whenever $X$ is an $r$-convex quasi-Banach function space for which
\[
M:X^r\to X^r,\quad M:\big[(X^r)'\big]^{\big(\frac{s}{r}\big)'}\to[(X^r)']^{\big(\frac{s}{r}\big)'}
\]
are bounded. Again, this is the case when \eqref{eq:conjxrsbound} holds, and the converse is conjectured in Conjecture~\ref{con:bfsboundconjecture1}. An interesting point to note here is that we now no longer have the requirement that $X$ is $s$-concave. This means that for operators satisfying sparse domination, we could actually get a result in Morrey spaces. It is an open problem whether the $s$-concavity condition can also be removed in our extrapolation theorem.

\begin{theorem}\label{thm:sparseimpliesmorrey}
Let $1\leq r<s\leq\infty$ and suppose $T$ is a (sub)linear operator so that for all $f,g\in L^\infty_c(\R^d)$ we have
\[
\|(Tf)g\|_{L^1(\R^d)}\leq C_T\|M_{r,s'}(f,g)\|_{L^1(\R^d)}.
\]
Then for all $p\in(r,s)$, $q\in[p,s)$, all weights $v$ satisfying
\[
\sup_Q|Q|^{-(\frac{1}{r}-\frac{1}{s}+\frac{1}{q}-\frac{1}{p})}\|v\ind_Q\|_{L^{\frac{1}{\frac{1}{q}-\frac{1}{s}}}(\R^d)}\|v^{-1}\ind_Q\|_{L^{\frac{1}{\frac{1}{r}-\frac{1}{p}}}(\R^d)}<\infty,
\]
and all $f\in L^q_v(\R^d)$, $T$ has an extension satisfying
\[
\|Tf\|_{\mathcal{L}^{p,q}_v(\R^d)}\lesssim_d C_T\|f\|_{\mathcal{L}^{p,q}_v(\R^d)}.
\]
\end{theorem}
The reason we only get bounds in $L_v^q(\R^d)\subseteq\mathcal{L}^{p,q}_v(\R^d)$ is because $\mathcal{L}^{p,q}_v(\R^d)$ is not order-continuous, and $L^\infty_c(\R^d)$ is not a dense subspace.
\begin{proof}[Proof of Theorem~\ref{thm:sparseimpliesmorrey}]
Since $v\in A_{q,(r,s)}$ and $L_v^q(\R^d)$ is order-continuous, $T$ can be extended so that
\[
\|(Tf)g\|_{L^1(\R^d)}\leq C_T\|M_{r,s'}(f,g)\|_{L^1(\R^d)}
\]
holds for all $f\in L^q_v(\R^d)$ and $g\in L^\infty_c(\R^d)$. Moreover, letting $X=\mathcal{L}_v^{p,q}(\R^d)$, we have $X'=\mathcal{B}^{p',q'}_{v^{-1}}(\R^d)$ and
\begin{equation}\label{eq:sparseimpliesmorrey1}
\|(Tf)g\|_{L^1(\R^d)}\leq C_T\|M_r\|_X\|M_{s'}\|_{X'}\|f\|_X\|g\|_{X'}.
\end{equation}
for all $f\in L^q_v(\R^d)$ and $g\in L^\infty_c(\R^d)\cap X'$, provided that 
\begin{equation}\label{eq:sparseimpliesmorrey2}
M_r:X\to X,\quad M_{s'}:X'\to X'
\end{equation}
Since the Block space $X'$ is $p'$-concave, it is order-continuous, and hence, $L^\infty_c(\R^d)\cap X'$ is dense in $X'$. Thus, \eqref{eq:sparseimpliesmorrey1} also holds for all $g\in X'$. We conclude that
\[
\|Tf\|_X=\sup_{\|g\|_{X'}=1}\|(Tf)g\|_{L^1(\R^d)}\leq C_T\|M_r\|_X\|M_{s'}\|_{X'}\|f\|_X
\]
for all $f\in L^q_v(\R^d)$, and it remains to check \eqref{eq:sparseimpliesmorrey2}.

By Theorem~\ref{thm:mboundblockspaces}, $M_{s'}$ is bounded on $X'$ for $v^{-p'}\in A_{p'(\frac{1}{s'}-\frac{1}{q'})+1}$, which is equivalent to $v^{-1}\in A_{p',(\frac{1}{\frac{1}{q}-\frac{1}{s}+\frac{1}{p'}},\infty)}$ or
\[
v\in A_{p,(1,\frac{1}{\frac{1}{s}+\frac{1}{p}-\frac{1}{q}})}\supseteq A_{p,(r,\frac{1}{\frac{1}{s}+\frac{1}{p}-\frac{1}{q}})},
\]
the latter being the condition assumed on $v$. 

By Proposition~\ref{prop:mqboundrescale}, $M_r$ is bounded on $X$ if and only if $M$ is bounded on $X^r=X_{r,\infty}$. By Theorem~\ref{thm:mrsboundmorrey} this is the case when
\[
v\in A_{p,(r,\frac{1}{\frac{1}{p}-\frac{1}{q}})}\supseteq A_{p,(r,\frac{1}{\frac{1}{s}+\frac{1}{p}-\frac{1}{q}})},
\]
where the latter condition is the one imposed on $v$. Thus, \eqref{eq:sparseimpliesmorrey2} holds, and the assertion follows.
\end{proof}

\subsection{Multilinear sparse domination and extrapolation}
An $m$-(sub)linear operator $T$ is said to have sparse domination in form if
\begin{equation}\label{eq:Tfullrangemultisd}
\|T(f_1,\ldots,f_m)g\|_{L^1(\R^d)}\leq C_T\|M_{\vec{1},1}(f_1,\ldots,f_m,g)\|_{L^1(\R^d)},
\end{equation}
where,
\[
M_{\vec{1},1}(f_1,\ldots,f_m,g):=\sup_Q \Big(\prod_{j=1}^m\langle f_j\rangle_{1,Q}\Big)\langle g\rangle_{1,Q}\ind_Q.
\]
Many multilinear operators satisfy this bound such as, for example,  multilinear Calder\'on-Zygmund operators \cite{CR16}.

Consider an $m$-tuple of Banach function spaces $X_1,\ldots, X_m$ such that the product space
\[
X:=\prod_{j=1}^m X_j
\]
is again a Banach function space. To be more precise, the space $X$ is defined so that $f\in X$ precisely when there exist $0\leq f_j\in X_j$ such that $|f|\leq\prod_{j=1}^m f_j$ and
\[
\|f\|_X:=\inf\prod_{j=1}^m\|f_j\|_{X_j},
\]
where the infimum runs over all possible $0\leq f_j\in X_j$ such that $|f|\leq\prod_{j=1}^m f_j$.

Now, it is clear that if
\[
M:X_j\to X_j
\]
for all $j\in\{1,\ldots,m\}$ and $M:X'\to X'$ are bounded, then
\[
\|M\|_{X_1\times\cdots X_m\times X'\to L^1(\R^d)}\leq\Big(\prod_{j=1}^m\|M\|_{X_j\to X_j}\Big)\|M\|_{X'\to X'}
\]
and, hence, using the fact that $X$ is a Banach function space, we have the bound
\begin{align*}
\|T(f_1,\ldots,f_m)\|_{X}
&=\sup_{\|g\|_{X'}=1}\|T(f_1,\ldots,f_m)g\|_{L^1(\R^d)}\\
&\leq C_T\Big(\prod_{j=1}^m\|M\|_{X_j\to X_j}\Big)\|M\|_{X'\to X'}\prod_{j=1}^m\|f_j\|_{X_j}.
\end{align*}
However, unlike in the case $m=1$, this does not in general give us the full range of possible bounds. As a matter fact, the condition that $M:X_j\to X_j$ and $M:X'\to X'$ is too strong in the multilinear case. As was originally shown in \cite{LOPTT09}, in the case where $X_J=L^{p_j}_{w_j}(\R^d)$ and $X=L^p_w(\R^d)$ for $p_j\in(1,\infty]$, $\frac{1}{p}=\sum_{j=1}^m\frac{1}{p_j}$, and $w=\prod_{j=1}^m w_j$, the bound
\[
M:X_1\times\cdots\times X_m\times X'\to L^1(\R^d)
\]
is characterized by the condition $\vec{w}\in A_{\vec{p}}$ given by
\[
[\vec{w}]_{\vec{p}}:=\sup_Q\Big(\prod_{j=1}^m\langle w_j^{-1}\rangle_{p_j',Q}\Big)\langle w\rangle_{p,Q}<\infty.
\]
They also showed that the class $A_{p_1}\times\cdots\times A_{p_m}$ is strictly contained in $A_{\vec{p}}$, i.e., there exist weights $\vec{w}\in A_{\vec{p}}$ such that there is a $j\in\{1,\ldots,m\}$ with $w_j^{p_j}\notin A_{p_j}$. Hence, in this case we do have
\[
M:X_1\times\cdots\times X_m\times X'\to L^1(\R^d),
\]
but not $M:X_j\to X_j$. Thus, the condition that $M:X_j\to X_j$ and $M:X'\to X'$ is too strong in the multilinear case. Define
\[
M_{\vec{1}}(f_1,\ldots,f_m):=\sup_Q \prod_{j=1}^m\langle f_j\rangle_{1,Q}\ind_Q.
\]
Then we have the following multilinear analogue of Proposition~\ref{prop:sparseboundsfullrangeequiv}:
\begin{proposition}
Let $X_1,\ldots,X_m$ be an $m$-tuple of Banach function space over $\R^d$ so that their product space $X$ is again a Banach function space. Then the following assertions are equivalent:
\begin{enumerate}[(i)]
\item\label{it:sparseboundfullrangemultiequiv1} $M_{\vec{1}}:X_1\times\cdots X_m\to X$ and
\[
M:X_1\times X_{j-1}\times X_{j+1}\times\cdots\times X_m\times X'\to X_j'
\]
are bounded for all $j\in\{1,\ldots,m\}$;
\item\label{it:sparseboundfullrangemultiequiv2} $M_{\vec{1},1}:X_1\times\cdots\times X_m\times X'\to L^1(\R^d)$ is bounded.
\end{enumerate}
\end{proposition}
The proof of this result is analogous to \cite[Proposition~4.3]{LN22} and is left to the interested reader.

The spaces for which $M$ is bounded is not the only obstruction. For example, even if $X_j=L^{p_j}(\R^d)$ for $p_j\in(1,\infty]$, we have
\[
X=L^p(\R^d),\quad \frac{1}{p}=\sum_{j=1}^m\frac{1}{p_j},
\]
which is not a Banach function space whenever $p\in(\frac{1}{m},1)$. Nonetheless, if $T$ satisfies \eqref{eq:Tfullrangemultisd}, then it does satisfy bounds in this range of $p$ as can be shown using multilinear extrapolation in Lebesgue spaces \cite{LMO20, Ni19, LMMOV21}. The way to solve this in the general setting of Banach function spaces is through a rescaling argument. If we only assume that $X_1,\ldots, X_m$ are Banach function spaces, but assume no extra condition on $X$, then $X$ is a $\frac{1}{m}$-convex quasi-Banach function space. Writing
\[
M_{\vec{m},1}(f_1,\ldots,f_m,g)=\sup_Q \Big(\prod_{j=1}^m\langle f_j\rangle_{m,Q}\Big)\langle g\rangle_{1,Q}\ind_Q,
\]
then we can impose the condition that
\[
M_{\vec{m},1}:X_1^{\frac{1}{m}}\times\cdots\times X_m^{\frac{1}{m}}\times(X^{\frac{1}{m}})'\to L^1(\R^d)
\]
is bounded. Indeed, if $X_j=L^{p_j}_{w_j}(\R^d)$ for $\vec{w}\in A_{\vec{p}}$, then the above bound holds, even for $p\in(\frac{1}{m},1)$.

We conjecture that the following should hold:
\begin{conjecture}
Let $p_1,\ldots,p_m\in(1,\infty]$ with $\frac{1}{p}=\sum_{j=1}^m\frac{1}{p_j}>0$. Suppose $T$ is an operator that is well-defined on all $m$-tuples of functions $f_j\in L^{p_j}_{w_j}(\R^d)$ with $\vec{w}\in A_{\vec{p}}$ for $j\in\{1,\ldots,m\}$. Moreover, assume that there is an increasing function $\phi:[0,\infty)\to[0,\infty)$ such that for all $\vec{w}\in A_{\vec{p}}$ and $f_j\in L^{p_j}_{w_j}(\R^d)$, $j\in\{1,\ldots,m\}$ we have
\[
\|T(f_1,\ldots,f_m)\|_{L^p_w(\R^d)}\leq\phi([\vec{w}]_{\vec{p}})\prod_{j=1}^m\|f_j\|_{L^{p_j}_{w_j}(\Omega)}.
\]
Let $X_1,\ldots, X_m$ be Banach function spaces over $\R^d$ for which
\[
M_{\vec{m},1}:X_1^{\frac{1}{m}}\times\cdots\times X_m^{\frac{1}{m}}\times(X^{\frac{1}{m}})'\to L^1(\R^d)
\]
is bounded. Then $T(f_1,\ldots,f_m)$ is well-defined for any $f_j\in X_j$ and there is an increasing function $\psi:[0,\infty)\to [0,\infty)$ so that
\[
\|T(f_1,\ldots,f_m)\|_X\leq\psi(\|M_{\vec{m},1}\|_{X_1^{\frac{1}{m}}\times\cdots\times X_m^{\frac{1}{m}}\times(X^{\frac{1}{m}})'\to L^1(\R^d)})\prod_{j=1}^m\|f_j\|_{X_j}.
\]
\end{conjecture}

The multilinear Rubio de Francia algorithm that was developed in \cite{Ni19} only seems to work in Lebesgue spaces, and it is not clear how to extend this to general Banach function spaces. However, the proof in \cite{LMO20} relies on the off-diagonal extrapolation theorem in the case $m=1$. Since we have proven off-diagonal extrapolation in the case $m=1$ in this current work, this gives us a possible avenue to prove the conjecture.

In the meantime, we are able to prove an extrapolation theorem in the product weight class $A_{p_1}\times\cdots\times A_{p_m}\subsetneq A_{\vec{p}}$. As a matter of fact, we do this in a limited range setting where the initial assumptions on the operator $T$ is equivalent to the one assumed in the limited range Lebesgue space extrapolation result of \cite{CM17}:

\begin{theorem}\label{thm:multilinearextrapolation}
Let $r_1,\ldots,r_m\in(0,\infty)$, $s_1,\ldots, s_m\in(0,\infty]$ with $r_j<s_j$, and $p_j\in[r_j,s_j]$ for $j\in\{1,\ldots,m\}$ with $\frac{1}{p}:=\sum_{j=1}^m\frac{1}{p_j}$. Let $(\Omega,|\cdot|)$ be a $\sigma$-finite measure space with a basis of sets $\mathcal{E}$. Let $V_j$ be sets and let $S_j:V_j\to L^0(\Omega)$ be maps. Moreover, suppose
\[
T:\bigcup_{w_1\in A_{p_1,(r_1,s_1)}(\mathcal{E})}S_1^{-1}(L^{p_1}_{w_1}(\Omega))\times\cdots\times \bigcup_{w_m\in A_{p_m,(r_m,s_m)}(\mathcal{E})}S_m^{-1}(L^{p_m}_{w_m}(\Omega))\to L^0(\Omega)
\]
is a map for which there is an increasing function $\phi:[0,\infty)^m\to[0,\infty)$ such that for all $w_j\in A_{p_j,(r_j,s_j)}(\mathcal{E})$ and $f_j\in V_j$ with $Sf_j\in L^{p_j}_{w_j}(\Omega)$, $j\in\{1,\ldots,m\}$, we have
\begin{equation}\label{eq:extrapolationmultiin}
\|T(f_1,\ldots,f_m)\|_{L^p_w(\Omega)}\leq\phi([w_1]^\mathcal{E}_{p_1,(r_1,s_1)},\ldots,[w_m]^\mathcal{E}_{p_m,(r_m,s_m)})\prod_{j=1}^m\|S_jf_j\|_{L^{p_j}_{w_j}(\Omega)}.
\end{equation}
Let $X_1,\ldots, X_m$ be quasi-Banach function spaces over $\Omega$ for which $X_j$ is $r_j$-convex and $s_j$-concave, and for all $j\in\{1,\ldots,m\}$ we have:
\begin{itemize}
\item  If $p_j\neq r_j$,
\[
M^\mathcal{E}:(X_j)_{r_j,s_j}\to (X_j)_{r_j,s_j}
\]
is bounded;
\item If $p_j\neq s_j$, then
\[
M^\mathcal{E}:\big[(X_j)_{r_j,s_j}\big]'\to\big[(X_j)_{r_j,s_j}\big]'
\]
is bounded.
\end{itemize}
Then $T(f_1,\ldots,f_m)$ is well-defined for all $f_j\in V_j$ with $Sf_j\in X_j$, and
\begin{equation}\label{eq:extrapolationmultiout}
\|T(f_1,\ldots,f_m)\|_X\leq2^{\frac{1}{r}-\frac{1}{s}}\phi(C_1,\ldots,C_m)\prod_{j=1}^m\|f_j\|_{X_j},
\end{equation}
where $\frac{1}{r}=\sum_{j=1}^m\frac{1}{r_j}$, $\frac{1}{s}=\sum_{j=1}^m\frac{1}{s_j}$, and
\[
C_j=2^{\frac{1}{r_j}-\frac{1}{s_j}}\|M^\mathcal{E}\|_{(X_j)_{r_j,s_j}\to (X_j)_{r_j,s_j}}^{\frac{1}{r_j}-\frac{1}{p_j}}\|M^\mathcal{E}\|_{\big[(X_j)_{r_j,s_j}\big]'\to\big[(X_j)_{r_j,s_j}\big]'}^{\frac{1}{p_j}-\frac{1}{s_j}}.
\]
\end{theorem}
\begin{proof}
Since $X_j$ is $r_j$-convex for all $j\in\{1,\ldots,m\}$, the product space $X=\prod_{j=1}^m X_j$ is $r$-convex, where $\frac{1}{r}=\sum_{j=1}^m\frac{1}{r_j}$. Moreover, by Lozanovskii's duality theorem, we have
\[
\big[(X^r)'\big]^{\frac{1}{r}}=\Big[\Big(\prod_{j=1}^m X_j^{r_j\frac{r}{r_j}}\Big)'\Big]^{\frac{1}{r}}=\Big[\prod_{j=1}^m\big[(\ X_j^{r_j})'\big]^{\frac{r}{r_j}}\Big]^{\frac{1}{r}}=\prod_{j=1}^m \big[(X_j^{r_j})'\big]^{\frac{1}{r_j}}.
\]
Fix $g\in\big[(X^r)'\big]^{\frac{1}{r}}$. Then there exist $0\leq g_j\in\big[(X_j^{r_j})'\big]^{\frac{1}{r_j}}$ so that $|g|\leq\prod_{j=1}^m g_j$.

Now, fix $j\in\{1,\ldots,m\}$ and $f_j\in V_j$ with $Sf_j\in X_j$. Then by Corollary~\ref{cor:abstractaprsextrapolation} there exists a weight $w_j\in A_{p_j,(r_j,s_j)}(\mathcal{E})$ with
\[
[w_j]^\mathcal{E}_{p_j,(r_j,s_j)}\leq 2^{\frac{1}{r_j}-\frac{1}{s_j}}\|M^\mathcal{E}\|_{(X_j)_{r_j,s_j}\to (X_j)_{r_j,s_j}}^{\frac{1}{r_j}-\frac{1}{p_j}}\|M^\mathcal{E}\|_{\big[(X_j)_{r_j,s_j}\big]'\to\big[(X_j)_{r_j,s_j}\big]'}^{\frac{1}{p_j}-\frac{1}{s_j}}=:C_j
\]
and $S_jf_j\in L^{p_j}_{w_j}(\Omega)$, $g_j\in L^{\frac{1}{\frac{1}{r_j}-\frac{1}{p_j}}}_{w_j^{-1}}(\Omega)$ with
\[
\|S_jf_j\|_{L^{p_j}_{w_j}(\Omega)}\|g_j\|_{L^{\frac{1}{\frac{1}{r_j}-\frac{1}{p_j}}}_{w_j^{-1}}(\Omega)}\leq 2^{\frac{1}{r_j}-\frac{1}{s_j}}\|S_jf_j\|_{X_j}\|g_j\|_{[(X_j^{r_j})']^{\frac{1}{r_j}}}.
\]
Thus, by \eqref{eq:extrapolationmultiin}, we have
\begin{align*}
\|&T(f_1,\ldots,f_m)g\|_{L^r(\Omega)}\\
&\leq\phi([w_1]_{p_1,(r_1,s_1)},\ldots,[w_m]_{p_m,(r_m,s_m)})\Big(\prod_{j=1}^m\|S_jf_j\|_{L^{p_j}_{w_j}(\Omega)}\Big)\|g\|_{L^{\frac{1}{\frac{1}{r}-\frac{1}{p}}}_{w^{-1}}(\Omega)}\\
&\leq\phi(C_1,\ldots,C_m)\prod_{j=1}^m\|S_jf_j\|_{L^{p_j}_{w_j}(\Omega)}\|g_j\|_{L^{\frac{1}{\frac{1}{r_j}-\frac{1}{p_j}}}_{w_j^{-1}}(\Omega)}\\
&\leq 2^{\frac{1}{r}-\frac{1}{s}}\phi(C_1,\ldots,C_m)\prod_{j=1}^m\|S_jf_j\|_{X_j}\|g_j\|_{\big[(X_j^{r_j})'\big]^{\frac{1}{r_j}}}.
\end{align*}
Taking an infimum over all possible representations $|g|\leq\prod_{j=1}^m g_j$, $0\leq g_j\in X_j$, we conclude that
\[
\|T(f_1,\ldots,f_m)g\|_{L^r(\R^d)}\leq 2^{\frac{1}{r}-\frac{1}{s}}\phi(C_1,\ldots,C_m)\prod_{j=1}^m\Big(\prod_{j=1}^m\|S_jf_j\|_{X_j}\Big)\|g\|_{\big[(X^r)'\big]^{\frac{1}{r}}}.
\]
The result now follows from the fact that
\[
\|T(f_1,\ldots,f_m)\|_X=\sup_{\|g\|_{\big[(X^r)'\big]^{\frac{1}{r}}}=1}\|T(f_1,\ldots,f_m)g\|_{L^r(\Omega)}.
\]
\end{proof}
Theorem~\ref{thm:B} follows by taking $V=L^0(\R^d)$ and $Sf:=f$.

We note that in the full range case $r_j=1$, $s_j=\infty$ for all $j\in\{1,\ldots,m\}$, the main ingredient of the proof was Lozanovskii's duality theorem, which allows us to write
\[
\big[(X^{\frac{1}{m}})'\big]^m=\prod_{j=1}^m X_j'.
\]
After this, the result follows in a straightforward way from the case $m=1$. As a matter of fact, one can prove an off-diagonal version of this result with essentially the same proof, this time using our off-diagonal extrapolation argument after reducing to the case $m=1$. We leave this case as an exercise for the reader.

\subsection{Multilinear vector-valued extrapolation}

Using the multilinear vector-valued extrapolation theorems from \cite{LN19, Nidiss}, we can also obtain a multilinear version of Theorem~\ref{thm:vectorvaluedmain} as an application of Theorem~\ref{thm:multilinearextrapolation}.

\begin{theorem}\label{thm:multivvextrapolation}
Let $r_1,\ldots,r_m\in(0,\infty)$, $s_1,\ldots,s_m\in(0,\infty]$ with $r_j<s_j$. Suppose $T$ is an $m$-(sub)linear operator such that for all $p_j\in(r_j,s_j)$ it is well-defined on all $m$-tuples of functions $(f_1,\ldots,f_m)$ with $f_j\in L^{p_j}_{w_j}(\R^d)$ with $w_j\in A_{p_j,(r_j,s_j)}$, $j\in\{1,\ldots,m\}$ and, moreover, there is an increasing function $\phi_{\vec{p}}:[0,\infty)^m\to[0,\infty)$ such that for all $w_j\in A_{p_j,(r_j,s_j)}$ and $f_j\in L^{p_j}_{w_j}(\R^d)$, $j\in\{1,\ldots,m\}$ we have
\[
\|T(f_1,\ldots,f_m)\|_{L^p_w(\R^d)}\leq\phi_{\vec{p}}([w_1]_{p_1,(r_1,s_1)},\ldots,[w_m]_{p_m,(r_m,s_m)})\prod_{j=1}^m\|f_j\|_{L^{p_j}_{w_j}(\R^d)}.
\]
Let $Y_1,\ldots,Y_m$ be an $m$-tuple of quasi-Banach function spaces over a $\sigma$-finite measure space $(\Omega,|\cdot|)$ such that $Y_j$ is order-continuous, $r_j$-convex and $s_j$-concave for all $j\in\{1,\ldots,m\}$ such that for all simple functions $f_j\in L^\infty_c(\R^d;Y_j)$ the function
\[
\widetilde{T}(f_1,\ldots,f_m)(x,y):=T(f_1(\cdot,y),\cdots, f_m(\cdot,y))(x)
\]
is strongly measurable in $Y=\prod_{j=1}^m Y$. If $Y_j^{r_j}$ and $((Y_j)_{r_j,s_j})'$ have the Hardy-Littlewood property for all $j\in\{1,\ldots,m\}$, then there is an increasing function $\phi_{\vec{Y},\vec{r},\vec{s}}:[0,\infty)^{2m}\to [0,\infty)$ such that for all $m$-tuples $X_1,\ldots, X_m$ of quasi-Banach function spaces over $\R^d$ for which $X_j$ is $r_j$-convex, $s_j$-concave, and for all $j\in\{1,\ldots,m\}$ we have that
\[
M:(X_j)_{r_j,s_j}\to (X_j)_{r_j,s_j},\quad M:\big[(X_j)_{r_j,s_j}\big]'\to\big[(X_j)_{r_j,s_j}\big]'
\]
are bounded, we have that
\[
\widetilde{T}:X_1(Y_1)\times\ldots\times X_m(Y_m)\to X(Y)
\]
is bounded, where $X=\prod_{j=1}^m X_j$, with
\begin{align*}
\|\widetilde{T}&\|_{X_1(Y_1)\times\ldots\times X_m(Y_m)\to X(Y)}\\
&\leq\phi_{\vec{Y},
\vec{r},\vec{s}}\big((\|M\|_{(X_j)_{r_j,s_j}\to (X_j)_{r_j,s_j}},\|M\|_{\big[(X_j)_{r_j,s_j}\big]'\to \big[(X_j)_{r_j,s_j}\big]'})_{j=1}^m\big).
\end{align*}
\end{theorem}
Note that this is Theorem~\ref{thm:C}.
\begin{proof}
Choosing $p_j=1+\frac{r_j}{s_j'}$, it follows from \cite[Theorem~9.1.1]{Nidiss} there is an increasing function $\phi_{\vec{Y},\vec{r},\vec{s}}:[0,\infty)^m\to[0,\infty)$ such that
\[
\widetilde{T}:L^{p_1}_{w_1}(\R^d;Y_1)\times\cdots\times L^{p_m}_{w_m}(\R^d;Y_m)\to L^p_w(\R^d;Y)\]
is bounded for all $w_j\in A_{p_j,(r_j,s_j)}$, with
\[
\|\widetilde{T}f\|_{L^p_w(\R^d;Y)}\leq\phi_{\vec{Y},p,r,s}([w_1]_{p_1,(r_1,s_1)},\ldots,[w_m]_{p_m,(r_m,s_m)})\prod_{j=1}^m\|f_j\|_{L^{p_j}_{w_j}(\R^d;Y_j)}.
\]
Then it follows from Theorem~\ref{thm:multilinearextrapolation} with $V_j=\bigcup_{w_j\in A_{p_j,(r_j,s_j)}}L_{w_j}^{p_j}(\Omega;Y_j)$ applied with $Tf$ replaced by $\mathcal{T}f(x):=\|\widetilde{T}f(x,\cdot)\|_Y$ and $S_jf_j(x):=\|f_j(x,\cdot)\|_{Y_j}$. That for all $f_j\in X_j(Y_j)$ we have
\begin{align*}
\|\widetilde{T}f\|_{X(Y)}&=\|\mathcal{T}f\|_X\leq 2^{\frac{1}{r}-\frac{1}{s}}\phi_{\vec{Y},\vec{r},\vec{s}}(C_1,\ldots,C_m)\prod_{j=1}^m\|S_jf_j\|_{X_j}\\
&=2^{\frac{1}{r}-\frac{1}{s}}\phi_{\vec{Y},\vec{r},\vec{s}}(C_1,\ldots,C_m)\prod_{j=1}^m\|f_j\|_{X_j(Y_j)}.
\end{align*}
The assertion follows.
\end{proof}

\subsection{An application to the Bilinear Hilbert transform}
The Bilinear Hilbert transform $\BHT$ is defined for $f_1,f_2\in\mathcal{S}(\R)$ by
\[
\BHT(f_1,f_2)(x):=\pv\int_R\!f_1(x-y)f_2(x+y)\,\frac{\mathrm{d}y}{y},
\]
and falls outside of the scope of bilinear Calder\'on-Zygmund theory. It was shown through a sparse domination method in \cite{CDO18} that this operator is bounded
\[
\BHT:L^{p_1}_{w_1}(\R)\times L^{p_2}_{w_2}(\R)\to L^p_w(\R),
\]
where $w=w_1w_2$, $\frac{1}{p}=\frac{1}{p_1}+\frac{1}{p_2}$, whenever there exist $r_1,r_2,r_3\in(1,\infty)$ satisfying
\[
\sum_{j=1}^3\max\Big(\frac{1}{r_j},\frac{1}{2}\Big)<2,
\]
with $r_j<p_j$ and
\[
[w_1,w_2]_{(p_1,p_2),(r_1,r_2,r_3')}:=\sup_Q\langle w_1^{-1}\rangle_{\frac{1}{\frac{1}{r_1}-\frac{1}{p_1}},Q}\langle w_2^{-1}\rangle_{\frac{1}{\frac{1}{r_2}-\frac{1}{p_2}},Q}\langle w\rangle_{\frac{1}{\frac{1}{p}-\frac{1}{r_3'}},Q}<\infty.
\]
While, as we discussed in the previous section, unfortunately our extrapolation theorem does not cover these more general weight classes, we can still get new bounds by noting that if $s_1,s_2\in(0,\infty]$ with $r_1<s_1$, $r_2<s_2$ satisfies $\frac{1}{s_1}+\frac{1}{s_2}=\frac{1}{r_3'}$, then, by H\"older's inequality, we have
\[
[w_1,w_2]_{(p_1,p_2),(r_1,r_2,r_3')}\leq[w_1]_{p_1,(r_1,s_1)}[w_2]_{p_2,(r_2,s_2)}.
\]
Hence, by Theorem~\ref{thm:multilinearextrapolation}, we obtain the following:
\begin{theorem}
Let $r_1,r_2\in(1,\infty)$, $s_1,s_2\in(1,\infty]$ with $r_j<s_j$ and, for $\frac{1}{s}=\frac{1}{s_1}+\frac{1}{s_2}$,
\[
\max\Big(\frac{1}{r_1},\frac{1}{2}\Big)+\max\Big(\frac{1}{r_2},\frac{1}{2}\Big)+\max\Big(\frac{1}{s'},\frac{1}{2}\Big)<2.
\]
Then
\[
\BHT:X_1\times X_2\to X
\]
is bounded for all $r_j$- and $s_j$-concave quasi-Banach function spaces $X_j$ over $\R$ for which
\[
M:(X_j)_{r_j,s_j}\to (X_j)_{r_j,s_j},\quad M:\big[(X_j)_{r_j,s_j}\big]'\to\big[(X_j)_{r_j,s_j}\big]'
\]
are bounded for $j\in\{1,2\}$.

In particular, we have:
\begin{itemize}
\item Lorentz spaces: For all $p_1,q_1\in(r_1,s_1)$, $p_2,q_2\in(r_2,s_2)$, and weights $(v_1,v_2)\in A_{p_1,(r_1,s_1)}\times A_{p_2,(r_2,s_2)}$, we have
\[
\BHT:L^{p_1,q_1}_{v_1}(\R^d)\times L^{p_2,q_2}_{v_2}(\R)\to L^{p,q}_v(\R),
\]
where $v=v_1v_2$, $\frac{1}{p}=\frac{1}{p_1}+\frac{1}{p_2}$, $\frac{1}{q}=\frac{1}{q_1}+\frac{1}{q_2}$.
\item Variable Lebesgue spaces: For all $p_1,p_2:\R\to(1,\infty)$ satisfying
\[
r_j<\essinf p_j\leq\esssup p_j< s_j
\]
and $p_j(\cdot)\in LH_0\cap LH_\infty$, $j\in\{1,2\}$,  and all weights $v_1,v_2$ satisfying
\[
\sup_Q|Q|^{-\big(\frac{1}{r_j}-\frac{1}{s_j}\big)}\|v^{-1}\ind_Q\|_{L^{\frac{1}{\frac{1}{r_j}-\frac{1}{p_j(\cdot)}}}(\R)}\|v\ind_Q\|_{L^{\frac{1}{\frac{1}{p_j(\cdot)}-\frac{1}{s_j}}}(\R)}<\infty,
\]
we have
\[
\BHT:L^{p_1(\cdot)}_{v_1}(\R)\times L^{p_2(\cdot)}_{v_2}(\R)\to L^{p(\cdot)}_v(\R),
\]
where $v=v_1v_2$, $\frac{1}{p(x)}=\frac{1}{p_1(x)}+\frac{1}{p_2(x)}$.
\end{itemize}

Moreover, if $Y_1$, $Y_2$ are respectively $r_1$- and $r_2$-convex and $s_1$- and $s_2$-concave order-continuous quasi-Banach function spaces such that $Y_j^{r_j}$ and $((Y_j)_{r_j,s_j})'$ have the Hardy-Littlewood property, then we also have that
\[
\widetilde{\BHT}:X_1(Y_1)\times X_2(Y_2)\to X(Y)
\]
is bounded for any of the above spaces $X_1$, $X_2$.
\end{theorem}

It was shown in \cite{CDO18} that we have the sparse domination
\[
\|\BHT(f_1,f_2)g\|_{L^1(\R)}\lesssim_{r_1,r_2,s} \|M_{r_1,r_2,s'}(f_1,f_2,g)\|_{L^1(\R)}
\]
for all $f_1,f_2,g\in L^\infty_c(\R)$ whenever $r_1,r_2,s\in(1,\infty)$ satisfy
\[
\max\Big(\frac{1}{r_1},\frac{1}{2}\Big)+\max\Big(\frac{1}{r_2},\frac{1}{2}\Big)+\max\Big(\frac{1}{s'},\frac{1}{2}\Big)<2.
\]
It was later also shown in \cite{BM17} that for these same parameters we also have the ``worse'' (in the sense of \cite[Conjecture~6.2]{LN22}) sparse domination
\[
\|\BHT(f_1,f_2)g\|_{L^r(\R)}\lesssim_{r_1,r_2,s} \|M_{r_1,r_2,\frac{1}{\frac{1}{r}-\frac{1}{s}}}(f_1,f_2,g)\|_{L^r(\R)}
\]
for all $f_1,f_2,g\in L^\infty_c(\R)$, where $\frac{1}{r}=\frac{1}{r_1}+\frac{1}{r_2}$. Using this result, we can prove bounds for $\BHT$ in weighted Morrey spaces:
\begin{theorem}\label{thm:bhtmorrey}
Let $r_1,r_2\in(1,\infty)$, $s_1,s_2\in(1,\infty]$ with $r_j<s_j$ and, for $\frac{1}{s}=\frac{1}{s_1}+\frac{1}{s_2}$,
\[
\max\Big(\frac{1}{r_1},\frac{1}{2}\Big)+\max\Big(\frac{1}{r_2},\frac{1}{2}\Big)+\max\Big(\frac{1}{s'},\frac{1}{2}\Big)<2.
\]
Then for all $p_1\in(r_1,s_1)$, $p_2\in(r_2,s_2)$,  $q\in[p,\infty)$, $\frac{1}{q_j}:=\frac{1}{p_j}\frac{p}{q}$, and all weights $v_j$ with
\[
\sup_Q|Q|^{-\big(\frac{1}{r_j}-\frac{1}{s_j}-\frac{1}{p_j}+\frac{1}{q_j}\big)}\|v\ind_Q\|_{L^{\frac{1}{\frac{1}{q_j}-\frac{1}{s_j}}}(\R)}\|v^{-1}\ind_Q\|_{L^{\frac{1}{\frac{1}{r_j}-\frac{1}{p_j}}}(\R)}<\infty,
\]
for $j\in\{1,2\}$, we have
\[
\BHT:\mathcal{L}_{v_1}^{p_1,q_1}(\R)\times \mathcal{L}_{v_2}^{p_2,q_2}(\R)\to\mathcal{L}_v^{p,q}(\R),
\]
where $v=v_1v_2$, $\frac{1}{p}=\frac{1}{p_1}+\frac{1}{p_2}$.
\end{theorem}
For this we need the following result:
\begin{lemma}\label{lem:factormorreybht}
Let $p_1,p_2\in(0,\infty]$ with $\frac{1}{p}=\frac{1}{p_1}+\frac{1}{p_2}>0$, $q\in[p,\infty)$, $\frac{1}{q_j}:=\frac{1}{p_j}\frac{p}{q}$, and $v_j$ a weight for $j\in\{1,2\}$. Then 
\[
\mathcal{L}_{v_1}^{p_1,q_1}(\R^d)\cdot \mathcal{L}_{v_2}^{p_2,q_2}(\R^d)=\mathcal{L}^{p,q}_v(\R^d),
\]
where $v=v_1v_2$.
\end{lemma}
\begin{proof}
Since
\begin{align*}
\mathcal{L}_{v_1}^{p_1,q_1}(\R^d)\cdot \mathcal{L}_{v_2}^{p_2,q_2}(\R^d)
&=\mathcal{L}^{p_1,q_1}(\R^d)(v_1)\cdot \mathcal{L}^{p_2,q_2}(\R^d)(v_2)\\
&=\big(\mathcal{L}^{p_1,q_1}(\R^d)\cdot \mathcal{L}^{p_2,q_2}(\R^d)\big)(v_1v_2),
\end{align*}
it suffices to prove the case $v_1=v_2=1$.

Let $h\in\mathcal{L}^{p_1,q_1}(\R^d)\cdot \mathcal{L}^{p_2,q_2}(\R^d)$. Then there are $0\leq f_j\in \mathcal{L}^{p_j,q_j}(\R^d)$ with $|h|\leq f_1f_2$. Hence, for every cube $Q$ in $\R^d$, it follows from H\"older's inequality that 
\[
|Q|^{\frac{1}{q}}\langle h\rangle_{p,Q}\leq|Q|^{\frac{1}{q}}\langle f_1\rangle_{p_1,Q}\langle f_2\rangle_{p_2,Q}\leq\|f_1\|_{\mathcal{L}^{p_1,q_1}(\R^d)}\|f_2\|_{\mathcal{L}^{p_2,q_2}(\R^d)}.
\]
Hence, by taking a supremum over all cubes $Q$ we conclude that $h\in\mathcal{L}^{p,q}(\R^d)$, and by taking an infimum over all decompositions $|h|\leq f_1f_2$ we conclude that
\[
\|h\|_{\mathcal{L}^{p,q}(\R^d)}\leq\|h\|_{\mathcal{L}^{p_1,q_1}(\R^d)\cdot \mathcal{L}^{p_2,q_2}(\R^d)}.
\]
For the converse, suppose $h\in\mathcal{L}^{p,q}(\R^d)$. Write $f_j:=|h|^{\frac{p}{p_j}}$. Then, $|h|=f_1f_2$ and, using $\frac{1}{q_j}=\frac{1}{p_j}\frac{p}{q}$, for a cube $Q$ we have
\[
|Q|^{\frac{1}{q_j}}\langle f_j\rangle_{p_j,Q}=\big(|Q|^{\frac{1}{q}}\langle h\rangle_{p,Q}\big)^{\frac{p}{p_j}}
\]
for $j\in\{1,2\}$. Taking a supremum over all cubes $Q$, we conclude that $f_j\in\mathcal{L}^{p_j,q_j}(\R^d)$ with
\[
\|f_j\|_{\mathcal{L}^{p_j,q_j}(\R^d)}=\|h\|_{\mathcal{L}^{p,q}(\R^d)}^{\frac{p}{p_j}}
\]
so that
\[
\|h\|_{\mathcal{L}^{p_1,q_1}(\R^d)\cdot\mathcal{L}^{p_2,q_2}(\R^d)}\leq\|f_1\|_{\mathcal{L}^{p_1,q_1}(\R^d)}\|f_2\|_{\mathcal{L}^{p_2,q_2}(\R^d)}=\|h\|_{\mathcal{L}^{p,q}(\R^d)}.
\]
This proves the assertion.
\end{proof}
\begin{proof}[Proof of Theorem~\ref{thm:bhtmorrey}]
Since $v_j\in A_{q_j,(r_j,s_j)}$, we can use order-continuity of the spaces $L^{q_j}_{v_j}(\R)$ to conclude that 
\[
\|\BHT(f_1,f_2)g\|_{L^r(\R)}\lesssim_{r_1,r_2,s} \|M_{r_1,r_2,\frac{1}{\frac{1}{r}-\frac{1}{s}}}(f_1,f_2,g)\|_{L^r(\R)}
\]
for all $f_j\in L^{q_j}_{v_j}(\R)$ and $g\in L^\infty_c(\R)$. Letting $X_j:=\mathcal{L}^{p_j,q_j}_{v_j}(\R)$ it follows from Lemma~\ref{lem:factormorreybht} that$X:=X_1\cdot X_2=\mathcal{L}^{p,q}_v(\R)$. Then we have that $X$ is $r$-convex and
\begin{equation}\label{eq:bhtmorrey1}
\begin{split}
\|\BHT&(f_1,f_2)g\|_{L^r(\R)}\lesssim_{r_1,r_2,s}\|M_{r_1,r_2}(f_1,f_2)\|_{X}\|M_{s'}g\|_{[(X^r)']^{\frac{1}{r}}}\\
&\leq\|M_{r_1}f_1\|_{X_1}\|M_{r_2}f_2\|_{X_2}\|M_{s'}g\|_{[(X^r)']^{\frac{1}{r}}}\\
&\leq\|M_{r_1}\|_{X_1\to X_1}\|M_{r_2}\|_{X_2\to X_2}\|M_{\frac{1}{\frac{1}{r}-\frac{1}{s}}}\|_{[(X^r)']^{\frac{1}{r}}\to [(X^r)']^{\frac{1}{r}}}\|f_1\|_{X_1}\|f_2\|_{X_2}\|g\|_{[(X^r)']^{\frac{1}{r}}}
\end{split}
\end{equation}
for all $f_j\in L^{q_j}_{v_j}(\R)$, $g\in L^\infty_c(\R)\cap [(X^r)']^{\frac{1}{r}}$, provided that
\begin{equation}\label{eq:bhtmorrey2}
M_{r_j}:X_j\to X_j,\quad M_{\frac{1}{\frac{1}{r}-\frac{1}{s}}}:[(X^r)']^{\frac{1}{r}}\to [(X^r)']^{\frac{1}{r}}
\end{equation}
are bounded for $j\in\{1,2\}$, where
\[
[(X^r)']^{\frac{1}{r}}=(\mathcal{B}^{(\frac{p}{r})',(\frac{q}{r})'}_{v^{-r}}(\R))^{\frac{1}{r}}.
\]
Since this space is order-continuous, this means that \eqref{eq:bhtmorrey1} also holds for all $g\in[(X^r)']^{\frac{1}{r}}$. Hence, we have
\begin{align*}
\|\BHT&(f_1,f_2)\|_X=\sup_{\|g\|_{[(X^r)']^{\frac{1}{r}}}=1}\|\BHT(f_1,f_2)g\|_{L^r(\R)}\\
&\lesssim_{r_1,r_2,s}\|M_{r_1}\|_{X_1\to X_1}\|M_{r_2}\|_{X_2\to X_2}\|M_{\frac{1}{\frac{1}{r}-\frac{1}{s}}}\|_{[(X^r)']^{\frac{1}{r}}\to [(X^r)']^{\frac{1}{r}}}\|f_1\|_{X_1}\|f_2\|_{X_2}.
\end{align*}
Thus, it remains to check \eqref{eq:bhtmorrey2}.

Fix $j\in\{1,2\}$. By Proposition~\ref{prop:mqboundrescale} we have that $M_{r_j}$ is bounded on $X_j$ if and only if $M$ is bounded on $X_j^{r_j}=(X_j)_{r_j,\infty}$. By Theorem~\ref{thm:mrsboundmorrey} this is the case when
\[
v_j\in A_{p_j,(r_j,\frac{1}{\frac{1}{p_j}-\frac{1}{q_j}})}\supseteq A_{p_j,(r_j,\frac{1}{\frac{1}{s_j}+\frac{1}{p_j}-\frac{1}{q_j}})}.
\]
Since $v_j$ is in the latter class, we conclude that $M_{r_j}$ is bounded on $X_j$.

Finally, it follows from Proposition~\ref{prop:mqboundrescale} that $M_{\frac{1}{\frac{1}{r}-\frac{1}{s}}}$ is bounded on $[(X^r)']^{\frac{1}{r}}$ if and only if $M_{(\frac{s}{r})'}$ is bounded on $(X^r)'=\mathcal{B}^{(\frac{p}{r})',(\frac{q}{r})'}_{v^{-r}}(\R)$. By Theorem~\ref{thm:mboundblockspaces} this is the case when $v^{-\frac{1}{\frac{1}{r}-\frac{1}{p}}}\in A_{(\frac{p}{r})'(\frac{r}{q}-\frac{r}{s})+1}$, which is equivalent to $v^{-1}\in A_{\frac{1}{\frac{1}{r}-\frac{1}{p}},(\frac{1}{\frac{1}{r}-\frac{1}{p}+\frac{1}{q}-\frac{1}{s}},\infty)}$, or
\begin{equation}\label{eq:bhtmorrey3}
v\in A_{p,(r,\frac{1}{\frac{1}{s}+\frac{1}{p}-\frac{1}{q}})}.
\end{equation}
Note that by H\"older's inequality we have
\[
\langle v\rangle_{\frac{1}{\frac{1}{q}-\frac{1}{s}},Q}\langle v^{-1}\rangle_{\frac{1}{\frac{1}{r}-\frac{1}{p}},Q}\leq [v_1]_{p_1,(r_1,\frac{1}{\frac{1}{s_1}+\frac{1}{p_1}-\frac{1}{q_1}})}[v_2]_{p_2,(r_2,\frac{1}{\frac{1}{s_2}+\frac{1}{p_2}-\frac{1}{q_2}})}
\]
for all cubes $Q$, and hence, that \eqref{eq:bhtmorrey3} holds by the assumed conditions on the $v_j$. The assertion follows.
\end{proof}

\subsection*{Acknowledgements}
The author wants to express her gratitude to Emiel Lorist for many fruitful discussions on the various necessary and unnecessary conditions one can impose on quasi-Banach function spaces resulting in a more streamlined theory, for providing the idea needed to improve the bounds obtained for weighted Block spaces, and for pointing out an error in the original manuscript related to Morrey spaces. The author also wishes to express her gratitude to Luz Roncal for providing numerous minor corrections to the overall presentation of this work. Finally, the author expresses her gratitude to the anonymous referee who provided several suggestions for improving the presentation of this work.

\bibliography{bieb}
\bibliographystyle{alpha}
\end{document}